\documentclass[10pt]{amsart}
\usepackage{amssymb}
\usepackage{dsfont}
\usepackage{amscd}
\usepackage[mathscr]{euscript} 
\usepackage[all]{xy}
\usepackage{url}
\usepackage{stmaryrd}
\usepackage[retainorgcmds]{IEEEtrantools}
\usepackage{comment}
\usepackage{hyperref}
\usepackage{tensor}
\title{Ax-Schanuel condition in arbitrary characteristic}
\author[P. KOWALSKI]{Piotr Kowalski$^*$}
\thanks{2010 \textit{Mathematics Subject Classification}. Primary 11J91; Secondary 12H05.}
\thanks{\textit{Key words and phrases}. Ax-Schanuel conjecture, formal maps, transcendence.}
\thanks{$^*$ Supported by Narodowe Centrum Nauki grants no. NCN 2012/07/B/ST1/03513, 2015/19/B/ST1/01150 and 2015/19/B/ST1/01151; and by ANR Modig (ANR-09-BLAN-0047)}
\address{Instytut Matematyczny\\
Uniwersytet Wroc{\l}awski\\pl. Grunwaldzki 2/4\\50-384
Wroc{\l}aw\\
Poland}
 \email{pkowa@math.uni.wroc.pl} \urladdr{http://www.math.uni.wroc.pl/\textasciitilde pkowa/ }  \DeclareMathOperator{\closed}{closed}\DeclareMathOperator{\locus}{locus}
  
  \DeclareMathOperator{\id}{id}
 \DeclareMathOperator{\fr}{Fr} \DeclareMathOperator{\lie}{Lie}

\DeclareMathOperator{\ch}{char}  
\DeclareMathOperator{\td}{trdeg} 
 \DeclareMathOperator{\alg}{alg}
 
\DeclareMathOperator{\arc}{Arc} 

\DeclareMathOperator{\trd}{trdeg}
\DeclareMathOperator{\hs}{HS}\DeclareMathOperator{\inv}{inv}
\DeclareMathOperator{\li}{\underleftarrow{\lim}}\DeclareMathOperator{\et}{et}

\DeclareMathOperator{\spec}{Spec}

\DeclareMathOperator{\pic}{Pic}
\DeclareMathOperator{\andeg}{andeg}

\newtheorem{theorem}{Theorem}[section]
\newtheorem{prop}[theorem]{Proposition}
\newtheorem{lemma}[theorem]{Lemma}

\newtheorem{fact}[theorem]{Fact}
\theoremstyle{definition}
\newtheorem{definition}[theorem]{Definition}
\newtheorem{example}[theorem]{Example}
\newtheorem{remark}[theorem]{Remark}
\newtheorem{question}[theorem]{Question}

\newtheorem{conj}[theorem]{Conjecture}

\begin{document}
\maketitle
\newcommand{\Zz}{{\mathds{Z}}}
\newcommand{\Ff}{{\mathds{F}}}
\newcommand{\Cc}{{\mathds{C}}}
\newcommand{\Rr}{{\mathds{R}}}
\newcommand{\Nn}{{\mathds{N}}}
\newcommand{\Qq}{{\mathds{Q}}}
\newcommand{\Kk}{{\mathds{K}}}
\newcommand{\Pp}{{\mathds{P}}}
\newcommand{\ddd}{\mathrm{d}}
\newcommand{\Aa}{\mathds{A}}
\newcommand{\dlog}{\mathrm{ld}}
\newcommand{\ga}{\mathbb{G}_{\rm{a}}}
\newcommand{\gm}{\mathbb{G}_{\rm{m}}}


\tableofcontents

\begin{abstract}
We prove a positive characteristic version of Ax's theorem on the intersection of an algebraic subvariety and an analytic subgroup of an algebraic group \cite{ax72}. Our result is stated in a more general context of a formal map between an algebraic variety and an algebraic group. We derive transcendence results of Ax-Schanuel type.
\end{abstract}

\newpage

\section{Introduction}\label{intro}

This paper is an attempt to generalize the results of Ax \cite{ax72} on the dimension of the intersection of an algebraic subvariety and a formal subgroup of an algebraic group (as well as its Ax-Schanuel applications) to the case of arbitrary characteristic. Instead of subvarieties of an algebraic group, we prefer to consider the case of a formal map between an algebraic variety and a formal group.

For reader's convenience, we discuss first the main theorem from \cite{ax72}. Let $G$ be an algebraic group over the field of complex numbers $\Cc$. Let $\mathcal{A}$ be a complex analytic subgroup of $G(\Cc)$, $\mathcal{K}$ an analytic subvariety of an open subset of $G(\Cc)$ and $V$ the Zariski closure of $\mathcal{K}$. We assume that $\mathcal{K}\subseteq \mathcal{A}$. Ax's theorem \cite[Theorem 1]{ax72} says that in such a case  $\mathcal{A}$ and $V(\Cc)$ tend to be in a general position, unless there is an obvious obstacle. More precisely,  there is an analytic subgroup $\mathcal{B}\subseteq G(\Cc)$ containing $V(\Cc)$ and $\mathcal{A}$ such that
$$\dim(\mathcal{K})\leqslant \dim(\mathcal{A})+\dim(V)-\dim(\mathcal{B}).$$
Ax also states a formal version of \cite[Theorem 1]{ax72} in the case of characteristic $0$ \cite[Theorem 1F]{ax72}. This formal version makes perfect sense in the case of positive characteristic case as well, and such an arbitrary characteristic version of \cite[Theorem 1F]{ax72} is our desired generalization. We will state it in Section \ref{secstat}.

A continuous map between Hausdorff spaces which is constant on a dense set is constant everywhere. The same principle applies to an algebraic map between algebraic varieties and to the Zariski topology (which is not Hausdorff). However, if we mix categories there is no reason for this principle to hold, e.g. there are non-constant \emph{analytic} maps between \emph{algebraic} varieties which are constant on a Zariski dense subset.
The main theorem of this paper roughly says that the principle above can be saved for certain formal maps (resembling homomorphisms) between an algebraic variety and an algebraic group at the cost of replacing the range of the map with its quotient by a formal subgroup of the controlled dimension. We briefly formulate this theorem below. Not to overload this introductory part with technical details, the statement is not fully precise yet, we postpone the actual statement till Section \ref{secstatmine} (it is Theorem \ref{mainthmintro}), where we also discuss how it is related to the mentioned above generalization of  \cite[Theorem 1F]{ax72}.
\begin{theorem}\label{short}
Let $V$ be an algebraic variety, $\mathcal{K}$ a Zariski dense formal subvariety of $V$, $A$ a ``good'' commutative algebraic group and $\mathcal{F}:\widehat{V}\to \widehat{A}$ a ``good'' formal map. Assume $\mathcal{F}$ vanishes on $\mathcal{K}$. Then there is a formal subgroup $\mathcal{C}\leqslant \widehat{A}$ such that $\mathcal{F}(\widehat{V})\subseteq \mathcal{C}$ and
$$\dim(\mathcal{C})\leqslant \dim(V)-\dim(\mathcal{K}).$$
\end{theorem}
Both notions of ``good'' will be clarified in Section \ref{secstatmine}. In the remaining part of the introduction we will state Ax's theorem on formal intersections \cite[Theorem 1F]{ax72} and formulate its (conjectural) generalization to the case of arbitrary characteristic. We will also formulate our question about formal maps vanishing on Zariski dense formal subvarieties and clarify its connection to Ax's theorem.
\subsection{Ax's theorem on formal intersections}\label{secstat}

We recall the formal scheme version of Ax's main theorem from \cite{ax72} below. For background on formal schemes we refer the reader to \cite[Section II.9]{ha}. Ax uses in \cite[Section 2]{ax72} the notion of a \emph{formal variety} rather than a formal scheme. Since all our formal schemes are formal subschemes of the formalization of an algebraic variety along a closed point, these two approaches are basically the same.

\begin{theorem}[{\cite[Theorem 1F]{ax72}}]\label{ax1f}
Let $G$ be an algebraic group over a field $C$ of characteristic $0$, $\widehat{G}$ the formalization of $G$ at the
origin and $\mathcal{A}$ a formal subgroup of $\widehat{G}$. Let $\mathcal{K}$ be a formal subscheme of $\mathcal{A}$ and let $V$ be
the Zariski closure of $\mathcal{K}$ in $G$. Then there is a formal subgroup $\mathcal{B}$ of $\widehat{G}$ which contains $\mathcal{A}$ and $\widehat{V}$ such that
$$\dim(\mathcal{B}) \leqslant  \dim(V ) + \dim(\mathcal{A}) - \dim(\mathcal{K}).$$
\end{theorem}
In \cite{ax72}, Ax uses Theorem \ref{ax1f} to reprove and generalize Schanuel type transcendence statements from \cite{ax71}. We recall one such statement below, which is a power series version of Schanuel's conjecture.
\begin{theorem}[{\cite[(SP)]{ax71}}]\label{axps}
Let $y_1,\ldots,y_n\in t\Cc\llbracket t\rrbracket$ be $\Qq$-linearly independent. Then
$$\td_{\Qq(t)}(y_1,\ldots,y_n,\exp(y_1),\ldots,\exp(y_n))\geqslant n.$$
\end{theorem}
Ax formulated the above theorem in \cite{ax71} for the field of complex numbers, but it is true for any field of characteristic $0$ replacing $\Cc$. As it is shown in \cite[Theorem 2]{ax72}, Theorem \ref{axps} follows from Theorem \ref{ax1f} by taking:
\begin{itemize}
\item $G$ as $\ga^n\times \gm^n$;

\item $\mathcal{A}$ as the graph of $\exp: \widehat{\ga^n}\to \widehat{\gm^n}$;

\item $V$ as the algebraic locus of $(x,\exp(x))$ (i.e. $V$ is the smallest algebraic subvariety of $G$ which is defined over $\Qq$ and which contains the point $(x,\exp(x))$);

\item $\mathcal{K}$ as the formal locus of $(x,\exp(x))$.
\end{itemize}
Ax generalizes Theorem \ref{axps} from the formal isomorphism $\exp: \widehat{\ga^n}\to \widehat{\gm^n}$ to the formal isomorphism of the form $\exp_A: \widehat{\ga^n}\to \widehat{A}$, where $A$ is an arbitrary semi-Abelian variety of dimension $n$ \cite[Theorem 3]{ax72}. It is easy then to generalize this argument further to a wider class of ``sufficiently non-algebraic'' formal isomorphisms $\mathcal{F}:\widehat{A}\to \widehat{B}$ where $A,B$ are commutative algebraic groups (see \cite{K5}).

One can state similar (to Theorem \ref{axps}) Ax-Schanuel inequalities in the case of positive characteristic. Actually, we have obtained in \cite{K6} an Ax-Schanuel statement for a ``sufficiently non-algebraic'' additive power series (so $A=B=\ga^n$) in the positive characteristic case. The aim of this line of  research is to find general geometric reasons for such Ax-Schanuel statements (with no restrictions for the characteristic of the base field). It is easy to see that \cite[Theorem 1.1]{K6} would follow from an arbitrary characteristic version of Theorem \ref{ax1f} (in the same way as in \cite{ax72}). We formulate it as a  question below.
\begin{question}\label{mainque}
Does Theorem \ref{ax1f} hold for a field $C$ of arbitrary characteristic?
\end{question}

\subsection{Some reductions}\label{secimp}
Since our main theorem is about formal maps from formalizations of algebraic varieties into \emph{commutative} algebraic groups, we will show that without loss of generality, one can assume that the ambient algebraic group $G$ (appearing in Theorem \ref{ax1f} and Question \ref{mainque}) is commutative. Ax's argument for such a reduction (Remark (5) in \cite[page 1196]{ax72}) can be generalized to the case of arbitrary characteristic using the following result.
\begin{prop}\label{redtocomm}
Suppose $G$ is an algebraic group (over an arbitrary field) and $\mathcal{A}$ is a formal subgroup of $\widehat{G}$. Then there is an algebraic subgroup $H\leqslant G$ such that
$$\widehat{H'}\subseteq \mathcal{A}\subseteq \widehat{H},$$
where $H'$ denotes the commutator subgroup of $H$.
\end{prop}
We postpone the proof till Section \ref{logdersec}, where we deal with higher tangent spaces. We obtain the following reduction.
\begin{theorem}
A positive answer to Question \ref{mainque} in the case of a commutative algebraic group $G$ implies a positive answer to Question \ref{mainque} in general.
\end{theorem}
\begin{proof}
We can assume that $G=H$, where $H$ comes from Proposition \ref{redtocomm}. Let $\pi:G\to G/G'$ denote the quotient morphism. Since $G/G'$ is commutative, using our assumption we get a formal subgroup $\mathcal{B}_0\leqslant \widehat{G/G'}$ such that
$$\dim(\mathcal{B}_0) \leqslant  \dim(\pi(V)) + \dim(\pi(\mathcal{A})) - \dim(\pi(\mathcal{K})).$$
Since $\dim(\mathcal{A})=\dim(\pi(\mathcal{A}))+\dim(G')$ and
$$\dim(\mathcal{K})- \dim(\pi(\mathcal{K}))\leqslant \dim(V)- \dim(\pi(V)),$$
we can set $\mathcal{B}$ as $\pi^{-1}(\mathcal{B}_0)$.
\end{proof}

From now on, we assume that the algebraic group $G$ appearing in the statements of Theorem \ref{ax1f} and Question \ref{mainque} is commutative.
We formulate below a weaker version of Question \ref{mainque} which we find more convenient to work with and which still implies Ax-Schanuel statements (see Remark \ref{remweak}).
\begin{question}\label{mainquew}
Does the arbitrary characteristic version of Theorem \ref{ax1f} hold, if we assume that $\mathcal{K}$ is absolutely irreducible and $\widehat{G}/\mathcal{A}$ is isomorphic to a formalization of an algebraic group?
\end{question}
\begin{remark}\label{remweak}
\begin{enumerate}
\item A formal scheme $\mathcal{K}$ is absolutely irreducible, if the ring $\mathcal{O}_{\mathcal{K}}\otimes_CC^{\alg}$ is a domain. The assumption that $\mathcal{K}$ is absolutely irreducible does not affect our Ax-Schanuel applications (i.e. Theorem \ref{schanuelax} and Theorem \ref{endoax}), where $\mathcal{K}$ is the graph of a formal isomorphism between absolutely irreducible algebraic groups.

\item Since for the applications to Ax-Schanuel type problems (discussed in general at the beginning of Section \ref{applications}), we have $\widehat{G}/\mathcal{A}\cong \widehat{B}$ (where $\mathcal{A}$ is the graph of $\mathcal{F}$ and $\mathcal{F}:\widehat{A}\to \widehat{B}$ is as in the comments below Theorem \ref{axps}), a positive answer to Question \ref{mainquew} is enough for Ax-Schanuel type applications.

\item In the case of characteristic $0$, Question \ref{mainquew}  is easily logically equivalent to Question \ref{mainque}, since any commutative formal group is isomorphic to the formalization of a vector group (see e.g. \cite[Theorem 14.2.3]{Hazew}).
\end{enumerate}
\end{remark}

\subsection{Special formal maps}\label{secstatmine}
In this subsection, we formulate the question which we attempt to answer positively in this paper, and discuss its connections to the arbitrary characteristic generalizations of Ax's theorem discussed above. We fix the following.
\begin{itemize}
\item $V$ is an algebraic variety over a fixed field $C$ and $v\in V(C)$.

\item $\widehat{V}$ is the formalization of $V$ along $v$.

\item $\mathcal{W}$ is an absolutely irreducible, Zariski dense formal subscheme of $\widehat{V}$.

\item $A$ is a commutative algebraic group.

\item $\mathcal{F}:\widehat{V}\to \widehat{A}$ is a \emph{special} formal map which vanishes on $\mathcal{W}$.
\end{itemize}
\begin{question}\label{mainquem}
Is there a formal subgroup $\mathcal{C}\leqslant \widehat{A}$ such that $\mathcal{F}(\widehat{V})\subseteq \mathcal{C}$ and
$$\dim(\mathcal{C})\leqslant \dim(V)-\dim(\mathcal{W})?$$
\end{question}
The motivation for introducing special formal maps comes from the proof of Theorem 1 in \cite{ax72}. The crucial property of the analytic maps considered in this proof is that they take invariant holomorphic forms on an algebraic group to \emph{algebraic} (and not merely holomorphic) differential forms.

This observation is the defining property of special formal maps in the case of characteristic $0$. The role of holomorphic differential forms is played by the completed module of K\"{a}hler differentials over the completion of the base ring $R$. Since the module of K\"{a}hler differentials over $R$ embeds in this completed module (see Lemma \ref{lemma1}), the definition of a formal map in the characteristic $0$ case is quite straightforward (see Definition \ref{defspecial}).

The situation is more complicated in the case of positive characteristic, since we need all the higher differential forms and we do not want to develop the formalism of higher invariant forms in this paper. Therefore, we use an equivalent notion (Definition \ref{defspecialagain}) which can be phrased in an easier way and which serves our purposes quicker. Different equivalent formulations of the notion of a special formal map are discussed briefly in Remark \ref{specequiv}.

 Morally, ``special'' should be thought of as ``homomorphism-like'' (even when the domain is not a group!), since formal homomorphisms between formalizations of algebraic groups are special (Proposition \ref{homspecial}(2)). For example, in the case of characteristic $0$ and a commutative algebraic group $A$, the formal exponential map
 $$\exp_A:\lie(A)\to A$$
 is special. In the case of the positive characteristic $p$, we have interesting special formal endomorphisms of $\ga$ of the form $\sum c_iX^{pi}$, and special formal endomorphisms of $\gm$ of the form $X^{\gamma}$, where $\gamma$ is a $p$-adic integer (see Example \ref{ax2ex}).

It is easy to give examples of formal maps which are not special. In the case when $V=A=\ga$, a formal map corresponds to a power series $F\in C\llbracket X\rrbracket$. If $F$ is special, then $F$ takes invariant forms on $\ga$ into algebraic forms on $\ga$. In the completed module of differential forms over $C\llbracket X\rrbracket$ we have
$$F^*(\ddd X)=F'\ddd X.$$
Thus if the derivative of $F$ is a not a polynomial, then $F$ is not special (it is ``if and only if'' in the case of characteristic $0$). For an arbitrary commutative algebraic group $A$, we need to replace ``derivative of $F$'' with ``logarithmic derivative (with respect to $A$) of $F$'' (see Remark \ref{specequiv}(1)). Hence it is very easy to construct formal maps which are not special and actually ``most of'' formal maps are not special. However, since the formalizations of algebraic maps are special (Proposition \ref{homspecial}(1)) and formal homomorphisms are special (Proposition \ref{homspecial}(2)), the class of special maps is still large enough for the purposes of Ax-Schanuel type considerations.

The next result says (together with Remark \ref{remweak}) that considering Question \ref{mainquem} is enough for Ax-Schanuel type applications.
\begin{theorem}\label{minetoax}
A positive answer to Question \ref{mainquem} implies a positive answer to Question \ref{mainquew}.
\end{theorem}
\begin{proof}
Let $A$ be a commutative algebraic group such that $\widehat{G}/\mathcal{A}\cong \widehat{A}$. We define $\mathcal{F}$ as the composition of the inclusion morphism $\widehat{V}\to \widehat{G}$ with the map $\widehat{G}\to \widehat{A}$ given by the above isomorphism. Since $\mathcal{F}$ is special by Proposition \ref{homspecial}, we are in the set-up of Question \ref{mainquem}. By our assumption, there is a formal subgroup $\mathcal{C}\leqslant \widehat{A}$ such that $\mathcal{F}(\widehat{V})\subseteq \mathcal{C}$ and
$$\dim(\mathcal{C})\leqslant \dim(V)-\dim(\mathcal{W}).$$
We take $\mathcal{B}$ as the preimage of $\mathcal{C}$ in $\widehat{G}$.
\end{proof}
In this paper, we aim to answer Question \ref{mainquem} positively. We explain below to which extend we have succeeded. In the case of characteristic $0$, we give the full positive answer. In the case of positive characteristic, we are forced to put extra assumptions on $\mathcal{F}$ and $A$. The theorem below is a compilation of Theorem \ref{shortchar0} and Theorem \ref{mainthm}.
\begin{theorem}\label{mainthmintro}
Assume $V$ is an algebraic variety over a perfect field $C$, $\mathcal{W}$ is an absolutely irreducible Zariski dense formal subscheme of $\widehat{V}$, $A$ is a commutative algebraic group and $\mathcal{F}:\widehat{V}\to \widehat{A}$ is a special formal map which vanishes on $\mathcal{W}$.
\begin{enumerate}
\item If $\ch(C)=0$, then there is a formal subgroup $\mathcal{C}\leqslant \widehat{A}$ such that $\mathcal{F}(\widehat{V})\subseteq \mathcal{C}$ and
$$\dim(\mathcal{C})\leqslant \dim(V)-\dim(\mathcal{W}).$$
\item If $\ch(C)>0$ and
\begin{itemize}
\item $\mathcal{F}$ is an \emph{$A$-limit formal map} (see Definition \ref{defalimit}) and

\item $A$ is an \emph{integrable algebraic group} (see Definition \ref{defintegr});
\end{itemize}
then there is a formal subgroup $\mathcal{C}\leqslant \widehat{A}$ such that $\mathcal{F}(\widehat{V})\subseteq \mathcal{C}$ and
$$\dim(\mathcal{C})\leqslant \dim(V)-\dim(\mathcal{W}).$$
\end{enumerate}
\end{theorem}
\noindent
By Proposition \ref{alimisspe}, any $A$-limit formal map $\widehat{V}\to \widehat{A}$ is special and we expect the converse to be true. Theorem \ref{alimitthm} is a partial result into this direction, since for example this theorem implies that special maps are $A$-limits if $A$ is a unipotent group (see Remark \ref{unipalimit}(2)). Thus the extra assumption on $\mathcal{F}$ seems not to be very restrictive and we hope that it can be eliminated in future by using the methods from Section \ref{secspeciallimit}. But, unfortunately, the extra integrability assumption on $A$ is quite restrictive, and we do not see how to avoid it using the methods of this paper (such as differential forms on complete rings or Hasse-Schmidt derivations).

On the positive side, Theorem \ref{mainthmintro}(2) still gives some positive characteristic Ax-Schanuel applications, including the ones from our previous work \cite{K6}. The applications of Theorem \ref{mainthmintro} are discussed in detail in Section \ref{applications}. In general, Theorem \ref{mainthmintro} allows to conclude that for a ``sufficiently non-algebraic'' formal isomorphisms between algebraic groups $A$ and $B$ the Ax-Schanuel property holds, that is if $x$ is a rational point of $A$ in the ``domain of $\mathcal{F}$'' and $x$ is not contained in any proper algebraic subgroup of $A$, then the transcendence degree of the tuple $(x,\mathcal{F}(x))$ is large. In the positive characteristic case, it covers sufficiently non-algebraic endomorphisms of vector groups and algebraic tori (see Section \ref{applications}). The most interesting formal map which is still not covered by our analysis is a formal isomorphism between the multiplicative group and an ordinary elliptic curve (in the positive characteristic case).

This paper is organized as follows. In Section 2, we gather necessary notions and facts from commutative algebra which will be needed in the rest of the paper. In Section 3, we state and prove the main theorem. In Section 4, we discuss applications of the main theorem to Ax-Schanuel type problems. In Section 5, we prove that formal maps arising from a large class of formal homomorphisms are ``good'' (as in Theorem \ref{short}).

I would like to thank the referee for her/his comments which helped to improve the presentation of this paper.

\section{Differential forms and formal maps}\label{forms}
In this section, we collect the necessary notions and results concerning differential forms. Differential forms appear in this paper in several ways and the interplay between these different types of forms is crucial for the proof of the main result (Theorem \ref{mainthm}). To be more precise, we will consider: modules of K\"{a}hler forms, global forms on schemes, complete forms and invariant forms. A more detailed analysis including the case of higher forms of Vojta \cite{voj} will be given in \cite{K7}.
\\
The aim of this section is to prove the following results.
\begin{enumerate}
\item A weak bound on the dimension of the kernel (Proposition \ref{rightdep}).

\item Homomorphisms induced by $A$-limit formal maps on forms coincide with homomorphisms induced by algebraic maps (Proposition \ref{limitforms}).

\item A characterization of vanishing of an $A$-limit formal map (Proposition \ref{limitsobs}).

\item Being $A$-limit passes to a factorization through the Frobenius morphism (Proposition \ref{frobfactorbetter}).
\end{enumerate}
The items $(1)$, $(2)$ and $(3)$ will be used in the proof of a strong bound on the dimension of the kernel (Proposition \ref{biglemma}). This strong bound together with the inductive process provided by $(4)$ will enable us to prove the main result (Theorem \ref{mainthm}).

\subsection{Notation}\label{notation}
In this subsection we set the notation and conventions which will be used throughout the paper.

All the rings considered here are commutative and have a unit $1$ with $1\neq 0$. All ring homomorphisms preserve the unit. We fix a perfect field $C$ and we assume that $C$ has characteristic $p>0$, unless clearly stated otherwise. Let $R$ be a $C$-algebra, $I$ a proper ideal of $R$ and $m\in \Nn$. By $I^m$, we denote the $m$-th power of the ideal $I$. We denote by $\fr^m_R$ (or just by $\fr^m$ if $R$ is clear from the context) the $m$-th power of the Frobenius endomorphism on $R$, by $R^{p^{m}}$ the image of $R$ by $\fr^m$ and by $R^{p^{\infty}}$, the intersection of all $R^{p^{m}}$. Since $C$ is perfect, $R^{p^{m}}$ and $R^{p^{\infty}}$ are $C$-subalgebras of $R$.
We denote by $\fr^m(I)$, the ideal of $R^{p^m}$ which is the image of $I$ by the Frobenius epimorphism $\fr^m:R\to  R^{p^m}$.
\\
A subset $\{r_1,\ldots,r_t\}\subseteq R$ is \emph{$p$-independent} (resp. \emph{$p$-basis}), if the set
$$\{r_1^{n_1}\cdot \ldots \cdot r_t^{n_t}\ |\ n_1,\ldots,n_t\in \{0,\ldots,p-1\}\}$$
is $R^p$-linearly independent (resp. a basis over $R^p$).
\\
For a local ring $T$, $\mathfrak{m}_T$ denotes its maximal ideal.
\\
We say that a $C$-algebra homomorphism $F:R\to S$ is a \emph{$0$-map} if $F(R)=C$, i.e. if $F$ factors through the structure map $C\to S$.  Since all maps of $C$-algebras take $1$ to $1$, the constant $0$-map can not appear, so our terminology shall cause no confusion. If $R$ is a local $C$-algebra with the residue field coinciding with $C$, then the following are equivalent
\begin{itemize}
\item $F$ is a $0$-map;

\item $\mathfrak{m}_R=\ker(F)$;

\item $F$ is the composition of the residue map $R\to C$ with the structure map $C\to S$.
\end{itemize}
\noindent
We call such an $F$ as above \emph{the} $0$-map (it also coincides with the categorical $0$-map in this case).
\\
We denote the $R$-module of K\"{a}hler forms $\Omega_{R/C}$ by $\Omega_R$. Clearly, for a ring extension $C\subseteq T\subseteq R^p$ we have $\Omega_R\cong\Omega_{R/T}$ and we will use this identification freely. For any multiplicative subset $E\subseteq R$, we will also identify $\Omega_{R_E}$ with $\Omega_R\otimes_RR_E$.
\\
For a local ring $R$, $\widehat{R}$ always denotes the completion of $R$ with respect to its maximal ideal (even when $R$ is an algebra over another local ring). If $M$ is an $R$-module, then $\widehat{M}$ denotes the $\widehat{R}$-module which is the completion of $M$ with respect to $\mathfrak{m}_RM$. Finally, we denote by $\widehat{\Omega}_{R}$ the completion of $\Omega_{R}$.

\subsection{Local algebras}\label{localalgebras}
\noindent%
In this subsection we will clarify different issues regarding completions and the Frobenius map. Everything is folklore but for reader's convenience we collect the necessary facts here.

For a ring $S$, an ideal $I$ and an $S$-module $M$, by the $I$-adic topology on $M$, we mean the topology given by the filtration $(I^mM)_m$. Let us fix a Noetherian local $C$-algebra $R$. By the standard topology on $R$, we mean the $\mathfrak{m}_R$-adic topology. Let us fix $m\in \Nn$ and let $q=p^m$.
\begin{lemma}\label{easytopology}
The standard topology on $R$ coincides with the $\mathfrak{m}_{R^q}$-adic one.
\end{lemma}
\begin{proof}
Since the radical of $\mathfrak{m}_{R^{q}}R$ coincides with $\mathfrak{m}_{R}$ and $R$ is Noetherian, the two topologies are the same.
\end{proof}
\noindent
We note below a result regarding the structure of complete $C$-algebras. It follows e.g. from the proof of \cite[Theorem 29.4(iii)]{mat}.
\begin{prop}\label{completestr}
Assume that $R$ is a complete ring of Krull dimension $r$ and with the residue field coinciding with $C$. Then for any system $\{x_1,\ldots,x_r\}$ of local parameters of $R$ we have:
\begin{enumerate}
\item the elements $x_1,\ldots,x_r$ are analytically independent over $C$;

\item the extension of rings $C\llbracket x_1,\ldots,x_r\rrbracket\subseteq R$ is finite.
\end{enumerate}
\end{prop}
\noindent
We can prove now the remaining necessary properties of $R$ under some extra assumptions.
\begin{prop}\label{completefrob}
Assume that the residue field of $R$ coincides with $C$ and the Frobenius map is injective on $R$. We also assume that $R$ is complete or $R$ is a localization of a $C$-algebra of finite type. Then we have:
\begin{enumerate}
\item The ring $R$ is finite over $R^q$.

\item The standard topology on $R^q$ coincides with the topology induced from $R$.

\item The natural map $\Psi:\widehat{(R^q)}\to \widehat{R}$ is injective and its image coincides with $(\widehat{R})^q$.

\item $R\cap(\widehat{R})^q=R^q$.
\end{enumerate}
\end{prop}
\begin{proof}
For $(1)$ note that if $T=C[t_1,\ldots,t_m]$, then $T=T^q[t_1,\ldots,t_m]$. If $R=T_S$, then
$$R=T^q_{S^q}[t_1,\ldots,t_m]=R^q[t_1,\ldots,t_m],$$
so $R$ is finite over $R^q$. In the complete case, Proposition \ref{completestr} implies that $R$ is finite over $R^{q}$.

For $(2)$, by the Artin-Rees lemma \cite[Theorem 8.6]{mat} and $(1)$, the topology on $R^{q}$ coincides with the subspace topology induced from the $\mathfrak{m}_{R^q}$-adic topology on $R$. By Lemma \ref{easytopology}, the latter topology on $R$ coincides with the standard topology.

For $(3)$, by $(2)$ and \cite[Theorem 8.1]{mat}, the map $\Psi$ is injective. Note also that for any sequence $(r_i)$ of elements of $R$, we have
$$\text{$(r_i)$ is a Cauchy sequence in $R\ \ \ \ \Longleftrightarrow\ \ \ \ $  $(r_i^q)$ is a Cauchy sequence in $R^q$.}$$
The left-to-right implication does not use any extra assumption on $R$ and implies that $(\widehat{R})^q$ is contained in the image of $\Psi$. The right-to-left implication uses the injectivity of the Frobenius map on $R$ and gives the reverse inclusion.

For the last part, it is enough to notice that the natural map
$$R\otimes_{R^q}(\widehat{R})^q\to \widehat{R}$$
is a ring isomorphism, which follows from
$(1)$, $(2)$ and \cite[Theorem 8.7]{mat}.
\end{proof}
\begin{definition}\label{r[m]}
If $R$ is a local ring, then by $R[m]$ we denote the quotient ring $R/(\fr^m_R(\mathfrak{m}_R)R)$.
\end{definition}

\begin{remark}\label{rnnotation}
If $I=(r_1,\ldots,r_n)$ is a finitely generated ideal, then the sequences $(I^m)_m$ and $(r_1^m,\ldots,r_n^m)_m$ give the same topologies. Hence for a local Noetherian ring $R$, we have $\widehat{R}\cong \li R[m]$. It will be more convenient for us to use the rings $R[m]$ rather than the rings $R/\mathfrak{m}_R^m$, which are usually used to define the completed ring $\widehat{R}$.
\end{remark}

\subsection{Differential forms and a dense formal subscheme}\label{dense}
\noindent%
The main result of this subsection is Proposition \ref{rightdep}. A much stronger result holds in the case of characteristic $0$ (Proposition \ref{char0dep}). In Section \ref{dependence}, we will prove a stronger version of Proposition \ref{rightdep} under extra assumptions.

First, we need two lemmas about differential forms over complete algebras. Both are folklore and the first one does not need our general assumption about the characteristic.
\begin{lemma}\label{lemma1}
Let $R$ be a local $C$-algebra. Then we have:
\begin{enumerate}

\item $\widehat{\Omega}_{R}\cong \widehat{\Omega}_{\widehat{R}}$.

\item If $R$ is a localization of an affine $C$-algebra, then $\widehat{\Omega}_{R}\cong \Omega_R\otimes_R\widehat{R}$.
\end{enumerate}
\end{lemma}
\begin{proof}
By the formula \cite[20.7.14.2]{EGAIVi}, both $\widehat{R}$-modules in $(1)$ are naturally isomorphic to $\li(\Omega_{R/\mathfrak{m}^m})$.
\\
Since $R$ is a localization of an affine $C$-algebra, $\Omega_R$ is a finitely generated $R$-module. By \cite[Theorem 7.2(a)]{comm}, we get the second part.
\end{proof}
The second lemma uses the characteristic assumption. We will comment on its characteristic zero version in Section \ref{char0}.
\begin{lemma}\label{lemma2}
Let $\mathcal{S}$ be a Noetherian complete local $C$-algebra with the residue field $C$ and Krull dimension $r$. Suppose that $\mathcal{S}$ is a domain and let $L$ be its fraction field. Then we have:
\begin{enumerate}
\item The $\mathcal{S}$-module $\Omega_{\mathcal{S}}$ is complete i.e. $\Omega_{\mathcal{S}}\cong \widehat{\Omega}_{\mathcal{S}}$.

\item $\dim_L(\Omega_L)=r$.
\end{enumerate}
\end{lemma}
\begin{proof}
Let $x_1,\ldots,x_r$ be a local system of parameters of $\mathcal{S}$ and $\mathcal{A}:=C\llbracket x_1,\ldots,x_r\rrbracket$. By Proposition \ref{completestr},
the extension $\mathcal{A}\subseteq \mathcal{S}$ is finite and $\mathcal{A}$ is isomorphic as a $C$-algebra to the power series algebra in $r$ variables.
By \cite[Lemme 21.9.4]{EGAIVi}, we get that
$$\widehat{\Omega}_{\mathcal{S}}\cong \Omega_{\mathcal{S}/C\llbracket x_1^p,\ldots,x_r^p\rrbracket}.$$
Since $C$ is perfect, we have $C\llbracket x_1^p,\ldots,x_r^p\rrbracket=\mathcal{A}^p\subseteq \mathcal{S}^p$ and the first part follows.
\\
For the second part, let $K$ be the fraction field of $\mathcal{A}$. Since $K$ is the field of Laurent  power series in $r$ variables, we have $\dim_K\Omega_K=r$. It is enough to show that for any finite field extension $K\subseteq L$, we have $\dim_L\Omega_L=r$. By \cite[Theorem 25.3]{mat}, we can assume that $L=K(a^{1/p})$ for some $a\in K\setminus K^p$. There is a $p$-basis (= differential basis) of $K$ of the form $\{a_1,\ldots,a_r\}$, where $a=a_1$. Then $\{a_1^{1/p},a_2,\ldots,a_r\}$ is a $p$-basis of $L$, therefore $\dim_L(\Omega_L)=r$.
\end{proof}
\begin{remark}\label{rstar}
For a characteristic $0$ version of Lemma \ref{lemma2}, we need to work with the module of the completed forms, and we still get
$$\dim_L(\widehat{\Omega}_{\mathcal{S}}\otimes_{\mathcal{S}}L)=\dim({\mathcal{S}}).$$ 
Note that in this case, the module $\Omega_{\mathcal{S}}$ is huge and the dimension of $\Omega_{\mathcal{S}}\otimes_{\mathcal{S}}L$ is infinite.
\end{remark}
\noindent
We can prove now the main result of this subsection, which is the item $(2)$ from the beginning of this section.
\begin{prop}\label{rightdep}
Let $R$ be a local domain which is a localization of an affine algebra over $C$ and whose residue field coincides with $C$.
Let $P$ be a prime ideal of $\widehat{R}$, $\mathcal{S}=\widehat{R}/P$ and $L$ be the fraction field of $\mathcal{S}$. Then $\mathcal{S}$ is Noetherian, complete and we have
\begin{enumerate}
\item the natural map
 $\Omega_R\otimes_R\mathcal{S}\to \Omega_{\mathcal{S}}$ is onto;

\item if the map $R\to \mathcal{S}$ is injective, then
$$\dim_L\ker(\Omega_R\otimes_RL\to \Omega_{\mathcal{S}}\otimes_{\mathcal{S}}L)=\dim(R)-\dim(\mathcal{S}).$$
\end{enumerate}
\end{prop}
\begin{proof}
The local ring  $\mathcal{S}$ is complete by \cite[Theorem 8.11]{mat} and Noetherian by \cite[Theorem 8.12]{mat}.

By Lemmas \ref{lemma1} and \ref{lemma2}, we have $\Omega_{\widehat{R}}\cong {\Omega}_R\otimes_R\widehat{R}$.
Since the map $\widehat{R}\to \mathcal{S}$ is onto, the induced map
$\Omega_{\widehat{R}}\otimes_{\widehat{R}}\mathcal{S}\to \Omega_{\mathcal{S}}$ is onto as well. By the associativity of the tensor product, the map $\Omega_R\otimes_R\mathcal{S}\to \Omega_{\mathcal{S}}$ is onto giving $(1)$.

We proceed to show the item $(2)$, so we assume that the map $R\to \mathcal{S}$ is injective.
Since the functor $\cdot\otimes_{\mathcal{S}}L$ is right-exact, the map
$$\Omega_R\otimes_RL\to \Omega_{\mathcal{S}}\otimes_{\mathcal{S}}L$$
is onto. By Lemma \ref{lemma2}(2), $\dim_L\Omega_L=\dim(\mathcal{S})$. Since, $\Omega_L\cong  \Omega_{\mathcal{S}}\otimes_{\mathcal{S}}L$, it is enough to show now that $\dim_L\Omega_R\otimes_RL=\dim(R)$.
\\
Let $K$ be the fraction field of $R$. Since the map $R\to \mathcal{S}$ is injective, $K$ embeds over $R$ into $L$. Since $R$ is a localization of an affine algebra over $C$, by \cite[Theorem 5.6]{mat} $\trd_CK=\dim(R)$, so by \cite[Theorem 26.10]{mat} we get that $\dim_K\Omega_K=\dim(R)$. It is enough to see now that $\Omega_R\otimes_RL\cong \Omega_K\otimes_KL$.
\end{proof}
\begin{remark}\label{char0onto}
Similarly we can show that if $\ch(C)=0$, then the map $\Omega_R\otimes_R\mathcal{S}\to \widehat{\Omega}_{\mathcal{S}}$ is onto and if
the map $R\to \mathcal{S}$ is injective, then
$$\dim_L\ker(\Omega_R\otimes_RL\to \widehat{\Omega}_{\mathcal{S}}\otimes_{\mathcal{S}}L)=\dim(R)-\dim(\mathcal{S}).$$
For the proof of the surjectivity part, we proceed as in the proof of Proposition \ref{rightdep} after noticing that the map $\widehat{\Omega}_{\widehat{R}}\otimes_{\widehat{R}}\mathcal{S}\to \widehat{\Omega}_{\mathcal{S}}$ is onto, since, by Lemma \ref{lemma1}, we have
$$\widehat{\Omega}_{\widehat{R}}\cong \Omega_R\otimes_R\widehat{R}.$$
For the dimension equality, we proceed again as in the proof of Proposition \ref{rightdep} using Remark \ref{rstar}. 
\end{remark}

\subsection{Differential forms and $p$-normal rings}\label{pnormal}
\noindent
To deal with rational points over $C$, we need to ``smoothen out'' our algebras a little bit. The notion of a normal ring is too strong for us, we will use a weaker version which is discussed in this subsection. We also prove here a result about the ``absolute constants'' of an absolutely irreducible formal scheme (Proposition \ref{lpinfty}).

Let $R$ be a $C$-algebra which is a domain, $L$ be its field of fractions, and $L^{\alg}$ a fixed algebraic closure of $L$. For each $m\in \Nn$, we denote by
\begin{itemize}
\item $R^{p^{-m}}$, the preimage of $R$ under the map $\fr^m:L^{\alg}\to L^{\alg}$;

\item $R^{p^{-\infty}}$, the union of all $R^{p^{-m}}$;

\item $R'$, the intersection of $R^{p^{-\infty}}$ and $L$.
\end{itemize}
\noindent
Clearly $R'$ is a $C$-algebra extension of $R$.
\begin{definition}
We say that $R$ is \emph{$p$-normal} if $R=R'$.
\end{definition}
\noindent
Obviously, normal rings are $p$-normal. We observe below that $p$-normality behaves like normality.
\begin{fact}\label{pnormallocal}
Let $R$ be as above. Then:
\begin{enumerate}
\item The ring $R'$ is $p$-normal.

\item If $R$ is $p$-normal and $S$ is a multiplicative subset of $R$, then $R_S$ is $p$-normal.

\item If $R$ is of finite type over $C$, then $R'$ is of finite type over $C$.
\end{enumerate}
\end{fact}
\begin{proof}
The first part is obvious. For the proof of the second part, let us take $\alpha\in (R_S)'$. Then there is $m\in \Nn$, $x\in R$ and $s\in S$ such that $\alpha^{p^m}=r/s$. Thus $(s\alpha)^{p^m}=rs^{p^m-1}$ and $s\alpha \in R'=R$, so $\alpha\in R_S$.
\\
For the last part, let $\bar{R}$ denote the normalization of $R$. By \cite[page 262]{mat}, $\bar{R}$ is finitely generated as an $R$-module. Since $R$ is Noetherian, $\bar{R}$ is Noetherian as an $R$-module, so $R'$ is finitely generated as an $R$-module as well. In particular, $R'$ is of finite type over $C$.
\end{proof}
\begin{example}
It is easy to see that if $p=2$ or $p=3$, then the ring $C[X^2,X^3]$ is not $p$-normal. Similar examples can be easily constructed for any prime $p$.
\end{example}
\noindent
The next result is the reason why we have introduced the notion of a $p$-normal ring.
\begin{lemma}\label{formppower}
Let $f:T\to R$ be a $C$-algebra homomorphism between domains, $K$ be the fraction field of $R$ and assume that $R$ is $p$-normal. Then the following are equivalent:
\begin{enumerate}

\item for any $\omega\in \Omega_T$, we have $f_*(\omega)=0$ in $\Omega_K$,

\item $f(T)\subseteq R^p$.

\end{enumerate}
\end{lemma}
\begin{proof}
Only the implication $(1)\Rightarrow (2)$ needs a proof. Take $x\in T$ and let $y=f(x)\in R$. By the assumption, $\ddd y=0$ in $\Omega_K$. Then $y\in K^p$, since otherwise $\{y\}$ could be extended to a $p$-basis of $K$. Since $R$ is $p$-normal, we get $y\in R^p$.
\end{proof}
\noindent
We will see now that taking the $p$-normalization does not affect the $C$-rational points, which could be the case for the usual normalization.
\begin{lemma}\label{pnormalization}
Let $R$ be an affine $C$-algebra which is a domain and $P$ be an ideal of $R$ such that $R/P=C$. Then there is an ideal $P'$ in $R'$ such that $R\cap P'=P$ and $R'/P'=C$.
\end{lemma}
\begin{proof}
The extension $R\subseteq R'$ is integral, so by \cite[Theorem 9.3]{mat}, there is a maximal ideal $P'$ in $R'$ such that $P'\cap R=P$. By the definition of $R'$, the field extension
$$C=R/P\subseteq R'/P'$$
is purely inseparable. Since $C$ is perfect, $R'/P'=C$.
\end{proof}
We show below that for absolutely irreducible formal scheme $\mathcal{W}$, the ``absolute constants'' of the field of fractions of $\mathcal{O}_{\mathcal{W}}$ coincide with $C$. This property is used in the proof of the main theorem, see Remark \ref{absirrrem}.
\begin{prop}\label{lpinfty}
Assume that $R$ is a complete $C$-algebra of Krull dimension $r$ with the residue field coinciding with $C$ and such that $R\otimes_CC^{\alg}$ is a domain. Then $L^{p^{\infty}}=C$.
\end{prop}
\begin{proof}
Let $R'$ denote the domain $R\otimes_CC^{\alg}$. Then the ring $L\otimes_CC^{\alg}$ is a localization of $R'$, so it is a domain as well. Hence $C$ is relatively algebraically closed in $L$ and it is enough to show that the extension $C\subseteq L^{p^{\infty}}$ is algebraic.

By Proposition \ref{completestr}(2), there is a finite field extension $C((x_1,\ldots,x_r))\subseteq L$. Since for any integer $n$ we have $(L^{p^n})^{p^{\infty}}=L^{p^{\infty}}$, we can assume that the extension $C((x_1,\ldots,x_r))\subseteq L$ is separable algebraic. For the notion of a \emph{Hasse-Schmidt derivation}, the reader is referred to \cite[Section 27]{mat} (where they are called \emph{higher derivations}). Let $\partial$ be the $r$-tuple of the standard Hasse-Schmidt derivations on $C((x_1,\ldots,x_r))$. By \cite[Theorem 27.2]{mat} (since a separable algebraic field extension is \'{e}tale), there is a unique $r$-tuple $\partial'$ of Hasse-Schmidt derivations on $L$ extending $\partial$. Clearly, $C$ is the field of the absolute constants of $\partial$. Let $C'$ be the field of the absolute constants of $\partial'$. Since the field extension $C((x_1,\ldots,x_r))\subseteq L$ is algebraic, the field extension $C\subseteq C'$ is algebraic as well. We have $L^{p^{\infty}}\subseteq C'$ which finishes the proof.
\end{proof}

\subsection{Factoring through Frobenius}\label{sectionftf}
In this subsection we collect necessary facts about the Frobenius homomorphism. Let $T$ and $R$ be $C$-algebras and $\sigma:C\to C$ be an automorphism. We define
$$R^{\sigma}:=R\otimes_{\sigma}C.$$
We present $R$ as $C[X]/I$ when $X$ is a (possibly infinite) tuple of variables. Then the following $C$-algebras are isomorphic:
\begin{itemize}
\item $R^{\sigma}$,

\item $R$ with the algebra structure given by $\iota\circ \sigma^{-1}$, where $\iota:C\to R$ is the original $C$-algebra structure on $R$,

\item $C[X]/(\sigma(I))$.
\end{itemize}
The definition of $R^{\sigma}$ naturally extends to $C$-schemes and we have the isomorphisms above, where the last item is understood as follows. If $V=\spec(R)$ is an affine variety, then $\sigma(V)$ (image inside the affine space) coincides with $\spec(C[X]/(\sigma(I)))$.
\\
\\
Assume now that $\sigma$ extends to an endomorphism $\sigma_R:R\to R$ (usually not a $C$-algebra map!). Then we have a $C$-algebra map
$$\sigma_R:R^{\sigma}\to R.$$
If $\sigma_R$ is a monomorphism, then there is one more isomorphism of $C$-algebras
$$\sigma_R(R)\cong R^{\sigma}$$
and $\sigma_R$ corresponds to the inclusion $\sigma_R(R)\subseteq R$.

We will apply the considerations above to the map $\fr:C\to C$. Clearly it extends to $\fr_R:R\to R$. We notice below equivalent conditions which in the case of a $p$-normal $R$ are also equivalent to the ones in Lemma \ref{formppower}.
\begin{fact}\label{facfroring}
Let $f:T\to R$ be a $C$-algebra map. The following are equivalent:
\begin{enumerate}
\item $f(T)\subseteq R^p$.

\item There is a $C$-algebra map $f_{(1)}:T\to R^{\fr}$ making the following diagram commutative
\begin{equation*}
  \xymatrix{
T\ar[rr]^{f}\ar[rrd]_{f_{(1)}} & & R \\
 & & R^{\fr}  \ar[u]^{\fr} }
\end{equation*}

\item There is a $C$-algebra map $f^{(1)}:T^{\fr^{-1}}\to R$ making the following diagram commutative
\begin{equation*}
  \xymatrix{
T^{\fr^{-1}}\ar[rrd]^{f^{(1)}} & &  \\
T\ar[u]^{\fr}\ar[rr]^{f} & & R   }
\end{equation*}
\end{enumerate}
\end{fact}

\begin{remark}\label{facfrosch}
Inverting the arrows in Fact \ref{facfroring}, we obtain that for any morphism of $C$-schemes $\varphi:V\to W$ the following are equivalent
\begin{enumerate}
\item There is a morphism $\varphi_{(1)}:V^{\fr}\to W$ making the following diagram commutative
\begin{equation*}
  \xymatrix{
V\ar[rr]^{\varphi}\ar[d]_{\fr} & & W \\
V^{\fr}\ar[rru]_{\varphi_{(1)}} & &  }
\end{equation*}

\item There is a morphism $\varphi^{(1)}:V\to W^{\fr^{-1}}$ making the following diagram commutative
\begin{equation*}
  \xymatrix{
 & & W^{\fr^{-1}} \ar[d]^{\fr}\\
V \ar[rru]^{\varphi^{(1)}}\ar[rr]^{\varphi} & & W   }
\end{equation*}
\end{enumerate}
\end{remark}


\subsection{Points and forms}\label{sectionpoints}
The aim of this subsection is to clarify the passage from local homomorphisms to rational points and its effect on differential forms. Throughout this subsection we fix:
\begin{itemize}
\item a $C$-algebra $R$;

\item an absolutely irreducible $C$-scheme $Y$;

\item $y\in Y(C)$;


\item $T:=\mathcal{O}_{Y,y}$;

\item $m\in \Nn$ and $q:=p^m$.
\end{itemize}
\noindent%
In the applications, $Y$ will be an irreducible algebraic group over $C$, so we will assume later that $Y$ is smooth.
\\
There is a natural morphism of $C$-schemes $\spec(T)\to Y$ such that the image of the closed point of $\spec(T)$ is the closed point of $Y$ underlying the rational point $y$ and the image of the generic point of $\spec(T)$ is the generic point of $Y$. By composing with $\spec(T)\to Y$, any $C$-algebra homomorphism $f:T\to R$ gives an $R$-rational point of $Y$ (i.e. a $C$-scheme morphism $\spec(R)\to Y$) which we denote by $f_Y\in Y(R)$. Similarly, any $C$-algebra homomorphism $\mathcal{F}:\widehat{T}\to \widehat{R}$ gives a point $\mathcal{F}_Y\in Y(\widehat{R})$ which is the composition of the following sequence of morphisms:
\begin{equation*}
  \xymatrix{\spec(\widehat{R}) \ar[r]^{\mathcal{F}}& \spec(\widehat{T})\ar[r]^{}& \spec(T) \ar[r]^{}& Y.}
\end{equation*}
\noindent
We will need the  observations below about such rational points.
\begin{lemma}\label{points}
Let $f:T\to R$ be a local $C$-algebra homomorphism and $\mathcal{F}:\widehat{T}\to \widehat{R}$ be a $C$-algebra homomorphism. We have the following.
\begin{enumerate}
\item The image of $f_Y$ in $Y(\widehat{R})$ coincides with $(\widehat{f})_Y$.

\item Suppose $R$ satisfies the assumptions in Proposition \ref{completefrob}. Then $\mathcal{F}_Y\in Y((\widehat{R})^{q})$ if and only if $\mathcal{F}(\widehat{T})\subseteq (\widehat{R})^{q}$.
\end{enumerate}
\end{lemma}
\begin{proof}
The first part is an easy diagram chase.
\\
The right-to-left implication in the second part is obvious. For the remaining implication, assume that $\mathcal{F}_Y\in Y((\widehat{R})^{q})$. It means that the morphism  $\mathcal{F}_Y:\spec(\widehat{R})\to Y$ factors through a morphism  $\mathcal{F}':\spec((\widehat{R})^{q})\to Y$. Clearly,  $\mathcal{F}'$ maps the closed point of $\spec((\widehat{R})^{q})$ to $y$, so it factors through a morphism $\mathcal{F}':\spec((\widehat{R})^{q})\to \spec(T)$. It means that $\mathcal{F}(T)\subseteq (\widehat{R})^{q}$. Hence it remains to show that $(\widehat{R})^{q}$ is complete in $\widehat{R}$. It follows from Proposition \ref{completefrob}(2) (the induced topology on $(\widehat{R})^{q}$ coincides with the standard topology) and Proposition \ref{completefrob}(3) ($(\widehat{R})^{q}$ is complete with respect to the standard topology), since, by \cite[Theorem 8.12]{mat}, $\widehat{R}$ is Noetherian.
\end{proof}
\noindent
Abusing the notation a bit, we denote by $\Omega_Y$, the $\mathcal{O}_Y(Y)$-module of global sections of the sheaf of K\"{a}hler differential forms over $C$ on $Y$ (see \cite[Section II.8]{ha}). Clearly, $\Omega_R=\Omega_{\spec(R)}$. Any morphism of $C$-schemes $f:W\to Y$ functorially induces a homomorphism of  $\mathcal{O}_Y(Y)$-modules $f^*:\Omega_Y\to \Omega_W$.
\begin{lemma}\label{newvanish}
If $f:W\to Y$ factors through Frobenius as in Remark \ref{facfrosch}, then $f^*:\Omega_Y\to \Omega_W$ is the $0$-map.
\end{lemma}
\noindent
Taking $W=\spec(T)$, we have a natural homomorphism $\Omega_Y\to \Omega_{T}$. We collect the necessary properties of this map.
\begin{prop}\label{localforms}
Assume that $Y$ is irreducible and smooth (so for each $t\in Y$, the local ring $\mathcal{O}_{Y,t}$ is regular). Then we have:
\begin{enumerate}
\item The map $\Omega_Y\to \Omega_{T}$ is one-to-one.

\item The map $\Omega_{T}\to \Omega_{\widehat{T}}$ is  one-to-one.
\end{enumerate}
\end{prop}
\begin{proof} By \cite[Theorem II.8.8]{ha}, the sheaf of K\"{a}hler differential forms over $C$ on $Y$ is locally free. The injectivity in $(1)$ follows.
\\
For $(2)$, by Lemma \ref{lemma1}(1) and Lemma \ref{lemma2}(1) it is enough to show that the map $\Omega_{T}\to \widehat{\Omega}_{T}$ is one-to-one. By \cite[Theorem 8.9]{mat}, the kernel of this map is trivial.
\end{proof}
\noindent
We can regard now $\Omega_Y$ as a $\mathcal{O}_Y(Y)$-submodule of $\Omega_{T}$ and
 $\Omega_{T}$ as a $T$-submodule of $\Omega_{\widehat{T}}$. Having this identifications in mind, by an easy diagram chase (left to the reader), we obtain the following.
\begin{lemma}\label{coincide}
Assume that $Y$ is irreducible and smooth. Let $f:T\to R$ and $\mathcal{F}:\widehat{T}\to \widehat{R}$ be local $C$-algebra homomorphisms. Then for any $\omega\in \Omega_Y$ we have:
\begin{enumerate}
\item $\mathcal{F}_*(\omega)=\mathcal{F}_Y^*(\omega)$.

\item $\widehat{f}_*(\omega)=f_*(\omega)=f_Y^*(\omega)=\widehat{f}^*_Y(\omega)$.
\end{enumerate}
\end{lemma}
\noindent
Let us now specialize the set-up to the following case:
\begin{itemize}
\item $Y=A$ is a commutative connected algebraic group over $C$ of dimension $n$;

\item $y=0\in A(C)$;

\item $T:=\mathcal{O}_{A,0}$.
\end{itemize}
\noindent
There is a $C$-subspace $\Omega^{\inv}_A$ of $\Omega_A$ consisting of \emph{invariant} (see \cite{Ros}) forms. We will need the following well-known results.
\begin{prop}\label{invforms}
We identify $\Omega_A$ with a subspace of $\Omega_T$ using Proposition \ref{localforms}. Then we have:
\begin{enumerate}
\item $\dim_C\Omega^{\inv}_A=n$,

\item $T\Omega^{\inv}_A=\Omega_T$,


\item for any $\omega\in \Omega^{\inv}_A$ and morphisms of $C$-schemes $f,g:V\to A$
$$(f+g)^*(\omega)=f^*(\omega)+g^*(\omega).$$
\end{enumerate}
\end{prop}
\begin{proof}
For $(1)$ and $(3)$, see \cite{Ros} and for $(2)$, see \cite[Theorem 8.8]{ha}.
\end{proof}

\subsection{Complete Hopf algebras}\label{hopf}
In this subsection, we collect the necessary facts about complete Hopf algebras. We consider the category of complete local $C$-algebras with residue fields coinciding with $C$. Let $\mathcal{R},\mathcal{S}$ be such local $C$-algebras. The coproduct in this category is the completed tensor product
$$\mathcal{R}\widehat{\otimes}\mathcal{S}:=\widehat{\mathcal{R}\otimes \mathcal{S}},$$
where the $C$-algebra $\mathcal{R}\otimes \mathcal{S}$ is completed with respect to the ideal generated by $\mathfrak{m}_\mathcal{R}\otimes \mathcal{S} + \mathcal{R}\otimes \mathfrak{m}_\mathcal{S}$.
\\
A \emph{complete Hopf algebra} is a quadruple $(H,\Delta,S,\varepsilon)$ such that $\mathcal{H}$ is a complete local $C$-algebra with the residue field $C$ and
$$\Delta:\mathcal{H}\to \mathcal{H}\widehat{\otimes}\mathcal{H},\ \ S:\mathcal{H}\to \mathcal{H},\ \ \varepsilon:\mathcal{H}\to C$$
are $C$-algebra homomorphisms such that the diagrams analogous to the diagrams from \cite[page 8]{Water} commute. For example the ``complete coassociativity'' is expressed by the following commutative diagram:
\begin{equation*}
  \xymatrix{
\mathcal{H}\widehat{\otimes}\mathcal{H}\widehat{\otimes}\mathcal{H} & & \mathcal{H}\widehat{\otimes}\mathcal{H}  \ar[ll]_{\ \ \ \ \ \ \Delta \widehat{\otimes}\id}\\
\mathcal{H}\widehat{\otimes}\mathcal{H} \ar[u]_{\id\widehat{\otimes} \Delta}& & \mathcal{H} \ar[ll]^{\Delta} \ar[u]^{\Delta}, }
\end{equation*}
i.e. $\otimes$ from the definition of a Hopf algebra is replaced with $\widehat{\otimes}$. The category of complete Hopf algebras is antiequivalent to the category of \emph{formal group schemes} over $C$, see e.g. \cite[p. 493]{Hazew}. The notions of a commutative complete Hopf algebra and a commutative formal group are clear. We will need a result about quotients of complete Hopf algebras. Recall our general assumption $\ch(C)=p>0$, the notation from Remark \ref{rnnotation} and rings of the form $\mathcal{A}[m]$ which were introduced in Definition \ref{r[m]}.
\begin{prop}[Lemma 1.1 in \cite{manin}]\label{frker}
Let $\mathcal{A}$ be a complete Hopf algebra and $m\in \Nn$. Then $\mathcal{A}[m]$ has a (complete) Hopf algebra structure such that the map $\mathcal{A}\to \mathcal{A}[m]$ is a complete Hopf algebra homomorphism.
\end{prop}
\noindent%
If $G$ is a group scheme over $C$ and $e\in G(C)$ is the identity, then $\widehat{\mathcal{O}_{G,e}}$ is naturally a complete Hopf algebra (commutative, if $G$ is commutative), see \cite[Section 2.2]{manin}. We denote the corresponding formal group scheme by $\widehat{G}$.

If $\mathcal{H}$ is a complete Hopf algebra over $C$ and $\mathcal{S}$ is a complete $C$-algebra, then the set of all local $C$-algebra maps from $\mathcal{H}$ to $\mathcal{S}$ has a natural structure of a commutative group and the $0$-map (as defined in Section \ref{notation}) is the identity of this group.

Let $V,v$ be as above and $A$ be a commutative algebraic group over $C$. We will need the following formal version of Proposition \ref{invforms}(3) which follows from the identifications of Proposition \ref{localforms} and Lemma \ref{coincide}.
\begin{prop}\label{foradd}
For any formal maps $\mathcal{F}_1,\mathcal{F}_2:\widehat{V}\to \widehat{A}$ and $\omega\in \Omega^{\inv}_A$ we have
$$(\mathcal{F}_1+\mathcal{F}_2)^*(\omega)=(\mathcal{F}_1)^*(\omega)+(\mathcal{F}_1)^*(\omega).$$
\end{prop}
\noindent
We finish with an easy observation regarding the formalization of algebraic group and points as in Section \ref{sectionpoints}. Its proof is an easy diagram chase.
\begin{lemma}\label{newadd}
Let $A$ be a group scheme over $C$ and $T=\mathcal{O}_{A,e}$. Let $\mathcal{F},\mathcal{G}:\widehat{T}\to \widehat{R}$ be $C$-algebra homomorphisms. Then
$$(\mathcal{F}+\mathcal{G})_A=\mathcal{F}_A+\mathcal{G}_A.$$
\end{lemma}

\subsection{Formal $A$-limit maps}\label{alimit}
In this subsection we will define the ``good'' formal maps from the introduction and prove their (good) properties. We fix the following:
\begin{itemize}
\item A local $C$-algebra $R$ such that $R$ is reduced, Noetherian and the residue field of $R$ coincides with $C$. We also assume that $R$ is complete or a localization of a $C$-algebra of finite type.

\item A local $C$-algebra $T$.

\item A commutative complete Hopf algebra $\mathcal{H}$ over $C$.

\item A complete local $C$-algebra $\mathcal{S}$.
\end{itemize}
\noindent
We will not always use all the properties imposed on $R$, but for simplicity of the presentation we make the assumptions as above.
In the next definition, we will use the notation from Remark \ref{rnnotation}.
\begin{definition}
Let $(f_m:T\to R)_m$ be a sequence of local $C$-algebra homomorphisms. We say that such a sequence is \emph{strongly Cauchy} if $(f_m[m]:T[m]\to R[m])_m$ is a morphism between inverse systems of rings.
\end{definition}
\begin{remark}
Let us fix a sequence $(f_m)_m$ as above.
\begin{enumerate}
\item
The following are equivalent:
\begin{enumerate}
\item $(f_m)_m$ is strongly Cauchy;

\item for each $m\in \Nn$ we have
$$f_m[m+1]=f_{m+1}[m+1]:T[m+1]\to R[m+1];$$

\item for each $m\in \Nn$ and $t\in T$ we have
$$f_m(t)-f_{m+1}(t)\in \fr^m_R(\mathfrak{m}_R)R.$$
\end{enumerate}
\noindent
Since $\fr^m_R(\mathfrak{m}_R)R\subseteq \mathfrak{m}_R^{p^m}$, a strongly Cauchy sequence is uniformly Cauchy i.e. for each $t\in T$, the sequence $(f_m(t))_m$ is Cauchy (uniformly in $t$). Hence for a strongly Cauchy sequence we obtain
$$\li(f_m)=\li(f_m[m]):\widehat{T}\to \widehat{S}.$$

\item Since for each $m$ we have $f_m[m]=\widehat{f_m}[m]$, the sequence $(f_m:R\to S)_m$ is strongly Cauchy if and only if the sequence $(\widehat{f_m}:\widehat{R}\to \widehat{S})_m$ is strongly Cauchy.
\end{enumerate}
\end{remark}
\noindent
We will need some properties of strongly Cauchy sequences.
\begin{lemma}\label{limitpower}
Let $(f_m:T\to R)$ be a strongly Cauchy sequence and assume that for almost all $m$, we have $f_m(T)\subseteq R^{p^k}$. Then $\li(f_m)(\widehat{T})\subseteq (\widehat{R})^{p^k}$.
\end{lemma}
\begin{proof}
By Proposition \ref{completefrob}(2) and \ref{completefrob}(3) (as in the proof of Lemma \ref{points}), the subring $(\widehat{R})^{p^k}$ is complete in $\widehat{R}$. Hence the result holds even for pointwise Cauchy sequences of maps.
\end{proof}
\noindent
We notice below that the limits preserve homomorphisms.
\begin{lemma}\label{limithom}
Assume $\mathcal{S}$ is a complete Hopf algebra and $(f_m:\mathcal{H}\to \mathcal{S})_m$ is a strongly Cauchy sequence of complete Hopf algebra homomorphisms. Then $\li(f_m)_m$ is a complete Hopf algebra homomorphism.
\end{lemma}
\begin{proof}
It is enough to notice that each $f_m:\mathcal{H}[m]\to \mathcal{S}[m]$ is a (complete) Hopf algebra homomorphism (see Proposition \ref{frker}).
\end{proof}
\noindent
In the next two lemmas, we will denote by $+$ the group operation on the set of all $C$-algebra homomorphisms from a complete Hopf algebra over $C$ to a complete $C$-algebra.
\begin{lemma}\label{limitadd}
Let $(\mathcal{F}_m)$ and $(\mathcal{T}_m)$ be strongly Cauchy sequences of maps from $\mathcal{H}$ to $\mathcal{S}$. Then we have
$$\li(\mathcal{F}_m+\mathcal{T}_m)=\li(\mathcal{F}_m)+\li(\mathcal{T}_m).$$
\end{lemma}
\begin{proof}
For each $m\in \Nn$ we have (using Proposition \ref{frker} for the second equality):
\begin{IEEEeqnarray*}{rCl}
\li(\mathcal{F}_m+\mathcal{G}_m)[m] & = & (\mathcal{F}_m+\mathcal{G}_m)[m]  \\
& = & \mathcal{F}_m[m]+\mathcal{G}_m[m] \\
& = & \li(\mathcal{F}_m)[m]+\li(\mathcal{G}_m)[m] \\
& = &  (\li(\mathcal{F}_m)+\li(\mathcal{G}_m))[m].
\end{IEEEeqnarray*}
By passing to the inverse limit we get the result.
\end{proof}
\begin{lemma}\label{frfactor}
Let $\mathcal{F},\mathcal{T}: \mathcal{H}\to \mathcal{S}$ be local $C$-homomorphisms. Assume that for some $m\geqslant 1$, $(\mathcal{F}-\mathcal{T})(\mathcal{H})\subseteq \mathcal{S}^{p^{m}}$. Then we have $$\mathcal{F}[m]=\mathcal{T}[m]:\mathcal{H}[m]\to \mathcal{S}[m].$$
\end{lemma}
\begin{proof}
Since the ideal $\mathfrak{m}_{\mathcal{S}}$ is radical (and $p^m\geqslant m$), we have $\mathfrak{m}_{\mathcal{S}}\cap \mathcal{S}^{p^{m}}\subseteq \mathfrak{m}_{\mathcal{S}}^{m}$. Therefore
$$(\mathcal{F}-\mathcal{T})[m](\mathcal{H}[m])=C.$$
It means that $(\mathcal{F}-\mathcal{T})[m]$ is the $0$-map. By Proposition \ref{frker}, we have
$$0=(\mathcal{F}-\mathcal{T})[m]=\mathcal{F}[m]-\mathcal{T}[m],$$
hence $\mathcal{F}[m]=\mathcal{T}[m]$.
\end{proof}
Let us fix now a commutative algebraic group $A$ over $C$ and we set $T:=\mathcal{O}_{A,0}$.
We also fix a $C$-scheme $V$, $v\in V(C)$ and denote by $\widehat{V}$ the formalization of $V$ along $v$. The next definition is modeled on the case of additive power series, which can be considered as limits (in a certain strong sense) of additive polynomials, see \cite[Def. 2.2]{K6}.
\begin{definition}
\begin{enumerate}
\item A sequence of local $C$-algebra homomorphisms $(\varphi_m:\mathcal{H}\to \mathcal{S})_{m\in \Nn}$ is called \emph{$\mathcal{H}$-compatible}, if for each $m$, we have
$$(\varphi_{m+1}-\varphi_{m})(\mathcal{H})\in \mathcal{S}^{p^{m+1}},$$
where ``$-$'' comes from the complete Hopf algebra structure.

\item A sequence of local $C$-algebra homomorphisms $(f_m:T\to R)_{m\in \Nn}$ is called \emph{$A$-compatible}, if the sequence $(\widehat{f_m}:\widehat{T}\to \widehat{R})_{m\in \Nn}$ is $\widehat{T}$-compatible.

\end{enumerate}
\end{definition}
\begin{remark}\label{acomprem}
For a sequence of local $C$-algebra homomorphisms $(f_m:T\to R)_{m\in \Nn}$ the following are equivalent:
\begin{enumerate}
\item The sequence $(f_m:T\to R)_{m\in \Nn}$ is $A$-compatible.

\item For each $m$, we have $(f_{m+1})_A-(f_{m})_A\in A(R^{p^{m+1}})$.

\item For each $m$, the morphism $(f_{m+1})_A-(f_{m})_A:\spec(R)\to A$ factors through $\fr^{m+1}_A:A^{\fr^{-m-1}}\to A$ (similarly as in Remark \ref{facfroring}).
\end{enumerate}
\noindent
Using Remark \ref{acomprem}(3), we can extend the definition of an $A$-compatible sequence to any sequence of $C$-scheme morphisms $(f_m:V\to A)_{m}$.
\end{remark}
\begin{lemma}\label{compconv}
An $A$-compatible sequence $(f_m:T\to R)_m$ is strongly Cauchy.
\end{lemma}
\begin{proof}
Let us fix $m\in \Nn$.
By the definition of an $A$-compatible sequence we have
$$(\widehat{f_{m+1}}-\widehat{f_{m}})(\widehat{T})\subseteq  \widehat{R}^{p^{m+1}}.$$
By Lemma \ref{frfactor}, we have
$$f_m[m+1]=\widehat{f_m}[m+1]=\widehat{f_{m+1}}[m+1]=f_{m+1}[m+1],$$
so the sequence is strongly Cauchy.
\end{proof}
\noindent
\begin{example}
The notion of an $A$-compatible sequence is much stronger than the notion of a strongly Cauchy sequence. For example consider $V=\Aa^1$ and $A=\ga$. Let $(f_m:V\to A)_m$ be a sequence of morphisms. We can understand each $f_m$ as an element of $C[X]$. Then we have:
\begin{itemize}
\item the sequence $(f_m)_m$ is strongly Cauchy if and only if for each $m$, $X^{p^m}$ divides $f_{m+1}-f_m$;

\item the sequence $(f_m)_m$ is $\ga$-compatible if and only if for each $m$, $f_{m+1}-f_m\in C[X^{p^m}]$.
\end{itemize}
\end{example}
\noindent
Below we define a certain class of formal maps, which we find most convenient to work with.
\begin{definition}\label{defalimit}
Let $\mathcal{F}:\widehat{V}\to \widehat{A}$ be a formal map. Then $\mathcal{F}$ is an \emph{$A$-limit}, if  there is an $A$-compatible sequence $f_m:T\to R$ such that $\mathcal{F}$ corresponds to $\li(f_{m})$.
\end{definition}

\begin{remark}\label{defalimit2}
If for any $m\in \Nn$ there is $\phi_m:V\to A$ such that $\mathcal{F}-\widehat{\phi_m}:\widehat{V}\to \widehat{A}$ factors through $\fr^{m+1}_{\widehat{V}}$ as in (the formal version of) Remark \ref{facfrosch}, then $\mathcal{F}$ is an $A$-limit map.
\end{remark}
\begin{example}\label{alimitexample}
\begin{enumerate}

\item For $A=\mathbb{G}_{\rm{a}}$, an $A$-limit is of the form
$$\psi_0+\psi_1^p+\ldots+\psi_m^{p^m}+\ldots$$
where $(\psi_m:V\to \ga)_m$ is a sequence of morphisms.

\item For $A=\mathbb{G}_{\rm{m}}$ an $A$-limit is of the form
$$\psi_0\cdot\psi_1^p\cdot\ldots\cdot\psi_m^{p^m}\cdot \ldots$$
where $(\psi_m:V\to \gm)_m$ is a sequence of morphisms.


\end{enumerate}
\end{example}
\noindent
Recall from Proposition \ref{localforms} that we can consider $\Omega^{\inv}_{A}$ as a $C$-subspace of $\Omega_{\widehat{T}}$. The following result is item $(2)$ from the beginning of this section.
\begin{prop}\label{limitforms}
Let $(f_m:T\to R)_m$ be an $A$-compatible sequence and $\mathcal{F}=\li (f_m)$. Then for any $k\in \Nn$, $\mathcal{F}_*$ coincides with $(f_k)_*$ on $\Omega^{\inv}_{A}$.
\end{prop}
\begin{proof}
Let us fix $k\in \Nn$. By Lemma \ref{limitadd}, we have
$$\li_m(\widehat{f_m}-\widehat{f}_k)=\li_m(\widehat{f_m})-\li_m(\widehat{f}_k)=\mathcal{F}-\widehat{f}_k.$$
As in the proof of Lemma \ref{compconv}, for each $m\in \Nn$ we have
$$(\widehat{f_m}-\widehat{f}_k)(\widehat{T})\subseteq \widehat{R}^p.$$
By Lemma \ref{limitpower}, we get
$$(\mathcal{F}-\widehat{f}_k)(\widehat{T})\subseteq \widehat{R}^p.$$
Therefore $(\mathcal{F}-\widehat{f}_k)_*=0$. By Lemma \ref{coincide} and Lemma \ref{points}(2), we get
$$0=(\mathcal{F}-\widehat{f}_k)_*=(\mathcal{F}-\widehat{f}_k)_A^*=(\mathcal{F}_A-(\widehat{f}_k)_A)^*.$$
Take $\omega\in \Omega^{\inv}_A$. By Proposition \ref{invforms}(2) and Lemma \ref{coincide} again, we get
$$0=(\mathcal{F}_A-(\widehat{f}_k)_A)^*(\omega)=\mathcal{F}_A^*(\omega)-(\widehat{f}_k)_A^*(\omega)
=\mathcal{F}_*(\omega)-(\widehat{f}_k)_*(\omega)=\mathcal{F}_*(\omega)-(f_k)_*(\omega).$$
Hence we get $\mathcal{F}_*(\omega)=(f_k)_*(\omega)$.
\end{proof}
\begin{remark}
By Proposition \ref{limitforms} and Proposition \ref{invforms}(2), we can consider $\mathcal{F}_*$ as a map from $\Omega_T$ to $\Omega_R$. This is the crucial property of \emph{special maps}, see Definition \ref{defspecial}.
\end{remark}
\noindent
We prove below the crucial condition about vanishing of an $A$-limit map.
\begin{prop}\label{limitsobs}
Let $(f_m:T\to R)_m$ be an $A$-compatible sequence and $\mathcal{F}=\li (f_m)$. Then the map $\mathcal{F}$ is the $0$-map if and only if for all $k\in \Nn$ we have $f_k(T)\subseteq R^{p^{k+1}}$.
\end{prop}
\begin{proof}
For the proof of the right-to-left implication, by Lemma \ref{limitpower} for each $k\in \Nn$, we have $\mathcal{F}(\widehat{T})\subseteq \widehat{R}^{p^{k+1}}$. By Krull's intersection theorem \cite[Theorem 8.10]{mat}, we get
$$\mathcal{F}(\mathfrak{m}_{\widehat{T}})\subseteq \bigcap_{k=1}^{\infty}(\mathfrak{m}_{\widehat{R}})^{p^k}=\{0\},$$
so $\mathcal{F}$ is the $0$-map.
\\
Assume now that $\mathcal{F}$ is the $0$-map and take $k\in \Nn$. As in the proof of Proposition \ref{limitforms} we have: $$\widehat{f}_k=\widehat{f}_k-\mathcal{F}=\li(\widehat{f}_k-\widehat{f_{k+m}})_m.$$
By the definition of an $A$-compatible sequence, for each $m$ we have $(\widehat{f}_k-\widehat{f_{k+m}})(T)\subseteq R^{p^{k+1}}$, so by Lemma \ref{limitpower}, we have $\widehat{f}_k(\widehat{T})\subseteq \widehat{R}^{p^{k+1}}$. Finally by Proposition \ref{completefrob}(4), we get $f_k(T)\subseteq R^{p^{k+1}}$.
\end{proof}
\begin{remark}\label{limitremark}
The right-to-left implication in Proposition \ref{limitsobs} holds even for sequences of maps which are point-wise Cauchy. However, the left-to-right implication holds only for $A$-compatible sequences and this implication is crucial in the proof of Proposition \ref{biglemma}.
\end{remark}

\begin{lemma}\label{addalimit}
Let $\mathcal{F},\mathcal{G}:\widehat{V}\to \widehat{A}$ be $A$-limit maps. Then the map $\mathcal{F}+\mathcal{G}:\widehat{V}\to \widehat{A}$ is an $A$-limit map.
\end{lemma}
\begin{proof}
It is enough to use Lemma \ref{limitadd}.
\end{proof}
\noindent
Below is the last main result in this section (recall the identifications in Proposition \ref{invforms}).
\begin{prop}\label{frobfactorbetter}
Assume that $R=\mathcal{O}_{V,v}$ is $p$-normal. Let $\mathcal{F}:\widehat{V}\to \widehat{A}$ be an $A$-limit map such that the following composition
\begin{equation*}
  \xymatrix{
\Omega^{\inv}_A \ar[rr]^{\mathcal{F}_*}& & \Omega_R \ar[rr]^{} & &  \Omega_K}
\end{equation*}
is the $0$-map. Then there is an $A^{\fr^{-1}}$-limit map $\mathcal{F}^{(1)}:\widehat{V}\to \widehat{A}^{\fr^{-1}}$ making the following diagram commutative
\begin{equation*}
  \xymatrix{
 & & \widehat{A}^{\fr^{-1}} \ar[d]^{\fr}\\
\widehat{V} \ar[rru]^{\mathcal{F}^{(1)}}\ar[rr]^{\mathcal{F}} & & \widehat{A}.  }
\end{equation*}
\end{prop}
\begin{proof}
Let $(f_m:T\to R)_m$ be an $A$-compatible sequence such that $\mathcal{F}=\li (f_m)$ (see Remark \ref{defalimit2}(1)). Let us fix $m\in \Nn$. By Proposition \ref{limitforms}, we have $$\mathcal{F}_*=(f_k)_*:\Omega^{\inv}_A\to \Omega_R.$$
Let $f_m':T\to K$ denote the composition of $f_k$ with the inclusion $R\to K$. By Proposition \ref{invforms}(2), the map $(f'_k)_*:\Omega_T\to \Omega_K$ is the 0-map. Since $R$ is $p$-normal, by Lemma \ref{formppower}, we get $f_k(T)\subseteq R^p$. By Lemma \ref{limitpower} (for $k=1$), we have
$\mathcal{F}(\widehat{T})\subseteq \widehat{R}^p$. Clearly, $\mathcal{F}$, considered as a map from $\widehat{T}$ to $\widehat{R}^p$, is an $A$-limit which is witnessed by the sequence $(f_k:T\to R^p)_k$.
\\
By Fact \ref{facfroring}, $f_k$ factors through $f_k^{(1)}:T^{\fr^{-1}}\to R$ and similarly $\mathcal{F}$ factors through $\mathcal{F}^{(1)}:\widehat{T}^{\fr^{-1}}\to \widehat{R}$. Then the sequence $(f_k^{(1)})_k$ witnesses that $\mathcal{F}^{(1)}$ is an $A^{\fr^{-1}}$-limit map.
\end{proof}

\noindent
The last result in this section is an easy diagram chase and we skip its proof.
\begin{lemma}\label{limitcomposing}
Assume $(\mathcal{F}_m:\mathcal{H}\to \mathcal{R})_m$ is a compatible sequence converging to $\mathcal{F}$, $\iota:\mathcal{A}\to \mathcal{H}$ is a complete Hopf algebra morphism and $\pi:\mathcal{R}\to \mathcal{S}$ is a local $C$-homomorphism between complete $C$-algebras. Then the sequence $\pi\circ \mathcal{F}_m\circ \iota:\mathcal{A}\to \mathcal{S}$ is compatible and converges to $\pi\circ \mathcal{F}\circ \iota$.
\end{lemma}

\section{The main theorem}\label{secpositive}
\noindent
In this section we prove the main theorem of this paper (Theorem \ref{mainthm}). First we set-up the algebraic data, then prove a strong bound on the dimension of a certain kernel (Proposition \ref{biglemma}) which refines the weak bound (Proposition \ref{rightdep}).
\subsection{Set-up}\label{setup}
In this subsection we fix the notation for the entire Section \ref{secpositive}. All the geometric objects are defined over $C$. We consider an algebraic variety $V$, a commutative algebraic group $A$, a formal subscheme $\mathcal{W}\subseteq \widehat{V}$ Zariski dense in $V$ and an $A$-limit formal map $\mathcal{F}:\widehat{V}\to \widehat{A}$ vanishing on $\mathcal{W}$. We fix below all the necessary algebraic data. It is a rather long list, so we divide in into three parts: a part ``related to $V$'', a part ``related to $A$'' and a part ``related to $\mathcal{W}$''.
\\
Related to $V$, we fix:
\begin{itemize}

\item a reduced, absolutely irreducible $C$-scheme $V$ of finite type;

\item $v\in V(C)$;

\item $R=\mathcal{O}_{V,v}$ (note that $R$ satisfies the assumptions from Section \ref{alimit});

\item  $\widehat{V}$, the formalization of $V$ at $v$ (i.e. the object dual to $\widehat{R}$);

\item $K$, the fraction field of $R$.
\end{itemize}
Related to $A$, we fix:
\begin{itemize}

\item a connected commutative algebraic group $A$ of dimension $n$;

\item $T=\mathcal{O}_{A,0}$;

\item  $\widehat{A}$, the formalization of $A$ at $0$.
\end{itemize}
Related to $\mathcal{W}$, we fix:
\begin{itemize}
\item A prime ideal $P$ in $\widehat{R}$ such that the map $R\to \widehat{R}/P$ is one-to-one;

\item An absolutely irreducible formal subscheme $\mathcal{W}\subseteq \widehat{V}$, which corresponds to $\mathcal{S}:=\widehat{R}/P$;

\item $L$, the fraction field of $\mathcal{S}$.

\end{itemize}
\begin{remark}\label{absirrrem}
By Proposition \ref{lpinfty}, $C$ is the ``field of constants'' of $\mathcal{W}$, i.e $L^{p^{\infty}}=C$, which is necessary for the proof of the main theorem. Note that the condition $L^{p^{\infty}}=C$ implies that $C^{\alg}\cap L=C$, so in particular $C^{\alg}\cap K=C$.
\end{remark}
Related to $\mathcal{F}$ we fix a local $C$-algebra map $\widehat{T}\to \widehat{R}$ (which we also denote by $\mathcal{F}$) which is the limit of an $A$-compatible sequence (see Definition \ref{defalimit}) such that the composition of  $\mathcal{F}$ with the map $\widehat{R}\to \mathcal{S}$ is the $0$-map.

Our main result (Theorem \ref{mainthm}) says that the reason of the above vanishing is that the image of $\mathcal{F}$ is contained in a formal subgroup of $\widehat{A}$ whose dimension is bounded by the codimension of $\mathcal{W}$ in $\widehat{V}$. Unfortunately, we will need to put further restrictions on $A$ to prove Theorem \ref{mainthm}.

\subsection{Linear dependence of forms}\label{dependence}
The main and only result of this subsection is a generalization of \cite[Prop. 2.5]{K6} from the case of a vector group to the case of an arbitrary commutative algebraic group. It is also a desired improvement of
Proposition \ref{rightdep}, which we are able to show only under the $A$-limit assumption. We use the notation from Section \ref{setup}.

Since $\mathcal{F}$ denotes either a ring homomorphism or a morphism of formal schemes, we write $\mathcal{F}_*$ when it acts covariantly on forms and $\mathcal{F}^*$ when it acts contravariantly on forms. Using Proposition \ref{localforms}, we consider $\mathcal{F}_*$ as a map from $\Omega_T$ to $\Omega_R$ and sometimes as a map from $\Omega^{\inv}_A$ to $\Omega_R$.
\begin{prop}\label{biglemma}
Let $\mathcal{F}^*_K:\Omega_A\to \Omega_K$ denote the composition of $\mathcal{F}^*:\Omega_A\to \Omega_R$ with the map $\Omega_R\to \Omega_K$. Then we have
$$\dim_C\mathcal{F}^*_K({\Omega}^{\inv}_{A})\leqslant \dim(V)-\dim(\mathcal{W}).$$
\end{prop}
\begin{proof}
Let $(f_m:T\to R)_m$ be an $A$-compatible sequence such that $\mathcal{F}=\li(f_m)$. We will identify $R$ with a subring of $\mathcal{S}$. Let us fix $m\in \Nn$ and let $f_{m,\mathcal{S}}:T\to \mathcal{S}$ be the composition of $f_m$ with the inclusion $R\subseteq \mathcal{S}$. By Lemma \ref{limitcomposing}, $(f_{m,\mathcal{S}})_m$ is a compatible sequence converging to the composition of $\mathcal{F}$ with the map $\widehat{R}\to \mathcal{S}$. By our assumptions, this last composition is the $0$-map. By Proposition \ref{limitsobs}, we have
$$f_{m,\mathcal{S}}(T)\subseteq {\mathcal{S}}^{p^{m+1}}.$$
Let $R_m$ denote $\mathcal{S}^{p^m}\cap R$. Then $f_m:T\to R$ factors through $f_m':T\to R_m$. Let $\gamma_m:\Omega_{R_m}\to \Omega_R$ denote the map induced by the inclusion $R_m\subseteq R$.
For any $\omega\in \Omega^{\inv}_A$, by Proposition \ref{limitforms} we have:
$${\mathcal{F}}_*(\omega)=(f_m)_*(\omega)=\gamma_m((f_m')_*(\omega)).$$
Let $\iota:R_{m+1}\to R_m$ denote the inclusion map. Since $(f_m)_m$ is a compatible sequence, we have $(f_{m+1})_A-(f_{m})_A\in A(R^{p^{m+1}})$.
We have a tower of finite extensions of rings
$$R^{p^{m+1}}\subseteq R_m^p\subseteq R_m\subseteq R.$$
Hence the corresponding morphism of affine schemes are epimorphisms (in the category of schemes) and we have
$$(\iota\circ f'_{m+1})_A-(f'_{m})_A\in A(R^p_m).$$
By Lemma \ref{newvanish} we get
$$\Omega_{R_m}\ni ((\iota\circ f'_{m+1})_A-(f'_{m})_A)^*(\omega)=0.$$
By Lemma \ref{coincide}(2) and Proposition \ref{invforms}(3), we get:
$$(f'_m)_*(\omega)=(f'_m)_A^*(\omega)=(\iota\circ f'_{m+1})_A^*(\omega)=\delta_m((f'_{m+1})_*(\omega)),$$
where $\delta_m:\Omega_{R_{m+1}}\to \Omega_{R_{m}}$ is the map induced by $\iota$.
\\
The following commutative diagram illustrates the situation:
\begin{equation*}
  \xymatrix{
 & & \Omega_R \ar[rr]^{ } & & \Omega_\mathcal{S} \\
\Omega^{\inv}_A\ar[urr]^{\mathcal{F}_*} \ar[rr]^{(f_m')_*} \ar[rrd]_{(f_{m+1}')_*}& & \Omega_{R_m} \ar[u]_{\gamma_m} \ar[rr]^{\beta_{m}}  & & \Omega_{\mathcal{S}^{p^m}}\ar[u]^{}\\
& & \Omega_{R_{m+1}} \ar[rr]^{\beta_{m+1}} \ar[u]_{\delta_m}& & \Omega_{\mathcal{S}^{p^{m+1}}}.\ar[u]^{0}}
\end{equation*}
Chasing this diagram, we get
$$(f_m')_*(\omega)\in  \ker(\beta_m).$$
Clearly, $\Omega_K$ embeds over $K$ into $\Omega_K\otimes_KL\cong \Omega_R\otimes_RL$, thus we will work inside $\Omega_R\otimes_RL$ from now on.
We have the following commutative diagram of $L^{p^m}$-linear maps (to ease the notation we do not put the tensor product symbol on the level of homomorphisms):
\begin{equation*}
 \xymatrix{\Omega_{R}\otimes_{R}L\ar[rr]^{\alpha}  & & \Omega_{L}\\
\Omega_{R_m}\otimes_{R_m}L^{p^m}\ar[rr]^{\beta_m} \ar[u]^{\gamma_m} & & \Omega_{L^{p^m}}\ar[u]^{}\\
\Omega_{R^{p^m}}\otimes_{R^{p^m}}L^{p^m}.\ar[u]^{} \ar[urr]_{\alpha_m}& &}
\end{equation*}
By Proposition \ref{rightdep}(1), $\alpha$ is onto. Applying the $m$-th power of the Frobenius map, we get that $\alpha_m$ is onto. Hence $\beta_m$ is onto as well. Let $c:=
\dim(R)-\dim(\mathcal{S})$. By Lemma \ref{lemma2}(2), we have $\dim_{L}\Omega_{L}=\dim(\mathcal{S})$. Applying the $m$-th power of the Frobenius map again, we see that
$$\dim_{L^{p^m}}\Omega_{L^{p^m}}=\dim_{L}\Omega_{L}.$$
Therefore, we obtain
\begin{equation}
\dim_{L^{p^m}}\Omega_{L^{p^m}}=\dim(\mathcal{S}).\tag{$*$}
\end{equation}
Let $K_m$ be the fraction field of $R_m$. We have a tower of fields $K^{p^m}\subseteq K_m\subseteq K$ such that
$$\td_CK^{p^m}=\td_CK=\dim(R),$$ so we get
\begin{equation}
\dim_{L^{p^m}}\Omega_{R_m}\otimes_{R_m}L^{p^m}=\dim(R).\tag{$**$}
\end{equation}
By $(*)$ and $(**)$, we finally obtain $\dim_{L^{p^m}}\ker(\beta_m)=c$ (since $\beta_m$ is onto).
\\
Therefore for any $m\in \Nn$ and any $c+1$ forms from $\Omega_A^{\inv}$, their images by $(f'_m)_*$ are dependent over $L^{p^m}$ in $\Omega_{R_m}\otimes_{R_m}L^{p^m}$. Since ${\mathcal{F}}^*=\gamma_m\circ (f'_m)_*$, we also get that the images of these forms by the map ${\mathcal{F}}^*$ are dependent over $L^{p^m}$ in $\Omega_{K}\otimes_{K}L=\Omega_{R}\otimes_{R}L$. Since $L^{p^{\infty}}=C$, these images are also dependent over $C$.
\end{proof}
\begin{remark}
In the characteristic $0$ case a much stronger result holds (Proposition \ref{char0dep}), where there is no algebraic group around and instead of invariant forms closed forms are considered.
\end{remark}

\subsection{Main theorem}\label{mainsection}
We can prove now our main theorem under some extra assumptions on the algebraic group $A$. As it is discussed in Section \ref{char0}, this assumptions are not restrictive at all in the case of characteristic 0, but are quite restrictive in the case of positive characteristic. Still the result below generalizes \cite[Prop. 3.1]{K6} from the case of $A=\ga^n$ to the case of $A=H^n$ for any 1-dimensional algebraic group $H$ defined over $\Ff_p$.
\begin{definition}\label{defintegr}
We call an algebraic group $A$ (over $C$) \emph{integrable}, if there is a one-dimensional algebraic group $H$ such that we have the following (in the case of $\ch(C)=0$, we drop the item $(3)$ below).
\begin{enumerate}
\item $\widehat{A}\cong \widehat{H^n}$.

\item The map $\mathrm{End}(\widehat{H})\to \mathrm{End}_C(\Omega^{\inv}_H)$($=C$) is onto.

\item $H$ is \emph{$\Ff_p$-isotrivial} i.e. $H\cong H^{\fr}$.
\end{enumerate}
\end{definition}
\begin{remark}
To make the item $(2)$ in the definition above meaningful, one needs to notice that a formal group homomorphism induces a map on invariant forms. It is well-known and follows e.g. from \cite[Corollary IV.4.3]{Si}. In a more general case of higher invariant forms, it follows from Proposition \ref{homspecial}.
\end{remark}

\begin{example}
The following algebraic groups are integrable:
\begin{itemize}
\item $\ga^n$ over any $C$ (so any commutative algebraic group, if $\ch(C)=0$);

\item Any $H^n$ for $C=\Ff_p$ and a one-dimensional algebraic group $H$.
\end{itemize}
\end{example}
\noindent
We can state and prove now our main theorem. Recall that we are still in the set-up from Section \ref{setup}, i.e. $\mathcal{F}:\widehat{V}\to \widehat{A}$ is a formal $A$-limit map vanishing on a Zariski dense formal subscheme $\mathcal{W}\subseteq \widehat{V}$. We will use in the proof the notation from Section \ref{setup}.
\begin{theorem}\label{mainthm}
Assume that $A$ is integrable. Then there is a formal subgroup $\mathcal{C}\leqslant \widehat{A}$ such that $\mathcal{F}(\widehat{V}) \subseteq \mathcal{C}$ and
$$\dim(\mathcal{C})\leqslant \dim(V)-\dim(\mathcal{W}).$$
\end{theorem}
\begin{proof}
Recall the notion of a $p$-normal domain and the notation $R'$ from Section \ref{pnormal}.
\\
\\
{\bf Claim} ($p$-normalization) We can assume that $R=R'$.
\begin{proof}[Proof of the Claim]
Without loss of generality $V$ is an affine variety over $C$. Let $O$ be the coordinate ring of $V$ and $M$ be the maximal ideal of $O$ corresponding to $v\in V(C)$. By Lemma \ref{pnormalization}, there is a maximal ideal $M'$ in $O'$ such that $M'\cap O=M$ and $O'/M'=C$. Therefore $M'$ corresponds to $v'\in V'(C)$ which is mapped to $v$ by the morphism $V'\to V$ corresponding to the ring extension $R\subseteq R'$. By Fact \ref{pnormallocal}(3), the ring $R'$ still satisfies the assumptions from Section \ref{alimit}.
\\
Let $R^{(1)}=O'_{M'}$ be the local ring of $V'$ at $v'$. (We are tempted to denote this local ring by $R'$, however it may be slightly bigger than the $p$-normalization of $R$.) By Fact \ref{pnormallocal}, $(R^{(1)})'=R^{(1)}$. We will show that we can replace $V$ with $V'$ and $v$ with $v'$. Let $\mathfrak{m}$ denote the maximal ideal of $R$ and $\mathfrak{m}'$ denote the maximal ideal of $R^{(1)}$. The extension $R\subseteq R^{(1)}$ is still integral and $\sqrt{\mathfrak{m}R^{(1)}}=\mathfrak{m}'$. Since $R^{(1)}$ is Noetherian, there is $N\in \Nn$ such that
$$(\mathfrak{m}')^N\subseteq \mathfrak{m}R^{(1)}\subseteq\mathfrak{m}'.$$
Therefore we have
$$\widehat{R^{(1)}}\cong \widehat{(R,\mathfrak{m}R^{(1)})}\cong \widehat{R}\otimes_RR^{(1)}.$$
Since $\widehat{R}$ is flat over $R$ \cite[Theorem 8.8]{mat}, the natural map $\widehat{R}\to \widehat{R^{(1)}}$ is an embedding, which we will regard as an inclusion. Since the extension $R\subseteq R'$ is integral, the extension $\widehat{R}\subseteq \widehat{R^{(1)}}$ is integral as well. Let $P$ be the kernel of the map $\widehat{R}\to \mathcal{S}$. Since $P$ is a prime ideal and the extension $\widehat{R}\subseteq \widehat{R^{(1)}}$ is integral, by \cite[Theorem 9.3]{mat} there is a prime ideal $P'$ in $\widehat{R^{(1)}}$ such that $P'\cap \widehat{R}=P$. Let $\mathcal{S}'=\widehat{R^{(1)}}/P'$ and $\mathcal{W}'$ be the corresponding formal subscheme of $\widehat{V'}$. Denote by $\varphi$ and $\varphi'$ the appropriate compositions of vertical arrows in the commutative diagram below.
\begin{equation*}
  \xymatrix{
R \ar[r]^{\subseteq}\ar[d]^{\subseteq} &  R^{(1)} \ar[d]^{\subseteq}\\
\widehat{R} \ar[r]^{\subseteq}\ar[d]^{} &  \widehat{R^{(1)}} \ar[d]^{}\\
\mathcal{S} \ar[r]^{\subseteq} &  \mathcal{S}'.}
\end{equation*}
Since $\mathcal{W}$ is Zariski dense in $V$, $\ker(\varphi)=0$. Therefore $\ker(\varphi')\cap R=0$. Since, the extension $R\subseteq R^{(1)}$ is integral, $\ker(\varphi')=0$, so $\mathcal{W}'$ is Zariski dense in $V'$. It is easy to see (since the extension $\widehat{R}\subseteq \widehat{R^{(1)}}$ is integral, $\mathcal{S}=\widehat{R}/P$ and $\mathcal{S}'=\widehat{R^{(1)}}/P'$) that the extension $\mathcal{S}\subseteq \mathcal{S}'$ is integral as well, thus $\dim(\mathcal{S})=\dim(\mathcal{S}')$ (see \cite[Exercise 9.2]{mat}).
\\
Let $\mathcal{F}'$ denote the composition of $\mathcal{F}$ with the morphism $\widehat{V'}\to \widehat{V}$. Since the map $\widehat{R}\to \widehat{R^{(1)}}$ is an embedding, any formal subgroup $\mathcal{C}$ of $\widehat{A}$ working for $v',V',\mathcal{W}',\mathcal{F}'$ also works for $v,V,\mathcal{W},\mathcal{F}$, hence the claim is proved.
\end{proof}
\noindent
By the Claim, we can assume that $R$ is $p$-normal. Let us take $H$ from Definition \ref{defintegr} (we have assumed that $A$ is integrable). For any formal map $\mathcal{T}:\widehat{V}\to \widehat{A}$, let $\mathcal{T}_i:\widehat{V}\to \widehat{H}$ be the composition of $\mathcal{T}$ with the $i$-th coordinate morphism $ \widehat{A}\to  \widehat{H}$. Let $\omega\in \Omega_H^{\inv}\setminus \{0\}$ and $T_H:=\mathcal{O}_{H,0}$. Let $c=\dim(R)-\dim(\mathcal{S})$. We can assume that $c<n$. By Proposition \ref{biglemma}, without loss of generality there is $0\leqslant r_0\leqslant c$ such that $\mathcal{F}^*_1(\omega),\ldots,\mathcal{F}^*_{r_0}(\omega)$ are linearly independent over $C$ in $\Omega_K$, and for each $r_0<l\leqslant n$, there are $c_{l,1},\ldots,c_{l,r_0}\in C$ such that inside $\Omega_K$ we have
$$\mathcal{F}^*_l(\omega)=\sum_{i=1}^{r_0}c_{l,i}\mathcal{F}^*_i(\omega).$$
For each $c_{l,i}$ as above, by the integrability of $H$, there is $\gamma_{l,i}:\widehat{H}\to \widehat{H}$ such that $\gamma_{l,i}^*=c_{l,i}$. We define
$$(\mathcal{F}_l)_{(1)}=\mathcal{F}_l-\sum_{i=1}^{r_0}\gamma_{l,i}\circ \mathcal{F}_i.$$
By Proposition \ref{foradd}, we get
\begin{IEEEeqnarray*}{rCl}
(\mathcal{F}_l)_{(1)}^*(\omega) & = & \mathcal{F}_l^*(\omega)-\sum_{i=1}^{r_0}(\gamma_{l,i}\circ \mathcal{F}_i)^*(\omega)  \\
& = & \mathcal{F}^*_l(\omega)-\sum_{i=1}^{r_0}c_{l,i}\mathcal{F}^*_i(\omega) \\
& = & 0.
\end{IEEEeqnarray*}
By Proposition \ref{invforms}(2), $(\mathcal{F}_l)_{(1)}^*:\Omega_{T_H}\to \Omega_K$ is the constant $0$-map. Since $R$ is $p$-normal (see Claim), we can use Proposition \ref{frobfactorbetter} and the fact that $H$ is $\Ff_p$-isotrivial to obtain an $H$-limit map $\mathcal{F}_l^{(1)}:\widehat{V}\to \widehat{H}$ such that the following diagram is commutative
\begin{equation*}
  \xymatrix{
 &  & & \widehat{H}\ar[d]^{\fr}\\
\widehat{V} \ar[rrr]^{(\mathcal{F}_l)_{(1)}} \ar[rrru]^{(\mathcal{F}_l)^{(1)}}&  & &\widehat{H}  .}
\end{equation*}
For $1\leqslant l\leqslant r_0$, let us define $\mathcal{F}_l^{(1)}$ as $\mathcal{F}_l$ and we define an $A$-limit formal map
$$\mathcal{F}^{(1)}:\widehat{V}\to \widehat{A},\ \ \ \mathcal{F}^{(1)}=(\mathcal{F}_{1}^{(1)},\ldots,\mathcal{F}_n^{(1)}).$$
There are algebraic endomorphisms $\phi,\iota:A\to A$ such that:
\begin{itemize}
\item $\phi$ is an automorphism;

\item $\iota$ is either the Frobenius map or the identity map on each coordinate;

\item $\widehat{\iota} \circ \mathcal{F}^{(1)}=\widehat{\phi}\circ \mathcal{F}$.
\end{itemize}
\noindent
Hence $\mathcal{F}^{(1)}$ vanishes on $\mathcal{W}$. Moreover, obtaining a ``right $\mathcal{C}$ for $\mathcal{F}^{(1)}$'' gives a ``right $\mathcal{C}$ for $\mathcal{F}$''. Thus we can replace $\mathcal{F}$ with $\mathcal{F}^{(1)}$. (Note that for $r_0=0$, we just get $\fr\circ \mathcal{F}^{(1)}=\mathcal{F}$.)
\\
Applying Proposition \ref{biglemma} to $\mathcal{F}^{(1)}$ we get (again without loss of generality) that there is $r_0\leqslant r_1\leqslant c$ ($r_0\leqslant r_1$, since the maps $\mathcal{F}_1,\ldots,\mathcal{F}_{r_0}$ have not changed) such that $(\mathcal{F}^{(1)}_1)^*(\omega),\ldots,(\mathcal{F}^{(1)}_{r_1})^*(\omega)$ are linearly independent over $C$ in $\Omega_K$, and for each $r_1<l\leqslant n$, there are $d_{l,1},\ldots,d_{l,r_0}\in C$ such that inside $\Omega_K$ we have
$$(\mathcal{F}_l^{(1)})^*(\omega)=\sum_{i=1}^{r_0}c_{l,i}(\mathcal{F}_i^{(1)})^*(\omega).$$
As before, there are $\delta_{l,i}$, endomorphisms of $H$ such that if we define
$$(\mathcal{F}_l)_{(2)}=\mathcal{F}-\sum_{i=1}^{r_0}\widehat{\delta_{l,i}}\circ \mathcal{F}_i^{(1)},$$
then there is a formal $H$-limit map $\mathcal{F}_l^{(2)}:\widehat{V}\to \widehat{H}$ such that the following diagram is commutative:
\begin{equation*}
  \xymatrix{
 &  & & \widehat{H} \ar[d]^{\fr}\\
\widehat{V} \ar[rrr]_{(\mathcal{F}_l)_{(2)}} \ar[rrru]^{(\mathcal{F}_l)^{(2)}}&  & &\widehat{H}  .}
\end{equation*}
Continuing as above, we get a sequence $0\leqslant r_0\leqslant r_1\leqslant r_2\leqslant \ldots \leqslant c$. Let $m\in \Nn$ be such that for each $j\geqslant m$ we have $r_m=r_j=:r$. We can replace $\mathcal{F}$ with $\mathcal{F}^{(r)}$ and assume that $r=r_0$.
\\
\\
If we continue the construction above, for each $t\in \Nn$ and $r<l\leqslant n$, we get an endomorphism $\gamma_{l,i,t}$ of $\widehat{H}$ such that there is a formal map $\mathcal{F}_l^{(t+1)}:\widehat{V}\to \widehat{H}$ making the following diagram commutative.
\begin{equation*}
  \xymatrix{
 &    & &  & \widehat{H}\ar[d]^{\fr}\\
\widehat{V} \ar[rrrru]^{\mathcal{F}_l^{(t+1)}} \ar[rrrr]_{\mathcal{F}_l^{(t)}-\mathcal{R}_l^{(t)}} &   & &  & \widehat{H} ,}
\end{equation*}
where $\mathcal{R}_l^{(t)}=\sum_{i=1}^{r}\gamma_{l,i,t}\circ \mathcal{F}_i^{(t)}$ and $\mathcal{F}_l^{(0)}=\mathcal{F}_l$.
\\
Therefore for any $t$ and $l$ as above, the following diagram is commutative:
\begin{equation*}
  \xymatrix{
&    & &   &  & \widehat{H} \ar[d]^{\fr^{t+1}}\\
\widehat{V} \ar[rrrrru]^{\mathcal{F}_l^{(t+1)}}\ar[rrrrr]_{\mathcal{F}_l-\mathcal{R}^{(0)}_l-\fr\circ\mathcal{R}_l^{(1)}-\ldots-\fr^t\circ \mathcal{R}_l^{(t)}} &   & &   &  & \widehat{H} .}
\end{equation*}
\noindent
For any $t$ as above let us define the following formal homomorphism:
$$\varphi_t:\widehat{A}\to \widehat{H}^{n-r},\ \ \ (\pi_{r+1}-\sum_{i=1}^{r}{\gamma_{r+1,i,t}}\circ \pi_i,\ldots,\pi_{n}-\sum_{i=1}^{r}{\gamma_{n,i,t}}\circ \pi_i),$$
where each $\pi_i:\widehat{A}\to \widehat{H}$ is the appropriate projection morphism.
\\
We finally define:
$$\Psi_t:\widehat{A}\to \widehat{H}^{n-r},\ \ \Psi_t=\varphi_0-\fr\circ \varphi_1-\ldots-\fr^t\circ \varphi_t.$$
Then $(\widehat{\Psi_t}:\widehat{A}\to \widehat{H^{n-r}})_t$ is a compatible sequence of formal group maps. Let ${\Psi}:=\li(\Psi_t)_t$. By Lemma \ref{limithom}, ${\Psi}$ is a formal group map as well. By Lemma \ref{limitcomposing}, we have
$${\Psi}\circ \mathcal{F}=\li(\Psi_t\circ \mathcal{F})_t.$$
By the construction, for each $t\in \Nn$, the formal map $\Psi_t\circ \mathcal{F}:\widehat{V}\to \widehat{H}^{n-r}$ factors through $\fr^{t+1}:\widehat{H}^{n-r}\to \widehat{H}^{n-r}$. By Proposition \ref{limitsobs}, we get that ${\Psi}\circ \mathcal{F}=0$. Hence if we take $\mathcal{C}$ as $\ker({\Psi})$ (for the existence of kernels in the category of commutative formal groups, see \cite[Proposition 1.3]{manin}), then $\mathcal{F}(\widehat{V})\subseteq \mathcal{C}$.
\\
It remains to check the codimension condition. Let $\alpha:\widehat{H}^{n-r}\to \widehat{H}^{n}$ be the inclusion map on the last $n-r$ coordinates. Since for each $t$, the map $\Psi_t\circ \alpha$ is the identity map, the map $\Psi\circ \alpha$ is the identity map as well. In particular, $\Psi$ is an epimorphism and we get
$$\dim(\mathcal{C})=n-(n-r)\leqslant c,$$
so the result follows.
\end{proof}

\subsection{The case of characteristic $0$}\label{char0}
\noindent%
In this subsection we drop our assumption on the characteristic of $C$. We keep the set-up from Section \ref{setup} with the following two changes:
\begin{itemize}
\item the $A$-limit assumption on $\mathcal{F}$ is dropped;


\item we denote by $C'$ the relative algebraic closure of $C$ in $K$ (note that the assumption $L^{p^{\infty}}=C$ in Section \ref{setup} implies $C'=C$).

\end{itemize}
\noindent
We will prove a very strong characteristic $0$ improvement of the weak bound on kernel (Proposition \ref{rightdep}(2)). It will easily imply (see Proposition \ref{bigspecial}) a characteristic $0$ version of the strong bound on kernel (Proposition \ref{biglemma}). The idea of the proof comes from Ax's proof of \cite[Theorem 1]{ax72}. Let $\Omega^{\closed}_K$ denote the $C$-subspace of $\Omega_K$ consisting of closed differential forms. Below, we identify $\Omega_K$ with a subspace of $\Omega_R\otimes_RL=\Omega_K\otimes_KL$.
\begin{prop}\label{char0dep}
If $\ch(C)=0$, then we have:
$$\dim_{C'}(\Omega^{\closed}_K\cap \ker(\Omega_R\otimes_RL\to \widehat{\Omega}_{\mathcal{S}}\otimes_{\mathcal{S}}L))\leqslant \dim(R)-\dim(\mathcal{S}).$$
\end{prop}
\begin{proof}
For any derivation $\partial:\mathcal{S}\to \mathcal{S}$ let $\partial_L$ denote its extension to $L$. Since $\mathcal{S}$ is complete, $\partial^*:\Omega_{\mathcal{S}}\to \mathcal{S}$ factors through $\widehat{\partial^*}:\widehat{\Omega}_{\mathcal{S}}\to \mathcal{S}$ and we have a commutative diagram:
\begin{equation*}
  \xymatrix{
   & &\widehat{\Omega}_{\mathcal{S}}\otimes_\mathcal{S}L \ar[rrd]^{\widehat{\partial^*}\otimes \id_L} & & \\
\Omega_R\otimes_RL\ \ar[rr]^{\gamma} \ar[rru]^{\beta}& &  \Omega_{L}=\Omega_{\mathcal{S}}\otimes_\mathcal{S}L \ar[u]^{} \ar[rr]^{\ \ \ \partial^*_L}& &L\ .}
\end{equation*}
\noindent%
After identifying $\Omega_R\otimes_RL$ with $\Omega_K\otimes_KL$, we have $\gamma:\Omega_K\otimes_KL\to \Omega_L$. Since $\ch(C)=0$, $\gamma$ is an embedding.
\\
Let $r:=\dim(R), s:=\dim(\mathcal{S}), n:=r-s+1$ and
$$\omega_1,\ldots,\omega_n\in \Omega^{\closed}_K\cap \ker(\Omega_R\otimes_RL\to \widehat{\Omega}_{\mathcal{S}}\otimes_{\mathcal{S}}L).$$
By Remark \ref{char0onto}, $\dim_L(\ker(\beta))=r-s$ and $\omega_1,\ldots,\omega_n$ are $L$-dependent. Let $\xi_i:=\gamma(\omega_i)$. Since $\gamma$ is a $C'$-linear embedding, it is enough to show that $\xi_1,\ldots,\xi_n$ are $C'$-dependent.
\\
By Proposition \ref{completestr}, $\mathcal{S}$ is a finite extension of the ring of power series in $s$ variables.
Hence, there are derivations $\partial_1,\ldots,\partial_s$ on $L$ (extending the standard partial derivations on the field of Laurent series) such that their common constant field coincides with $C'$.
By the diagram above, we have $\partial_i^*(\xi_j)=0$ for each $i,j$. Since each $\xi_j$ is closed, we can use the Lie derivative trick as in \cite[page 1198]{ax72}, to conclude that $\xi_1,\ldots,\xi_n$ are $C'$-dependent. 

We briefly recall the Lie derivative argument below. For any $C$-derivation $\partial$ on $L$, one defines the \emph{Lie derivative} $L_{\partial}$ on the module $\Omega_L$ by the following formula:
$$L_{\partial}:\Omega_L\to \Omega_L,\ \ \ L_{\partial}(a\ddd b):=\partial(a)\ddd b + a\ddd(\partial(b)).$$
If $\omega\in \Omega_L$ is a closed form such that $\partial^*(\omega)=0$, then $L_{\partial}(\omega)=0$ (see the lines 4 and 5 on page 1199 of \cite{ax72}). In particular, we get $L_{\partial_i}(\xi_j)=0$ for each $i,j$. Therefore, if we apply the Lie derivatives $L_{\partial_1},\ldots,L_{\partial_s}$ to a (minimal) $L$-linear combination of $\xi_1,\ldots,\xi_n$ witnessing their $L$-dependence, then we get that $\xi_1,\ldots,\xi_n$ are linearly dependent over the common constant field of $\partial_1,\ldots,\partial_s$, which is exactly $C'$.
\end{proof}
\begin{remark}\label{remque}
It is not clear whether the result above holds for $C$ of positive characteristic. The proof breaks at ``Since $\ch(C)=0$, $\gamma$ is an embedding.''. We have replaced Proposition \ref{char0dep} with Proposition \ref{biglemma} to handle the positive characteristic case.
\end{remark}
\noindent
We define below a class of formal maps for which we can state in the characteristic $0$ case, a very general version of Theorem \ref{mainthm}. This version also includes Ax's theorem (\cite[Theorem 1F]{ax72}) which will be discussed in Section \ref{applications}. This definition makes sense in arbitrary characteristic, however in the case of positive characteristic it needs to be corrected to include higher differential forms (see Section \ref{secspecial}).
\begin{definition}\label{defspecial}
We consider $\Omega_R$ as an $R$-submodule of $\widehat{\Omega}_R\cong \widehat{\Omega}_{\widehat{R}}$ (see Lemma \ref{lemma1}(1)). A formal map $\mathcal{F}:\widehat{V}\to A$ is \emph{special}, if
$$\mathcal{F}^*(\Omega^{\inv}_A)\subseteq \Omega_R.$$
\end{definition}
\begin{prop}\label{bigspecial}
Assume that $\mathcal{F}$ is special and let $\mathcal{F}^*_K$ denote the composition of $\mathcal{F}^*$ with the map $\Omega_R\to \Omega_K$. We have
$$\dim_{C'}\mathcal{F}^*_K({\Omega}^{\inv}_{A})\leqslant \dim(V)-\dim(\mathcal{W}).$$
\end{prop}
\begin{proof}
Let $r,s,n$ be as in the proof of Proposition \ref{char0dep} and take $\eta_1,\ldots,\eta_n\in \Omega^{\inv}_A$. For each $i$, let $\omega_i:=\mathcal{F}^*_K(\eta_i)$. Since $\eta_1,\ldots,\eta_n$ are closed forms, we get by our assumptions that
$$\omega_1,\ldots,\omega_n\in \Omega^{\closed}_K\cap \ker(\Omega_R\otimes_RL\to \widehat{\Omega}_{\mathcal{S}}\otimes_{\mathcal{S}}L).$$
By Proposition \ref{char0dep}, $\omega_1,\ldots,\omega_n$ are $C'$-dependent.
\end{proof}
\begin{example}\label{exspecial}
The composition of special formal maps is special. The following classes of formal maps are special.
\begin{enumerate}
\item Formalizations of algebraic maps (clear).

\item Formal group homomorphisms (see \cite[Corollary IV.4.3]{Si}).

\end{enumerate}
\end{example}
\begin{theorem}\label{shortchar0}
Assume that $C$ has characteristic $0$ and that $\mathcal{F}$ is special. Then there is a formal subgroup $\mathcal{C}\leqslant \widehat{A}$  such that $\mathcal{F}(\widehat{V})\subseteq \mathcal{C}$ and
$$\dim(\mathcal{C})\leqslant \dim(V)-\dim(\mathcal{W}).$$
\end{theorem}
\begin{proof}
The proof is similar to the first step of the proof of Theorem \ref{mainthm} (and to the proof of \cite[Theorem 1]{ax72}), so we will be brief. Since $\ch(C)=0$, we get $\widehat{A}\cong \widehat{\ga}^n$ and we can assume $A=\ga^n$. For $i\in \{1,\ldots,n\}$, let $\mathcal{F}_i:\widehat{V}\to \widehat{\ga}$ be the composition of $\mathcal{F}$ with the appropriate projection.
\\
As in the proof of Theorem \ref{mainthm} (using now Proposition \ref{bigspecial}), we get
$$0\leqslant r_0\leqslant \dim(V)-\dim(\mathcal{W})<n$$
and $c_{l,i}\in C'$ such that for each $l\in \{r+1,\ldots,n\}$ the formal map
$$f_l:=\mathcal{F}\times_CC'-\sum_{i=1}^{r_0}c_{l,i}(\mathcal{F}_i\times_CC'):\widehat{V}\times_CC'\to \widehat{A}\times_CC'$$
induces the $0$-map from $\Omega^{\inv}_{\ga}$ to $\Omega_K$ (note that $c_{l,i}=\gamma_{l,i}$ for $H=\ga$). Since $C'$ is relatively algebraically closed in $K$, we get that each $f_l$ is the $0$-map. Thus $\mathcal{F}(\widehat{V})\subseteq \widehat{D}$, where $D$ is an algebraic subgroup of $\ga^n$ of dimension smaller than $\dim(V)-\dim(\mathcal{W})$ and which is defined over $C'$. We can take $\mathcal{C}$ as the formalization of $\bigcap_{\sigma}\sigma(D)$, where $\sigma$ runs over the absolute Galois group of $C$.
\end{proof}
\begin{example}\label{nonspecial}
We will give an example of a formal map which does not satisfy the conclusion of Theorem \ref{shortchar0}. Let $\mathcal{E}:\Cc\to \Cc$ be an analytic function whose graph is Zariski dense in $\Cc^2$ and such that $\mathcal{E}(0)=0$. Let us define
$$\mathcal{F}:\Cc^2\to \Cc^2,\ \ \mathcal{F}(x,y)=(y-\mathcal{E}(x),\mathcal{E}(y-\mathcal{E}(x))).$$
We denote also by $\mathcal{E}$ the corresponding formal map $\widehat{\Aa^1}\to \widehat{\Aa^1}$. Let us take $V=\Aa^2, A=\ga^2$ and consider
$\mathcal{F}$ as a formal map $\widehat{V}\to \widehat{A}$. We take $\mathcal{W}$ as the graph of $\mathcal{E}$. Then $\dim(V)-\dim(\mathcal{W})=1$, but the image of $\mathcal{F}$ (which is the graph of $\mathcal{E}$ as well) is not contained in any proper formal subgroup of $\widehat{A}$ (since all the formal subgroups of $\ga^2$ are also algebraic being just linear subspaces).
\\
This example can be easily modified to work in the positive characteristic case: we just need an extra assumption (automatically satisfied above) that $\mathcal{E}$ is non-additive.
\end{example}
\begin{remark}\label{arbitrary}
\begin{enumerate}
\item It is easy to formulate and prove a complex-analytic version of Theorem \ref{shortchar0}.

\item Proposition \ref{biglemma} is still not enough to tackle Question \ref{mainquem} without imposing extra assumptions on $A$. One may need a proper notion of a compatible sequences of higher invariant forms to prove a positive characteristic version of Remark \ref{formsgps}(1). Unfortunately, all our attempts to define such a notion were just leading into a system of forms coming from a formal endomorphism (see Remark \ref{formsgps}(3)).
\end{enumerate}
\end{remark}

\section{Applications to Ax-Schanuel inequalities}\label{applications}
In this section we will apply Theorems \ref{mainthm} and \ref{shortchar0} to get a version of Ax's theorem (\cite[Theorem 1]{ax72}) as well as some Ax-Schanuel type transcendental statements. In this section $C$ is a perfect field of an arbitrary characteristic.

We start with chronological remarks about the circle of topics around Schanuel's Conjecture. The reader is referred to Pila's notes \cite{Pilanotes} for a comprehensive survey.
\begin{itemize}
\item In 1960's Schanuel stated his famous conjecture \cite[(S)]{ax71} (see \cite[page 30-31]{lang1966}) as well as its function field (or formal) \cite[(SP)]{ax71} and differential \cite[(SD)]{ax71} versions.

\item The last two conjectures were proved by Ax \cite[Theorem 3]{ax71}. We will use the phrase ``Ax-Schanuel'' referring to results of the type \cite[(SP)]{ax71} and \cite[(SD)]{ax71}.

\item Shortly after, Ax generalized \cite[Theorem 1]{ax71} (a multivariable version of \cite[(SP)]{ax71}) from the case of the exponential map on an algebraic torus to the case of the exponential map on a semi-abelian variety \cite[Theorem 3]{ax72}. This generalization follows from the result \cite[Theorem 1F]{ax72} about intersections of algebraic subvarieties and formal subgroups of an algebraic group.

\item Bertrand  \cite{ber} extended \cite[Theorem 3]{ax72} to commutative algebraic groups not having vector quotients (e.g. a maximal non-split vectorial extensions of a semi-abelian variety).

\item Kirby \cite{ki} generalized \cite[(SD)]{ax71} to arbitrary semi-abelian varieties. This generalization was not included in \cite{ax72}, however it is closely related.

\item The differential Schanuel conjecture \cite[(SD)]{ax71} is generalized further to ``very non-algebraic formal maps'' in \cite[Theorem 5.5]{K5}. This generalization includes a differential version of Bertrand's result and a differential Ax-Schanuel type result about raising to a non-algebraic power on an algebraic torus.

\item The first (to my knowledge) positive characteristic version of \cite[(SP)]{ax71} is \cite[Theorem 1.1]{K6}, where the exponential map is replaced with a non-algebraic additive power series.

\item There is a variety of Ax-Schanuel type results for additive maps coming from \emph{Drinfeld modules}, see \cite{Brow}. However (to my knowledge), such results never include a version of the full Ax-Schanuel statement.

\item The original Schanuel conjecture \cite[(S)]{ax71} remains wide open. For example it is still unknown whether $\pi + e$ is irrational.

\end{itemize}
\noindent
We discuss now different applications of the Ax-Schanuel type. Schanuel's Conjecture \cite[(S)]{ax71} regards the exponential map on an algebraic torus and more generally the exponential map on a semi-abelian variety. Such maps do not exists in the positive characteristic case, but the most natural replacement (as in \cite{K5}) is a formal isomorphism which is ``very non-algebraic''.
In practise, it is a formal isomorphism between two ``very different'' algebraic groups or a formal endomorphism ``far'' from algebraic endomorphisms.
As our transcendental statements will concern images of rational points of algebraic groups under formal maps, we need to specify to what kind of rational points formal maps may be applied.

Let us fix $C$-schemes $V$ and $W$, and rational points (for simplicity with the same names) $0\in V(C),0\in W(C)$. By $\widehat{V}$ and $\widehat{W}$ we mean the formalizations at these fixed $C$-rational points. For a morphism of $C$-schemes $x:W\to V$, let $\locus_C(x)$ denote the Zariski closure of the image of $x$ and let
$$\trd_C(x):=\dim \locus_C(x).$$
If $W$ is the spectrum of a field, then the definition above corresponds to the classical notions of the algebraic locus and the transcendental degree of a rational point. If $W,V$ are affine and $x$ corresponds to the morphism of $C$-algebras $f$, then $\locus_C(x)$ coincides with the closed subscheme of $V$ given by $\ker(f)$.
\\
For a local $C$-algebra $U$, we denote by $V(U)_*$ the set of morphisms $x:\spec(U)\to V$ which take the closed point of $\spec(U)$ to $0$. It is easy to see that $V(U)_*$ corresponds exactly to the set of local $C$-algebra homomorphisms $\mathcal{O}_{V,0}\to U$. Clearly, for any $x\in V(U)_*$, we have $0\in \locus_C(x)$.
\\
Assume now that $\mathcal{R}$ is a complete $C$-algebra. Then $V(\mathcal{R})_*$ corresponds to the set of local $C$-algebra homomorphisms $\widehat{\mathcal{O}}_{V,0}\to \mathcal{R}$ being completions of local $C$-algebra homomorphisms $\mathcal{O}_{V,0}\to \mathcal{R}$.
\begin{remark}\label{pointexample}
Any formal map $\mathcal{F}:\widehat{V}\to \widehat{W}$ naturally induces a map
$$\mathcal{F}_{\mathcal{R}}: V(\mathcal{R})_*\to  W(\mathcal{R})_*.$$
\end{remark}
\begin{proof}
Let us take $x\in V(\mathcal{R})_*$ and consider the following diagram
\begin{equation*}
  \xymatrix{ & \spec(\mathcal{R})  \ar[rd]^{\mathcal{F}_{\mathcal{R}}(x)}  \ar[ld]^{x} \ar[ldd]^{\bar{x}}&  \\
 V & & W \\
\spec(\widehat{\mathcal{O}}_{V,0}) \ar[rr]^{\mathcal{F}} \ar[u]^{}& &\spec(\widehat{\mathcal{O}}_{W,0}) \ar[u]^{},}
\end{equation*}
where $\mathcal{F}_{\mathcal{R}}(x)$ is the obvious composition map.
\end{proof}
\begin{example}
Let us take $C=\Qq$, $\mathcal{R}=\Qq \llbracket X_1,\ldots,X_r\rrbracket$, $x\in \ga(\mathcal{R})_*$ and
$$\mathcal{F}=\exp:\widehat{\ga}\to \widehat{\gm}.$$
Then $x$ corresponds to a power series $f$ without the constant term and $\mathcal{F}(x)$ corresponds to
$$\exp(f)=1+f+\frac{f^2}{2!}+\frac{f^3}{3!}+\ldots,$$
which makes sense, since $f$ has no constant term. Thus the map
$$\mathcal{F}_{\mathcal{R}}:\ga(\mathcal{R})_*\to \gm(\mathcal{R})_*$$
is the same as the exponential map evaluated at the maximal ideal of $\mathcal{R}$.
\end{example}
\noindent
We will need the notion of the formal locus of a point.
\begin{definition}
For $x\in V(\mathcal{R})_*$ we define.
\begin{enumerate}
\item The formal locus of $x$ over $C$ as the formal subscheme of $\widehat{V}$ corresponding to the image of the map $\widehat{\mathcal{O}}_{V,0}\to \mathcal{R}$.

\item The number $\andeg(x)$ denoting the dimension of the formal locus of $x$ over $C$.
\end{enumerate}
\end{definition}
\begin{remark}\label{starpoint}
\begin{enumerate}
\item It is easy to see that for $x\in V(\mathcal{R})_*$, the formal locus of $x$ is contained in (the formalization at $0$ of) the algebraic locus of $x$. Hence we have
    $$\andeg_C(x)\leqslant \trd_C(x).$$

\item If $\mathcal{R}$ is a power series algebra and $V$ is an affine space, then $\andeg(x)$ coincides with the rank of the Jacobian of $x$ in the case of $\ch(C)=0$.
\end{enumerate}
\end{remark}
\noindent%
The result below is a generalization (from the case of a vector group to an arbitrary algebraic group) of \cite[Proposition 4.4]{K6}. The proof is an easy adaptation of the proof from \cite{K6}, so we skip it.
\begin{prop}\label{algform}
Assume $G$ is an algebraic group, $V$ a subvariety containing the identity element of $G$ and $\mathcal{H}$ a formal subgroup of $\widehat{G}$. If $\widehat{V}\subseteq \mathcal{H}$, then $\widehat{H}\subseteq \mathcal{H}$, where $H$ is the algebraic subgroup of $G$ generated by $V$.
\end{prop}
\begin{remark}\label{formsgps}
\begin{enumerate}
\item In the case of characteristic $0$, Proposition \ref{algform}, maybe restated in the following form: if $\omega\in \Omega^{\inv}_G$ and $\omega|_V=0$, then $\omega|_H=0$ (see \cite[Propositon 2.5]{K5}).

\item The above restatement does not hold in the positive characteristic case, e.g. one can take $C=\Ff_3$, $\omega=\ddd X\in \Omega^{\inv}_{\ga^2}$ and $V$ given by $X^2-Y^3$.

\item One could imagine proving a weaker version of Theorem \ref{mainthm} for a ``right'' (i.e. satisfying item $(1)$ above) system of higher invariant forms in place of a formal subgroup. Unfortunately (as already mentioned in Remark \ref{arbitrary}) all the notions of ``right'' higher invariant forms we could come up with reduced to collections of higher forms coming from a formal subscheme.
\end{enumerate}
\end{remark}

\noindent%
We set now the notation for the remainder of this section. We fix the following.
\begin{itemize}
\item A complete $C$-algebra $\mathcal{R}$ with the residue field $C$ such that $\mathcal{R}$ is linearly disjoint from $C^{\alg}$ over $C$ and in the case of characteristic $p$ such that $L^{p^{\infty}}=C$, where $L$ is the fraction field of $\mathcal{R}$ (e.g. $\mathcal{R}$ may be the power series algebra).

\item Commutative algebraic groups $A,B$ of dimension $n$ over $C$.

\item A formal isomorphism $\mathcal{E}:\widehat{A}\to \widehat{B}$ defined over $C$.
\end{itemize}
\noindent
Moreover we assume that
\begin{itemize}
\item $\widehat{B}\cong \widehat{H}^n$, where $H$ is an integrable $1$-dimensional algebraic group (see Definition \ref{defintegr} and note that this assumption is restrictive only in the positive characteristic case);

\item if $\ch(C)>0$, then $\mathcal{E}$ is a $B$-limit map.
\end{itemize}
\noindent
There will be a common independence condition implying transcendence (as linear dependence over $\Qq$ in the statement of Schanuel's Conjecture). We define this notion below.
\begin{definition}\label{subinddef}
 Take $x\in A(\mathcal{R})$ and assume that for any proper algebraic subgroup $A_0<A$ defined over $C$, we have $x\notin A_0(\mathcal{R})$. Then we call $x$ \emph{subgroup independent}.
\end{definition}
\begin{remark}
In \cite{Pila11}, an element satisfying a similar condition as above is called \emph{geodesically independent}.
\end{remark}

\begin{example}
\begin{enumerate}
\item For $A=\ga^n$, the subgroup independence coincides with the $C$-linear independence if $\ch(C)=0$, and with the $C[\fr]$-linear independence if $\ch(C)>0$.

\item For $A=\gm^n$, the subgroup independence coincides with the $\Zz$-linear independence (in the commutative group $\gm^n(\mathcal{R})$).
\end{enumerate}
\end{example}

\subsection{Ax-Schanuel type I}\label{secaxsch}
In this part we deal with the case of  formal isomorphism between two ``very different'' algebraic groups. We recall a definition from \cite{K5}.
\begin{definition}\label{defalgor}
We say that $A$ and $B$ are \emph{essentially different}, if any connected algebraic subgroup of $A\times B$ is of the form $A_0\times B_0$, where $A_0$ is an algebraic subgroup of $A$ and $B_0$ is an algebraic subgroup of $B$.
\end{definition}
\begin{remark}
It is easy to see that $A$ and $B$ are essentially different if and only if they do not have infinite algebraic isomorphic \emph{sections}, where an algebraic section of an algebraic group is a quotient of its algebraic subgroup by a Zariski closed normal subgroup.
\end{remark}

\begin{example}\label{exao}
The following $A,B$ are essentially different. In the case of $\ch(C)=0$, there is always a formal isomorphism $\mathcal{E}:\widehat{A}\to \widehat{B}$.
\begin{enumerate}
\item $A=\ga^n$ and $B$ is a semi-abelian variety. No formal isomorphism exists for $\ch(C)>0$.

\item $A=\gm^n$ and $B$ is a vectorial extension of an abelian variety. In some cases there is a formal isomorphism for $\ch(C)>0$, for example if $B$ is a Cartesian power of an ordinary elliptic curve.
\end{enumerate}
\end{example}
\noindent
The theorem below is a mild generalization of \cite[Theorem 3]{ax72}, where $\mathcal{E}$ is the formal inverse of the exponential map on a semi-abelian variety in characteristic $0$.
\begin{theorem}\label{schanuelax}
Assume that $A,B$ are essentially different. Then for any subgroup independent $x\in A(\mathcal{R})_*$ we have
$$\trd_C(x,\mathcal{E}_{\mathcal{R}}(x))\geqslant n+\andeg_C(x).$$
\end{theorem}
\begin{proof}
The proof goes basically as in \cite[Theorem 3]{ax72}. Let $g:=(x,\mathcal{E}_{\mathcal{R}}(x))$, $G:=A\times B$, $V\subseteq G$ be the algebraic locus of $g$ over $C$ and $\mathcal{W}$ be the formal one. By our assumptions on $\mathcal{R}$, $V$ and $\mathcal{W}$ satisfy the assumption from Theorem \ref{mainthmintro} (see Remark \ref{starpoint}(1)).
Let us define
$$\mathcal{T}:\widehat{G}\to \widehat{B},\ \ \mathcal{T}=\mathcal{E}\circ\pi_1-\pi_2,$$
where $\pi_1,\pi_2$ are the appropriate coordinate projections. Let $\mathcal{A}=\ker(\mathcal{T})$ be the ``graph of $\mathcal{E}$''. By Theorem \ref{minetoax}, there is $\mathcal{B}$, a formal subgroup of $G$ such that $V,\mathcal{A}\subseteq \mathcal{B}$ and
$$\dim(\mathfrak{B})\leqslant \dim(\mathcal{A})+\dim(V)-\dim(\mathcal{W})=n+\trd_C(g)-\andeg_C(x).$$
By Proposition \ref{algform}, there is a connected algebraic subgroup $H\leqslant G$ containing $V$ such that $\widehat{H}\subseteq \mathfrak{B}$. By our assumptions, $H=H_A\times H_B$, where $H_A,H_B$ denote the appropriate projections of $H$. Since $x$ is subgroup independent and $x\in H_A(\mathcal{R})$, we get $H_A=A$. Therefore $\widehat{A},\mathcal{A}\subseteq \mathfrak{B}$, hence $\mathfrak{B}=\widehat{G}$. Thus $\dim(\mathfrak{B})=2n$ giving the desired inequality.
\end{proof}
\begin{remark}
Unfortunately, the positive characteristic restrictions we were forced to put in Theorem \ref{mainthm} eliminate all the positive characteristic cases here (the most important one being Example \ref{exao}(2)).
\end{remark}

\subsection{Ax-Schanuel type II}\label{sectionendo}
In this part we consider a case rather opposite to the Ax-Schanuel type I situation. Our algebraic groups are not ``very different'' (being identical!) but the formal endomorphism is ``very non-algebraic''. Such a case was first considered in \cite[Theorem 6.12]{K5} for a characteristic $0$ torus (differential version, i.e. corresponding to \cite[(SD)]{ax71}) and then in \cite{K6} for positive characteristic vector group (formal version, i.e. corresponding to \cite[(SP)]{ax71}). Defining the right notion of non-algebraicity is very easy and natural here. Let us fix:
\begin{itemize}
\item a positive integer $n$;

\item a one-dimensional algebraic group $H$ over $C$ and $A=B=H^n$.
\end{itemize}
\noindent
We introduce the following notation.
\begin{itemize}
\item Let $\mathbf{R}$ denote the ring of algebraic endomorphisms of $H$.

\item Let $\mathbf{S}$ denote the ring of formal endomorphisms of $\widehat{H}$. We restrict our attention to algebraic groups $H$ such that $\mathbf{S}$ is a commutative domain. We regard $\mathbf{R}$ as a subring of $\mathbf{S}$.

\item Let $\mathbf{K}$ denote the field of fractions of $\mathbf{R}$ and $\mathbf{L}$ be the field of fractions of $\mathbf{S}$. We regard $\mathbf{K}$ as a subfield of $\mathbf{L}$.
\end{itemize}
\begin{example}\label{ax2ex}
Note that in the case of the characteristic $0$, we always have $\mathbf{S}=C$, so we can consider any one-dimensional algebraic groups as $H$.  We give examples of possible $H,\mathbf{R},\mathbf{S},\mathbf{K},\mathbf{L}$ below.
\begin{enumerate}
\item If $H=\ga$ and characteristic is $0$, then $\mathbf{R}=\mathbf{S}=C$.

\item If $H=\ga$ and characteristic is $p>0$, then $\mathbf{R}=C[\fr]$ and $\mathbf{S}=C\llbracket \fr\rrbracket$. Thus we need to take $C=\Ff_p$ to guarantee that $\mathbf{S}$ is commutative. This case is analyzed in \cite{K6}.

\item If $H=\gm$, then $\mathbf{R}=\Zz$. The case of characteristic $0$ was analyzed in \cite[Theorem 6.12]{K5}. In the case of characteristic $p>0$, we have $\mathbf{S}=\Zz_p$ ($p$-adic integers, see Example \ref{alimitexample}) and new interesting non-algebraic maps.

\end{enumerate}
\end{example}
\noindent
Below is our transcendental statement about formal endomorphisms.
\begin{theorem}\label{endoax}
Take $\gamma\in \mathbf{S}$ such that $[\mathbf{K}[\gamma]:\mathbf{K}]>n$ and $\gamma:\widehat{H}\to \widehat{H}$ is an $H$-limit map. Let $\mathcal{E}:\widehat{A}\to \widehat{A}$ be the $n$-th cartesian power of $\gamma$.
Then for any subgroup independent $x\in A(\mathcal{R})_*$ we have
$$\trd_C(x,\mathcal{E}_K(x))\geqslant n+\andeg_C(x).$$
\end{theorem}
\begin{proof}
Let us denote (as in the proof of Theorem \ref{schanuelax})
$$g:=(x,\mathcal{E}_{\mathcal{R}}(x))\in H^{2n}(\mathcal{R})_*.$$
Let $V$ be the algebraic locus of $g$ over $C$ and $\mathcal{W}$ the formal one.
\\
We have
$$\dim(V)=\td_C(g),\ \ \ \dim(\mathcal{W})=\andeg_C(x).$$
Let $\mathcal{F}:\widehat{V}\to \widehat{H}^n$ be the restriction to $\widehat{V}$ of the following formal map
$$\widetilde{\mathcal{F}}:=\mathcal{E}\circ\pi_1-\pi_2:\widehat{H}^n\times \widehat{H}^n\to \widehat{H}^n,$$
where $\pi_1,\pi_2$ are the appropriate coordinate projections. As in the proof of Theorem \ref{schanuelax},  $V$ and $\mathcal{W}$ satisfy the assumptions from Theorem \ref{mainthm}.
\\
By Theorem \ref{mainthm}, there is a formal subgroup $\mathcal{A}\subseteq \widehat{H}^n$ containing $\mathcal{F}(\widehat{V})$ such that
$$\dim(\mathcal{A})\leqslant \dim(V)-\dim(\mathcal{W})=\td_C(g)-\andeg_C(x).$$
Assume now that $\td_C(g)< n+\andeg_C(x)$, which implies that $\mathcal{A}$ is proper. We will reach a contradiction.
\\
Since $\mathcal{A}$ is proper, there are $\alpha_1,\ldots,\alpha_n\in \mathbf{S}$ not all zero, such that
$$h:=\alpha\circ \widetilde{\mathcal{F}}:\widehat{H}^{2n}\to \widehat{H}$$
vanishes on $V$, where $\alpha:\widehat{H}^{n}\to \widehat{H}$ is given by $(\alpha_1,\ldots,\alpha_n)$.
\\
Let $V_H$ be the algebraic subgroup of $H^{2n}$ generated by $V$. By Proposition \ref{algform}, $h$ vanishes on $V_H$.
For any $i\leqslant 2n$, let $\pi_i:H^{2n}\to H^i$ denote the projection map on the first $i$ coordinates. Since $x$ does not belong to any proper algebraic subgroup of $H^n$, we get $\pi_n(V_H)=H^n$. Let $0\leqslant m<n$ be such that $\dim(V_H)=n+m$.
\\
Since $H^{2n}/V_H\cong H^{n-m}$, there is a matrix $M\in M_{2n,n-m}(R)$ of rank $n-m$ such that
$$H=\ker(M:H^{2n}\to H^{n-m}).$$
We replace now the object $H$ acted on by $\mathbf{R}$ and $\mathbf{S}$ (in the appropriate categories), by $\mathbf{L}$ acted on by  $\mathbf{R}$ and $\mathbf{S}$ in the obvious way. Let $\mathbf{W}$ be $\ker(M:\mathbf{L}^{2n}\to \mathbf{L}^{n-m})$. We have the following sequence of maps whose composition is the $0$-map:
\begin{equation*}
  \xymatrix{
\mathbf{W} \ar[r]^{ \subset} &  \mathbf{L}^{2n} \ar[r]^{f} &  \mathbf{L}^n \ar[r]^{a} & \mathbf{L},}
\end{equation*}
where $f$ is given by the matrix over $\mathbf{S}$ coming from $\widetilde{F}$ and $a$ is given by the matrix over $\mathbf{S}$ coming from $\alpha$. Since $\pi_n(V_H)=H^n$ and $\mathbf{W}$ is a vector subspace defined over $\mathbf{K}$, there is a matrix $N\in M_{n+m,2n}(\mathbf{K})$ such that the diagram below is commutative and the composition of the upper row is the $0$-map
\begin{equation*}
  \xymatrix{
\mathbf{L}^n\times \mathbf{L}^m \ar[rr]^{N} \ar[rd]_{\pi_1} & & \mathbf{L}^n\times \mathbf{L}^n \ar[ld]^{\pi_1} \ar[r]^{\ \ \ f} &  \mathbf{L}^n \ar[r]^{a} & \mathbf{L}\\
 & \mathbf{L}^n & & & ,}
\end{equation*}
where both maps denoted by $\pi_1$ are projections on the first (multi)coordinate.
Thus for all $\bar{l}=(l_1,\ldots,l_{n+m})\in \mathbf{L}^{n+m}$ we have
$$\sum_{k=1}^m\alpha_{\mathbf{k}}(l_{n+k}-\gamma l_{\mathbf{k}})+\sum_{k=m+1}^n\alpha_{\mathbf{k}}(N(\bar{l})_{\mathbf{k}}-\gamma l_{\mathbf{k}})=0,$$
where $N(\bar{l})_{\mathbf{k}}$ denotes the $k$-th coordinate of $N(\bar{l})$. Since $\alpha\neq 0$, there is $m<k\leqslant n$ such that $\alpha_{\mathbf{k}}\neq 0$. Putting $l_1=l_{n+1}=\ldots=l_m=l_{n+m}=0$, we get
$$(\alpha_{m+1},\ldots,\alpha_n)N'=(\gamma \alpha_{m+1},\ldots,\gamma \alpha_n),$$
where $N'\in M_{n-m}(\mathbf{K})$ is an appropriate block (the ``middle square'') of $N$. Hence $\gamma$ is a characteristic value of $N'$. By the Cayley-Hamilton theorem, $\gamma$ is algebraic over $\mathbf{K}$ of degree at most $n$, which gives a contradiction.
\end{proof}
\begin{remark}
The $H$-limit assumption from Theorem \ref{endoax} is not very restrictive, since in all the cases considered in Example \ref{ax2ex}, a corresponding formal map is an $H$-limit map if and only if it is a special map (see Remark \ref{unipalimit}(2)).
\end{remark}

\section{Special formal maps and $A$-limit maps}\label{secspeciallimit}
\noindent%
In this section we aim to find a general criterion for a formal map $\widehat{V}\to \widehat{A}$ to be an $A$-limit map. Our aim is to show that any formal homomorphism $\widehat{B}\to \widehat{A}$ is an $A$-limit map, but we fall a little short of it, i.e. we need to put some restrictions on $A$.
Our set-up for this section is as follows:
\begin{itemize}
\item Let $C$ be a perfect field of characteristic $p>0$.

\item Let $A$ be a commutative algebraic group over $C$.


\item Let $V$ be an absolutely irreducible scheme over $C$ and $\dim(V)=t$.

\item Let $v\in V(C)$ be a smooth point and $R$ be the local ring of $V$ at $v$.

\item Let $K$ be the fraction field of $R$.

\item Let $L$ be the fraction field of $\widehat{R}$.
\end{itemize}
\noindent
Unlike in the previous sections, we assume now that $v$ is a smooth point. We can do it, since we want to apply the main result of this section (Theorem \ref{alimitthm}) in the case when $v$ is the neutral element of an algebraic group.
\\
In the following lemma, we notice a property of the rings defined above which will be crucial later. We consider $R$ as a subring of both $\widehat{R}$ and $K$. We also consider $\widehat{R}$ and $K$ as subrings of $L$. For the notion of a $p$-basis, see Section \ref{notation}.
\begin{lemma}\label{allringsok}
Any system of regular parameters of $R$
 is a $p$-basis of each of $R,K,\widehat{R}$ and $L$.
\end{lemma}
\begin{proof}
Let $\{r_1,\ldots,r_t\}$ be a system of regular parameters of $R$. Since $\widehat{R}$ is isomorphic to the power series $C$-algebra in $\{r_1,\ldots,r_t\}$, the result is clear for $\widehat{R}$ and $L$.
\\
It is well-known that the set $\{\ddd r_1,\ldots,\ddd r_t\}$ is a basis of $\Omega_R$ (as an $R$-module): this set clearly generates $\Omega_R$ over $R$ and it is $R$-independent, since it is $\widehat{R}$-independent in $\Omega_{\widehat{R}}$ (see also \cite[Theorem II.8.8]{ha}). Thus the set $\{\ddd r_1,\ldots,\ddd r_t\}$ is also a basis of $\Omega_K$ (as a $K$-vector space). By \cite[Theorem 26.5]{mat}, $\{r_1,\ldots,r_t\}$ is a $p$-basis of $K$.
\\
It remains to check that $\{r_1,\ldots,r_t\}$ is a $p$-basis of $R$. We have the following:
\begin{IEEEeqnarray*}{rCl}
R & = & K\cap \widehat{R}  \\
& = & (K^pC[r_1,\ldots,r_t])\cap \widehat{R} \\
& = & (K^p\cap \widehat{R}^p)C[r_1,\ldots,r_t] \\
& = & R^pC[r_1,\ldots,r_t].
\end{IEEEeqnarray*}
The first equality and the last equality follow from the fact that the extension $R\subset \widehat{R}$ is faithfully-flat  \cite[Theorem 56 page 172]{mat0}. The second equality follows from the fact that $\{r_1,\ldots,r_t\}$ is a $p$-basis of $K$ and the third one follows from the fact that $\{r_1,\ldots,r_t\}$ is a $p$-basis of $\widehat{R}$. Since the $p$-independence of $\{r_1,\ldots,r_t\}$ in $R$ is clear (by e.g. the $p$-independence in $K$), we get that $\{r_1,\ldots,r_t\}$ is a $p$-basis of $R$.
\end{proof}
\noindent

\subsection{Weil restriction}
In this subsection, we collect the properties of the Weil restriction, which will be needed in the sequel.

\subsubsection{General facts about the Weil restriction}\label{weilgeneralsec}
Everything in this section (except representability) is ``abstract nonsense'', i.e. works in any category with products (replacing the category of schemes). We will consider schemes as covariant functors from the category of algebras (over a fixed ring) to the category of sets.
\\
Let us fix a ring $\mathbf{r}$ and an $\mathbf{r}$-algebra $\mathbf{s}$. We define a functor
$$\Pi_{\mathbf{s}}:\mathrm{Func}(\mathbf{Alg}_{\mathbf{r}},\mathbf{Set})\to \mathrm{Func}(\mathbf{Alg}_{\mathbf{r}},\mathbf{Set})$$
as follows. For a functor $F:\mathbf{Alg}_{\mathbf{r}}\to \mathbf{Set}$ and an $\mathbf{r}$-algebra $\mathbf{r}'$ let
$$(\Pi_{\mathbf{s}}F)(\mathbf{r}'):=F(\mathbf{s}\otimes_{\mathbf{r}}\mathbf{r}').$$
The functor $\Pi_{\mathbf{s}}$ is the composition of the base change functor from $\mathbf{r}$ to $\mathbf{s}$ with the \emph{Weil restriction} functor
from $\mathbf{s}$ to $\mathbf{r}$ (see \cite[Section 7.6]{neron} for a discussion about the Weil restriction). It is clear that the assignment $\mathbf{s}\mapsto \Pi_{\mathbf{s}}$ is a functor itself (from the category $\mathbf{Alg}_{\mathbf{r}}$ to the category of endofunctors on the category $\mathrm{Func}(\mathbf{Alg}_{\mathbf{r}},\mathbf{Set})$).
\\
Let $\mathbf{Sch}^{\mathrm{f}}_{\mathbf{r}}$ denote the full subcategory of the category of schemes over $\mathbf{r}$ consisting of schemes $V$ such that each finite set of points of $V$ is contained in an affine open subscheme of $V$ (e.g. a quasi-projective scheme satisfies this condition). We collect below (from \cite{neron}) the properties of the functor $\Pi$ which will be needed.
\begin{theorem}\label{weilneeded}
Suppose $\mathbf{r}\subseteq \mathbf{s}$ is a finite ring extension and $V\in \mathbf{Sch}^{\mathrm{f}}_{\mathbf{r}}$.
\begin{enumerate}
\item If a functor $F:\mathrm{Alg}_{\mathbf{r}}\to \mathrm{Set}$ is representable by a scheme from $\mathbf{Sch}^{\mathrm{f}}_{\mathbf{r}}$, then the functor $\Pi_{\mathbf{s}}F$ is representable by a scheme from $\mathbf{Sch}^{\mathrm{f}}_{\mathbf{r}}$, so we get a functor:
    $$\Pi_{\mathbf{s}}:\mathbf{Sch}^{\mathrm{f}}_{\mathbf{r}}\to \mathbf{Sch}^{\mathrm{f}}_{\mathbf{r}}.$$

\item There is a natural transformation of functors
$$\iota_V:V\to \Pi_{\mathbf{s}}V,$$
which is a closed embedding.

\item There is a section of $\iota_V$ defined over $\mathbf{s}$ i.e. a natural transformation
$$s_V:\Pi_{\mathbf{s}}(V)\times_{\mathbf{r}}\mathbf{s}\to V\times_{\mathbf{r}}\mathbf{s}$$
such that $s_V\circ (\iota_V\times_{\mathbf{r}}\mathbf{s})$ is the identity map.

\item The items $(1)$, $(2)$ and $(3)$ above hold also in the category of groups in $\mathbf{Sch}^{\mathrm{f}}_{\mathbf{r}}$.
\end{enumerate}
\end{theorem}
\begin{proof}
For $(1)$, we quote \cite[page 194 (Theorem 4)]{neron}
\\
For $(2)$, we quote \cite[page 197 (bottom)]{neron}
\\
For $(3)$ and $(4)$, we quote \cite[page 192 (top)]{neron}
\end{proof}
\begin{remark}\label{iota}
Note that the natural morphism $\iota_V:V\to \Pi_{\mathbf{s}}V$ coincides with $\Pi_{\iota}V$, where the morphism $\iota:\mathbf{r}\to \mathbf{s}$ is the inclusion map (clearly, $\Pi_{\mathbf{r}}V=V$).
\end{remark}

\subsubsection{Specific facts about Weil restriction}\label{secweilsp}
Let us fix $\mathbf{k}:=C[X_1,\ldots,X_t]$. The role of the ring $\mathbf{r}$ from Section \ref{weilgeneralsec} will be played by the $C$-algebra $\mathbf{k}^p$. \\
First we note a crucial correspondence between $p$-th powers and tensor products.
\begin{lemma}\label{p1}
Let $W$ be a $C$-algebra having a $p$-basis of cardinality $t$ and let $V$ be a $\mathbf{k}^p$-scheme. Then we have:
\begin{enumerate}

\item A choice of a $p$-basis gives $W$ a $\mathbf{k}$-algebra structure and the natural map
$$W^p\otimes_{\mathbf{k}^p}\mathbf{k}\to W $$
is both a $\mathbf{k}$-algebra isomorphism and a $W^p$-algebra isomorphism. This isomorphism is natural with respect to $C$-algebra extensions which preserve $p$-bases.

\item There is a natural (in the above sense on $W$ and without restrictions on $V$) bijection $\Pi_{\mathbf{k}}V(W^p)\to V(W)$ such that the following diagram is commutative
\begin{equation*}
  \xymatrix{V(W^p) \ar[rr]^{\iota_V} \ar[rrd]_{\subseteq}& &\Pi_{\mathbf{k}}V(W^p) \ar[d]^{} \\
& &V(W) .}
\end{equation*}

\end{enumerate}
\end{lemma}
\begin{proof}
By an argument as in \cite[Theorem 26.8]{mat}, any $p$-basis of $W$ is algebraically independent over $C$, hence a choice of a $p$-basis gives $W$ a $\mathbf{k}$-algebra structure. Since $\mathbf{k}$ contains a $p$-basis of $W$, the natural map $W^p\otimes_{\mathbf{k}^p}\mathbf{k}\to W $ is onto. By comparing the ranks over $W^p$, we see that this map is an isomorphism proving $(1)$. The item $(2)$ follows.
\end{proof}
\noindent
Let us introduce a new ring
$$\mathbf{k}[\varepsilon]:=\mathbf{k}[Y_1]/(Y_1^2)\times_{\mathbf{k}} \ldots \times_{\mathbf{k}} \mathbf{k}[Y_t]/(Y_t^2).$$
Clearly, a $\mathbf{k}$-algebra homomorphism $\mathbf{k}\to \mathbf{k}[\varepsilon]$ corresponds to a $t$-tuple of derivations $(\partial_1,\ldots,\partial_t)$ on $\mathbf{k}$. Since we are not interested in the interactions between $\partial_1,\ldots,\partial_t$, we prefer to use the ring $\mathbf{k}[\varepsilon]$ rather than the bigger (tensor product) ring $\mathbf{k}[Y_1,\ldots,Y_t]/(Y_1^2,\ldots,Y_t^2)$.
\\
Let us assume that $\mathbf{A}$ is a commutative algebraic group scheme over $\mathbf{k}^p$ (in the applications we take $\mathbf{A}=A\times_C\mathbf{k}^p$).
We will use the group schemes $\Pi_{\mathbf{k}}\mathbf{A}$ and $\Pi_{\mathbf{k}[\varepsilon]}\mathbf{A}$ to understand the cokernel of the map $\mathbf{A}(R^p)\to \mathbf{A}(R)$. See Section \ref{logdersec} for the interpretation of $\Pi_{\mathbf{k}[\varepsilon]}\mathbf{A}$ in terms of tangent spaces.
\\
Any $t$-tuple $\mathbf{\partial}$ of derivations on $\mathbf{k}$ gives a $\mathbf{k}^p$-algebra map (denoted by the same symbol)
$$\mathbf{\partial}:\mathbf{k}\to \mathbf{k}[\varepsilon].$$
By Theorem \ref{weilneeded}(1,4), we also get a group scheme homomorphism
$$\Pi_{\mathbf{\partial}}\mathbf{A}:\Pi_{\mathbf{k}}\mathbf{A}\to \Pi_{\mathbf{k}[\varepsilon]}\mathbf{A}.$$
We will consider the tuple of $0$-derivations $\mathbf{\partial}_0$ on $\mathbf{k}$ and the tuple of standard partial derivations $\mathbf{\partial}_{\mathbf{k}}$ on $\mathbf{k}$. For the notion of a \emph{principal homogenous space} (PHS) of a group scheme over a ring, the reader is advised to consult \cite[Section III.4]{milne1etale}.
\begin{prop}\label{weilneeded1}
Let $W$ be a $C$-algebra as in Lemma \ref{p1}. We define
$$f_\mathbf{A}:=\Pi_{\mathbf{\partial}_0}\mathbf{A}-\Pi_{\mathbf{\partial}_{\mathbf{k}}}\mathbf{A}:\Pi_{\mathbf{k}}\mathbf{A}\to \Pi_{\mathbf{k}[\varepsilon]}\mathbf{A}.$$
Then we have:
\begin{enumerate}

\item $f_\mathbf{A}\circ \iota_\mathbf{A}=0$;

\item for any $v\in f_\mathbf{A}(\Pi_{\mathbf{k}}(\mathbf{A})(W^p))$, the fiber $f^{-1}(v)$ is a PHS of $\mathbf{A}$ over $W^p$;

\item $(\Pi_{\mathbf{k}[\varepsilon]}\mathbf{A}\to \mathbf{A})\circ f_\mathbf{A}=0$.
\end{enumerate}
\end{prop}
\begin{proof}
For $(1)$ it is enough to see that
$$\Pi_{\mathbf{\partial}_0}\mathbf{A}\circ \iota_\mathbf{A}= \Pi_{\mathbf{\partial}_{\mathbf{k}}}\mathbf{A}\circ \iota_\mathbf{A}.$$
By Remark \ref{iota}, $\iota_\mathbf{A}=\Pi_{\iota}\mathbf{A}$, where $\iota:\mathbf{k}^p\to \mathbf{k}$ is the inclusion map. Hence we get
$$\Pi_{\mathbf{\partial}_0}\mathbf{A}\circ \iota_\mathbf{A}=\Pi_{\mathbf{\partial}_0\circ \iota}\mathbf{A},\ \ \Pi_{\mathbf{\partial}_{\mathbf{k}}}\mathbf{A}\circ \iota_\mathbf{A}=\Pi_{\mathbf{\partial}_{\mathbf{k}}\circ \iota}\mathbf{A}.$$
Since $\partial_{\mathbf{k}}$ is trivial on $\mathbf{k}^p$, we get $f_\mathbf{A}\circ \iota_\mathbf{A}=0$.
\\
For the proof of $(2)$, note that the item $(1)$ gives an action of $\mathbf{A}\times_{\mathbf{k}^p}W^p$ on $f^{-1}(v)$ (a group scheme action).
Since $\mathbf{k}^p$ is the field of constants of $\mathbf{\partial}_{\mathbf{k}}$, the inclusion $\mathbf{k}^p\to \mathbf{k}$ is the equalizer of the maps $\mathbf{\partial}_0,\mathbf{\partial}_{\mathbf{k}}:\mathbf{k}\to \mathbf{k}[\varepsilon]$. Since the equalizer functor commutes with the flat base change  (tensoring with a flat module commutes with equalizers), we get that for any flat affine covering of $\spec(W^p)$ the scheme action above satisfies the conditions from \cite[Prop. 4.1]{milne1etale} (note that $\mathbf{A}\times_{\mathbf{k}^p}W^p$ is flat over $W^p$). Hence the fiber $f^{-1}(v)$ is a PHS of $\mathbf{A}$ over $W^p$.
\\
For the proof of $(3)$ (similarly as in the proof of $(1)$), it is enough to notice that
$$\pi_{\mathbf{k}}\circ \mathbf{\partial}_0=\id_{\mathbf{k}}=\pi_{\mathbf{k}}\circ \mathbf{\partial}_{\mathbf{k}},$$
where $\pi_{\mathbf{k}}:\mathbf{k}[\varepsilon]\to \mathbf{k}$ is the projection map.
\end{proof}

\subsection{The $(*)$-property}
In this section we prove a general fact about rational points of commutative algebraic groups. It will be used in the proof of the main result of this section (Step 2 of the proof of Theorem \ref{alimitthm}). The proof of this result requires the Weil restriction techniques developed in Section \ref{secweilsp} and \'{e}tale cohomology.
\begin{prop}\label{star}
Assume that the maximal algebraic torus in $A$ is diagonalizable. Let $x\in A(\widehat{R})$ and $y\in A(K)$ be such that $x-y\in A(L^p)$. If $A$ is affine or $\dim(R)=1$, then there is $z\in A(R)$ such that $x-z\in A(\widehat{R}^p)$.
\end{prop}
\begin{proof}
Assume that the maximal algebraic torus in $A$ is diagonalizable. Let us say that such an $A$ satisfies the $(*)$-property if the proposition above is true for $A$. We will see first that if $A$ is affine, then $A$ satisfies the $(*)$-property. For any $C$-algebra $W$ as in Lemma \ref{p1} and any $t\in A(W)$, let $t'\in \Pi A(W^p)$ denote the element obtained using the bijection from Lemma \ref{p1}(2). Let us take $x$ and $y$ as in the assumptions of the $(*)$-property. By Lemma \ref{p1}(2) and Proposition \ref{weilneeded1}(1), we get that $f(x')=f(y')$. Let us denote
$$v:=f(x')=f(y')\in \beta A(K^p\cap \widehat{R}^p).$$
Since the extension $R\subset \widehat{R}$ is faithfully-flat \cite[Theorem 56 page 172]{mat0}, we have $K^p\cap \widehat{R}^p=R^p$. Therefore $v\in A(R^p)$.
\\
Let $P=f^{-1}(v)$ (the scheme-theoretic fiber). By Proposition \ref{weilneeded1}(2), $P$ is a PHS of $A$ over $R^p$.  By \cite[Remark 4.8(a)]{milne1etale}, the isomorphism classes of such principal homogenous spaces are classified by the elements of the group $H^1(\spec(R^p)_{\et},\mathcal{A})$, where $\mathcal{A}$ is the sheaf of commutative groups on $\spec(R^p)_{\et}$ given by $A$. We will see below that $H^1(\spec(R^p)_{\et},\mathcal{A})$ is trivial (which should be folklore), hence there is $z'\in P(R^p)$ giving our $z\in A(R)$ as required.
\\
By our assumption about maximal torus and structure theorems about commutative affine algebraic groups (see e.g. \cite[Chap. III, 2.1]{serre2002galois}), the group $A$ has a composition series over $C$ with factors $\ga$ or $\gm$. By the long exact sequence of sheaf cohomology \cite[Prop. III.4.5]{milne1etale}, we can assume that $A=\ga$ or $A=\gm$. If $A=\ga$, then we have
$$H^1(\spec(R^p)_{\et},\mathcal{A})=H^1(\spec(R^p)_{\et},\mathcal{O}_{\spec(R^p)_{\et}}).$$
By \cite[Remark III.3.8]{milne1etale}, we have
$$H^1(\spec(R^p)_{\et},\mathcal{O}_{\spec(R^p)_{\et}})\cong H^1(\spec(R^p),\mathcal{O}_{\spec(R^p)}).$$
The latter group vanishes for any Noetherian ring (playing the role of $R^p$) and any quasi-coherent sheaf
\cite{sercoh}.
\\
If $A=\gm$, then by \cite[Prop. III.4.9]{milne1etale}
$$H^1(\spec(R^p)_{\et},\mathcal{A})\cong \pic(R^p).$$
But the Picard group of any local ring is trivial \cite[Lemma III.4.10]{milne1etale}.
\\
\\
Let us take now an arbitrary commutative algebraic group $A$ (with a split maximal torus) and assume that $\dim(R)=1$ (note that then $R$ and $\widehat{R}^p$ are PID). By Chevalley's theorem (see \cite{chevsplit} and \cite{conrad}), there is an exact sequence of algebraic groups over $C$
\begin{equation*}
  \xymatrix{0\ar[r]^{} & N \ar[r]^{\iota} &A \ar[r]^{\pi} & H\ar[r]^{} & 0,}
\end{equation*}
where $N$ is affine and $H$ is an Abelian variety. Let us take $x$ and $y$ as in the assumptions of the $(*)$-property. Since $H$ is projective and $R,\widehat{R}^p$ are PID, by \cite[Theorem II.7.1]{ha} we get $\pi(y)\in H(R)$ and $\pi(x-y)\in H(\widehat{R}^p)$. Since $\pi$ is smooth and $N$ is affine, we can use the fact that an appropriate cohomology group is trivial (as in the affine case above) to obtain $v\in A(R)$ and $w\in A(\widehat{R}^p)$ such that $\pi(y)=\pi(v)$ and $\pi(x-y)=\pi(w)$. By the exactness of our sequence again, there are $y_N\in N(K)$ and $w_N\in N(L^p)$ such that
$$\iota(y_N)=y-v,\ \ \iota(w_N)=x-y-w.$$
Let $x_N:=w_N+y_N$. Since $\iota(x_N)=x-v-w\in A(\widehat{R})$ and $\iota$ is a closed embedding, we have $x_N\in N(\widehat{R})$. Clearly, $x_N-y_N=w_N\in N(L^p)$. Since $N$ is affine, it satisfies the $(*)$-property, so there is $z_N\in N(R)$ such that $x_N-z_N\in N(\widehat{R}^p)$. We can define now $z:=v+\iota(z_N)\in A(R)$. Then we have
$$x-z=x-v-\iota(z_N)=\iota(x_N)+w-\iota(z_N)\in A(\widehat{R}^p),$$
since $\iota(x_N)-\iota(z_N)\in A(\widehat{R}^p)$ (by the choice of $z_N$) and $w\in A(\widehat{R}^p)$.
\end{proof}
\begin{conj}
We believe that the result above is true for an arbitrary $A$ and for $R$ of an arbitrary dimension, but we do not know how to show it.
\end{conj}

\subsection{Special formal maps and the $\Pi$-functor}\label{secspecial}
In this subsection we introduce the notion of a special formal map in arbitrary characteristic, generalizing Definition \ref{defspecial}. We also prove Proposition \ref{redtocomm}. The notion of a special map is supposed to capture the properties of formal homomorphisms which are necessary for Schanuel type results. We will need the notions of higher derivations, higher tangent spaces and logarithmic derivative.

\subsubsection{Higher tangent spaces}\label{logdersec}
In this part we will focus on the Weil restriction from $C[\varepsilon]$ into $C$ (recall that $C[\varepsilon]:=C[Y_1]/(Y_1^2)\times_C \ldots \times_C C[Y_t]/(Y_t^2)$). We define a similar functor as in Section \ref{weilgeneralsec}.
$$\Pi_{C[\varepsilon]}:\mathrm{Sch}^{\mathrm{f}}_C\to \mathrm{Sch}^{\mathrm{f}}_C$$
such that for any $V\in \mathrm{Sch}^{\mathrm{f}}_C$ and any $C$-algebra $W$, there is a natural bijection
$$\Pi_{C[\varepsilon]}V(W)\to V(W[\varepsilon]).$$
It is easy to see that
$$\Pi_{C[\varepsilon]}V\cong T^{\times t}V,$$
where $T^{\times t}V$ is the fibered (over $V$) $t$-th Cartesian power of the tangent bundle of $V$ (we will not need this observation).
\\
We need now to generalize derivations ($m=1$) to higher derivations (arbitrary $m>0$). This is rather straightforward, we just replace the ring $C[\varepsilon]$ with the ring $C[\varepsilon^{(m)}]$, where
$$C[\varepsilon^{(m)}]:=C[Y_1]/(Y_1^{p^m})\times_C \ldots \times_C C[Y_t]/(Y_t^{p^m}).$$
\begin{remark}
For notational reasons we prefer to consider derivations of order $p^m-1$ rather than $m$. In particular $C[\varepsilon^{(1)}]\neq C[\varepsilon]$.
However, since any derivation $\partial$ gives a derivation of order $p-1$ using the formula $(\partial_i=\partial^{(i)}/i!)_{i<p}$, by considering derivations of order $p-1$ we also cover the case of usual derivations and we may work with rings of the form $C[\varepsilon^{(m)}]$.
\end{remark}
\noindent
We again get a functor
$$\Pi_{C[\varepsilon^{(m)}]}:\mathrm{Sch}_C\to \mathrm{Sch}_C.$$
It is easy to see that
$$\Pi_{C[\varepsilon^{(m)}]}V\cong (\arc^{(m)})^{\times t}V,$$
where the $\arc^{(m)}$ is the $(p^m-1)$-th arc space functor (see \cite{DeLo2000}) and we consider the $t$-th fibered Cartesian power over $V$.
\\
If $V=\spec(W)$, then we have
$$\Pi_{C[\varepsilon^{(m)}]}V=\spec(\hs^{(m)}_W).$$
Here
$$\hs^{(m)}_W:=\hs^{p^m-1}_{W/C}\otimes_W\ldots \otimes_W \hs^{p^m-1}_{W/C},$$
where $\hs^{p^m-1}_{W/C}$ is the space of \emph{higher Vojta's forms} of order $p^m-1$ (see \cite{voj}) and the tensor product is taken $t$ times. We will use the fact (following from \cite[Corollary 1.8]{voj}) that the functor $W\mapsto \hs^{(m)}_W$ is left-adjoint to the functor $W\mapsto W[\varepsilon^{(m)}]$. In particular, any $t$-tuple of higher derivations $\partial$ of order $p^m-1$ on $W$ corresponds to a $W$-linear map
$$\partial^*:\hs^{(m)}_W\to W.$$
\begin{remark}\label{derexist}
By \cite[Prop. 5.1]{voj} and \cite[Lemma 5.7]{voj}, we have
$$\hs^{(m)}_{\mathbf{k}}\cong \mathbf{k}[X_1^{(m)},\ldots,X_t^{(m)}],$$
where each $X_i^{(m)}$ is a ($p^m-1$)-tuple of variables corresponding to $(\ddd_1t_i,\ldots,\ddd_{p^m-1}t_i)$. Thus there is a canonical tuple of ``partial'' higher derivations $\partial^{(m)}_{\mathbf{k}}$ of order $p^m-1$ on $\mathbf{k}$. Assume that $W$ is a $C$-algebra as in Lemma \ref{p1}. Then
we have
$$W\cong W^{p^m}\otimes_{\mathbf{k}^{p^m}}\mathbf{k}.$$
By \cite[Lemma 5.5]{voj} we get
$$\hs^{(m)}_{W}\cong W[X_1^{(m)},\ldots,X_t^{(m)}].$$
Hence there is also a canonical tuple of higher derivations $\partial^{(m)}_W$ of order $p^m-1$ on $W$. They can be explicitly defined as
$$\partial^{(m)}_W:=\id_{W^{p^m}}\otimes_{\mathbf{k}^{p^m}}\partial^{(m)}_{\mathbf{k}}.$$
\end{remark}
\noindent
For an algebraic group $G$, we focus on the following kernel
$$U^{(m)}_G:=\ker(\Pi_{C[\varepsilon^{(m)}]}G\to G).$$
From now on we denote $\mathcal{O}_{G,e}$ by $\mathcal{O}_{G}$. We define the following $G$-algebras
$$U^{(m)}_{\mathcal{O}_G}:=\hs^{(m)}_{\mathcal{O}_G}\otimes_{\mathcal{O}_G}C,\  \ \ \ \ U^{(m)}_{\widehat{\mathcal{O}}_G}:=\hs^{(m)}_{\widehat{\mathcal{O}}_G}\otimes_{\widehat{\mathcal{O}}_G}C.$$
\begin{remark}\label{uhopf}
Each of the $C$-algebras above has a natural structure of a Hopf algebra over $C$. We will see it in the case of $U^{(m)}_{\mathcal{O}_G}$ (see the lemma below for the other $C$-algebra). The multiplication morphism on $G$ induces a $C$-algebra map
$$\mu:\mathcal{O}_G\to \mathcal{O}_{G\times G}\cong (\mathcal{O}_G\otimes_C\mathcal{O}_G)_I,$$
where $I$ is the ideal generated by $\mathfrak{m}_{\mathcal{O}_G}\otimes 1+1\otimes \mathfrak{m}_{\mathcal{O}_G}$. Hence we get a map
$$\hs^{(m)}_{\mu}:\hs^{(m)}_{\mathcal{O}_G}\to \hs^{(m)}_{\mathcal{O}_{G\times G}}\cong (\hs^{(m)}_{\mathcal{O}_G\otimes_C\mathcal{O}_G})_I,$$
where the last isomorphism follows from \cite[Lemma 4.2]{voj}. Finally, we get a map
$$U^{(m)}_{\mu}:U^{(m)}_{\mathcal{O}_G}\to
(\hs^{(m)}_{\mathcal{O}_G\otimes_C\mathcal{O}_G})_I\otimes_{\mathcal{O}_{G\times G}}C
\cong U^{(m)}_{\mathcal{O}_G}\otimes_C U^{(m)}_{\mathcal{O}_G}.$$
(using \cite[Lemma 4.2]{voj} again and \cite[Lemma 5.7]{voj}) which is a Hopf algebra comultiplication.
\end{remark}
\noindent
We need the following.
\begin{lemma}\label{unip}
\begin{enumerate}
\item We have natural isomorphisms
$$\spec(U^{(m)}_{\widehat{\mathcal{O}}_G})\cong \spec(U^{(m)}_{\mathcal{O}_G})\cong U^{(m)}_G.$$

\item If $G=A$ is commutative, then there are natural ``Hopf algebra'' splitting maps
$$\Pi_{C[\varepsilon^{(m)}]}A\to U^{(m)}_A,\ \ U^{(m)}_{\mathcal{O}_A}\to \hs^{(m)}_{\mathcal{O}_A},\ \ U^{(m)}_{\widehat{\mathcal{O}}_A}\to \hs^{(m)}_{\widehat{\mathcal{O}}_A}.$$

\item If $G=A$ is commutative, then there is a commutative diagram
\begin{equation*}
  \xymatrix{U^{(m)}_{\mathcal{O}_A} \ar[d]_{}\ar[r]^{} & \hs^{(m)}_{\mathcal{O}_A} \ar[d]_{\partial_0^*}\\
C  \ar[r]^{}& \mathcal{O}_A
.}
\end{equation*}

\item For each formal subgroup $\mathcal{A}\leqslant \widehat{G}$, $U^{(m)}_{\mathcal{A}}$ is an algebraic subgroup of $U^{(m)}_G$.

\item For each formal subgroups $\mathcal{A},\mathcal{B}\leqslant \widehat{G}$, we have $\mathcal{A}=\mathcal{B}$ if and only if for each $m$, we have $U^{(m)}_{\mathcal{A}}=U^{(m)}_{\mathcal{B}}$.

\end{enumerate}
\end{lemma}
\begin{proof}
By Remark \ref{uhopf}, we get
$$\hs^{(m)}_{\mathcal{O}_G}\cong \mathcal{O}_G[\bar{X}],\ \ \ \hs^{(m)}_{\widehat{\mathcal{O}}_G}\cong \widehat{\mathcal{O}}_G[\bar{X}].$$
Hence we get the first isomorphism in $(1)$. For the second one, we notice first that for any open neighbourhood $V$ of $0$ in $G$, we have $U^{(m)}_G\cong U^{(m)}_V$. Now it is enough to pass to the limit with respect to affine open neighbourhoods of $0$ in $G$.

For $(2)$, we just want to follow (and dualize) the elementary argument regarding commutative groups. If $A_1=\ker(f:A_2\to A_3)$, then any section $s:A_3\to A_2$ gives a section $A_2\to A_1$, which comes from a map $A_2\to A_2$ given by $s\circ f -\id$. There is no problem to write this argument in terms of diagrams which gives the first splitting map. In the Hopf algebra case, there is a little problem, since the $C$-algebras $\hs^{(m)}_{\mathcal{O}_A}, \hs^{(m)}_{\widehat{\mathcal{O}}_A}$ are not exactly Hopf algebras. Nevertheless, we will show below (just in the case of  $\hs^{(m)}_{\mathcal{O}_A}$) that the splitting argument still works. The corresponding map $\hs^{(m)}_{\mathcal{O}_A}\to \hs^{(m)}_{\mathcal{O}_A}$ is the composition of the sequence below:
\begin{equation*}
\xymatrix{\hs^{(m)}_{\mathcal{O}_A} \ar[rr]^{\hs_{\mu}} & &\hs^{(m)}_{\mathcal{O}_{A\times A}} \ar[rr]^{((\pi\circ s)\otimes \id)_I} & &
\hs^{(m)}_{\mathcal{O}_{A\times A}} \ar[rr]^{(\cdot)_I} & & \hs^{(m)}_{\mathcal{O}_A}
,}
\end{equation*}
where we use the fact that $\hs^{(m)}_{\mathcal{O}_{A\times A}}$ is isomorphic to the localization of $\hs^{(m)}_{\mathcal{O}_A}\otimes \hs^{(m)}_{\mathcal{O}_A}$ with respect to the ideal $I$ (as in Remark \ref{uhopf}). From the construction, the splitting map is natural with respect to group scheme morphisms $A\to B$ (and respectively formal group morphisms $\widehat{A}\to \widehat{B}$).

Item $(3)$ follows from the construction of the splitting map.

Item $(4)$ follows from item $(1)$, and item $(5)$ follows by inspection.
\end{proof}
We can prove now Proposition \ref{redtocomm}, which was needed for the reduction of Question \ref{mainque} to the commutative case.
\begin{proof}[Proof of Prop. \ref{redtocomm}]
By Proposition \ref{algform}, we can assume that $\mathcal{A}$ is Zariski dense in $G$. We will show that we can take $H=G$. Let $m>0$.

The algebraic group $G$ acts on $G[m]=\ker(\fr^m_G)$ by conjugation which is a group scheme action. Let $\mathcal{H}$ be the complete Hopf algebra corresponding to $\mathcal{A}$. By Proposition \ref{frker}, $\mathcal{H}[m]$ is a Hopf algebra and a complete Hopf algebra. Let $\mathcal{A}[m]$ be the finite group scheme (being also a formal group) corresponding to $\mathcal{H}[m]$ and $H_m$ be the ``setwise'' stabilizer of $\mathcal{A}[m]$ with respect to the above action. Then $H_m$ is a group scheme such that $\mathcal{A}\subseteq \widehat{H_m}$, so by the density of $\mathcal{A}$ in $G$, we get $H_m=G$ for all $m$. Therefore $\mathcal{A}$ is a normal formal subgroup of $\widehat{G}$.

Let $G'$ be the commutator subgroup of $G$. We need to show that $\widehat{G'}\leqslant \mathcal{A}$. We do it in three steps, following the idea of Ax's argument from \cite{ax72}.

We consider an algebraic group action $\mathrm{ad}^{(m)}$ of $G$ on $U^{(m)}_G$ by conjugation (``higher adjoint action''). By Lemma \ref{unip}(4), $U^{(m)}_{\mathcal{A}}$ is an algebraic subgroup of $U^{(m)}_{G}$, hence the following group
$$W:=\left\{x\in G\ |\ \mathrm{ad}^{(m)}_x=\id_{U^{(m)}_G} \ \left(\mathrm{mod}\ U^{(m)}_{\mathcal{A}}\right)\right\}$$
is an algebraic subgroup of $G$. We also have $\mathcal{A}\leqslant \widehat{W}$, since $\mathcal{A}$ is normal in $\widehat{G}$. Hence $W=G$.

It is easy to see that $\widehat{G'}\leqslant \mathcal{A}$ if and only if for each $x\in G$ we have
$$\alpha_x=\id_{\widehat{G}}\ (\mathrm{mod}\ \mathcal{A}),$$
where $\alpha_x$ is the conjugation automorphism.

It is enough to notice now that
$$\forall x\ \ \alpha_x=\id_{\widehat{G}}\ (\mathrm{mod}\ \mathcal{A})$$
if and only if
$$\forall  m\ \ \mathrm{ad}^{(m)}_x=\id_{U^{(m)}_G} \ \left(\mathrm{mod}\ U^{(m)}_{\mathcal{A}}\right),$$
and this equivalence follows from Lemma \ref{unip}(5).
\end{proof}
\begin{remark}
The real difference between the above proof and Ax's proof from \cite{ax72} is in the last step. In the case of characteristic $0$, one needs to check only the case of $m=1$, so it is enough to look at the Lie algebra of $G$ and the (standard) adjoint action. However in the case of positive characteristic, we need to consider all positive integers $m$, since Lemma \ref{unip}(5) does not hold just on the level of Lie algebras (i.e. on $U^{(1)}_G$).
\end{remark}


\subsubsection{Special maps}\label{pitansec}
We will extend (in an obvious way) the results and definitions of Section \ref{secweilsp} from the case of $\mathbf{k}^p$ to the case of $\mathbf{k}^{p^m}$ for an arbitrary $m>0$.
\begin{itemize}
\item By $\Pi^{(m)}$ we mean the functor $\Pi$ but for $\mathbf{k}^{p^m}$ playing the role of $\mathbf{k}^p$.

\item Similarly, we get a morphism
$$f^{(m)}_A:\Pi^{(m)}_{\mathbf{k}}(A\times_C\mathbf{k}^{p^m})\to \Pi^{(m)}_{\mathbf{k}[\varepsilon^{(m)}]}(A\times_C\mathbf{k}^{p^m})$$
such that $f^{(1)}_A$ coincides with $f_A$ from Proposition \ref{weilneeded1}.

\item We have obvious higher order analogues of Lemma \ref{p1} and Proposition \ref{weilneeded1}.
\end{itemize}
\begin{remark}\label{hoan}
Let us specify what do we mean by the ``higher order analogue'' of Proposition \ref{weilneeded1}(3) from the last item above. It means that
$$\left(\Pi^{(m)}_{\mathbf{k}[\varepsilon^{(m)}]}(A\times_C\mathbf{k}^{p^m})\to A\right)\circ f^{(m)}_A=0,$$
for the functor $\Pi^{(m)}$ and the map $f^{(m)}_A$ defined as above.
\end{remark}
\noindent
We will often just write ``$A$'' in place of ``$A\times_C\mathbf{k}^{p^m}$''. We need one more lemma.
\begin{lemma}\label{ukernel}
Let $W$ be as in Lemma \ref{p1}. There is a commutative diagram
\begin{equation*}
  \xymatrix{U^{(m)}_A(W)\ar[rr]^{\ker(\pi)}\ar[d]^{\cong}  & & A(W[\varepsilon^{(m)}])\ar[d]^{} \ar[d]^{\cong}\ar[rr]^{\pi}& &A(W) \ar[d]^{\cong}\\
  \Pi^{(m)}_{\mathbf{k}}U^{(m)}_A(W^{p^m})\ar[rr]^{\ker(\widetilde{\pi})}  & & \Pi^{(m)}_{k[\varepsilon^{(m)}]}A(W^{p^m}) \ar[rr]^{\widetilde{\pi}}& &\Pi^{(m)}_{\mathbf{k}}A(W^{p^m}),
}
\end{equation*}
where $\widetilde{\pi}=\Pi^{(m)}_{\mathbf{k}\to \mathbf{k}[\varepsilon^{(m)}]}$ and $\pi$ is induced by $W[\varepsilon^{(m)}]\to W$.
\end{lemma}
\begin{proof} It is enough to notice a natural isomorphism
$$\Pi^{(m)}_{k[\varepsilon^{(m)}]}A\cong \Pi^{(m)}_{k[\varepsilon^{(m)}]}(\Pi_{C[\varepsilon^{(m)}]}A)$$
and use the definition of $U^{(m)}_A$.
\end{proof}
\noindent
We define below our ``well behaved'' class of formal maps.
\begin{definition}\label{defspecialagain}
We say that a formal map $\mathcal{F}:\widehat{V}\to \widehat{A}$ is \emph{special} if for each $m>0$ we have
$f^{(m)}_A(v_m)\in \Pi^{(m)}_{\mathbf{k}[\varepsilon^{(m)}]}A(R^{p^m})$, where $v_m\in \Pi^{(m)}_{\mathbf{k}[\varepsilon^{(m)}]}A(\widehat{R}^{p^m})$ corresponds to $\mathcal{F}_A\in A(\widehat{R})$.
\end{definition}
\noindent
We will see later that the definition above generalizes Definition \ref{defspecial}. First we need to understand it in terms of the logarithmic derivative.
\begin{prop}\label{weiltan}
Suppose $W$ is as in Lemma \ref{p1}. Then for each $m>0$, there is a closed embedding
$$h_A:\Pi^{(m)}_{\mathbf{k}}U^{(m)}_A\to \Pi^{(m)}_{\mathbf{k}[\varepsilon^{(m)}]}A,$$
such that the morphism $f^{(m)}_A$ factors (by $\Psi$) as in the following commutative diagram:
\begin{equation*}
  \xymatrix{A(W^{p^m})\ar[d]^{=} \ar[rr]^{\subseteq}& &A(W) \ar[d]_{\cong}\ar[rr]^{}& & U^{(m)}_A(W)\ar[d]_{\cong}\\
A(W^{p^m}) \ar[rr]^{\iota_A} & &\Pi^{(m)}_{\mathbf{k}}A(W^{p^m}) \ar[rrd]_{f^{(m)}_A} \ar[rr]^{\Psi}& &\Pi^{(m)}_{\mathbf{k}}U^{(m)}_A(W^{p^m})\ar[d]_{h_A}\\
 & & & &\Pi^{(m)}_{\mathbf{k}[\varepsilon^{(m)}]}A(W^{p^m}),}
\end{equation*}
where the map $A(W)\to U^{(m)}_A(W)$ is induced by $\Psi$.
\\
Moreover, if $W$ is local and $v\in A(W)_*$ corresponds to the map $v:\mathcal{O}_A\to W$, then the image of $v$ under the map $A(W)\to U^{(m)}_A(W)$ corresponds to the map
which is given by the following composition
\begin{equation*}
  \xymatrix{U^{(m)}_{\mathcal{O}_{A}} \ar[rr]^{} \ar[rr]^{}& & \hs^{(m)}_{\mathcal{O}_{A}} \ar[rr]^{\hs^{(m)}_{v}} & &\hs^{(m)}_{W} \ar[rr]^{(\partial_W^{(m)})^*} & & W,
}
\end{equation*}
where $\partial^{(m)}_W$ comes from Remark \ref{derexist}.
\end{prop}
\begin{proof}
Using Lemma \ref{ukernel} and a higher order analogue of Proposition \ref{weilneeded1}(3) (see Remark \ref{hoan}), we observe that the morphism $f^{(m)}_A$ factors through a closed embedding
$$\Psi:\Pi^{(m)}_{\mathbf{k}}A\to \ker(\Pi^{(m)}_{k[\varepsilon^{(m)}]}A \to A)\cong \Pi^{(m)}_{\mathbf{k}}U^{(m)}_A.$$
Hence we can take $h_A$ as the kernel of the morphism $\Pi^{(m)}_{k[\varepsilon^{(m)}]}A \to A$.
\\
For the moreover part, by Remark \ref{derexist} we see first that the following diagram is commutative:
\begin{equation*}
  \xymatrix{A(W) \ar[d]_{\cong}\ar[rr]^{\partial_W-\partial_0}& & A(W[\varepsilon^{(m)}])\ar[d]_{\cong}\\
\Pi^{(m)}_{\mathbf{k}}A(W^{p^m})  \ar[rr]^{f^{(m)}_A}& &\Pi^{(m)}_{\mathbf{k}[\varepsilon^{(m)}]}A(W^{p^m})
.}
\end{equation*}
From Lemma \ref{ukernel} again, the upper arrow factors through
$$\ker(A(W[\varepsilon^{(m)}])\to A(W))=U^{(m)}_A(W).$$
For $v:\mathcal{O}_A\to W$ corresponding to an element of $A(W)_*$, we need to understand the image of $v$ under the map $A(W)\to U^{(m)}_A(W)$.
The image of $v$ by $\partial_W-\partial_0$ corresponds to the map
$$\partial_W^*\circ \hs^{(m)}_v-\partial_0^*\circ \hs^{(m)}_v:\hs^{(m)}_{\mathcal{O}_A}\to W$$
restricted to $U^{(m)}_{\mathcal{O}_A}$ (see Lemma \ref{unip}(1)). It is enough to see that the map above coincides with $\partial_W^*\circ \hs^{(m)}_v$ after restricting to $U^{(m)}_{\mathcal{O}_A}$. In other words, we need to see that the map $\partial_0^*\circ \hs^{(m)}_v$ is the $0$-map after restricting to $U^{(m)}_{\mathcal{O}_A}$. It is enough to consider the following diagram
\begin{equation*}
  \xymatrix{U^{(m)}_{\mathcal{O}_A}\ar[rr]^{}\ar[d]^{}  & &\hs^{(m)}_{\mathcal{O}_A}\ar[d]^{(\partial^{\mathcal{O}_A}_0)^*} \ar[rr]^{\hs^{(m)}_v} & &\hs^{(m)}_W \ar[d]^{(\partial_0^W)^*}\\
  C\ar[rr]^{}   &  & \mathcal{O}_A \ar[rr]^{v}& & W,
}
\end{equation*}
which is commutative by Lemma \ref{unip}(3) and the naturality of the $0$-derivation.
\end{proof}
\begin{definition}
We denote the map $A(W)\to U^{(m)}_A(W)$ (from the proposition above) by $l_A\partial^{(m)}_W$.
\end{definition}

\begin{remark}\label{specequiv}
\begin{enumerate}
\item The map $l_A\partial^{(m)}_W$ coincides with the \emph{logarithmic derivative} map with respect to $\partial^{(m)}_W$ (for a positive characteristic version of the logarithmic derivative map, see e.g. \cite{JaMo}).

\item Let $\mathcal{F}:\widehat{V}\to \widehat{A}$ be a formal map. The following are equivalent:
\begin{enumerate}
\item for each $m$, we have $l_A\partial^{(m)}_{\widehat{R}}(\mathcal{F}_A)\in U^{(m)}_A(R)$;

\item the formal map $\mathcal{F}$ is special.
\end{enumerate}

\item The condition (a) above is equivalent (even for any formal group $\mathcal{A}$ in place of $\widehat{A}$) to the following
$$\mathcal{F}^*(\hs^{\inv}_A)\subseteq \hs_R.$$
The reason is that the (properly defined) ring of higher invariant forms $\hs^{\inv}_A$ coincides with $U_{\mathcal{O}_A}$.

\item We also have $\Omega^{\inv}_A\subseteq \hs^{\inv}_A$ (such issues will be discussed in the forthcoming paper \cite{K7}). Hence the notion of a special formal map from this section generalizes Definition \ref{defspecial}.

\end{enumerate}
\end{remark}
\noindent
We show below that the formal maps we are interested in are special. In the most important case of formal homomorphisms, the main reason is (as I see it) that the ``higher Lie algebra'' of an algebraic group (that is $U^{(m)}_{\mathcal{O}_A}$) coincides with the ``higher Lie algebra'' of its formalization (that is $U^{(m)}_{\widehat{\mathcal{O}}_A}$) by Lemma \ref{unip}(1).
\begin{prop}\label{homspecial}
The following classes of formal maps consist of special maps:
\begin{enumerate}
\item Formalizations of algebraic maps.

\item Formal homomorphisms between formalizations of algebraic groups.
\end{enumerate}
\end{prop}
\begin{proof}
The first item is clear, since for each $m>0$, we have $v_m\in \Pi^{(m)}_{\mathbf{k}[\varepsilon^{(m)}]}A(R^{p^m})$ where $v_m$ comes from Definition \ref{defspecialagain}.
\\
Let $\mathcal{F}:\widehat{B}\to \widehat{A}$ be a formal homomorphism which we identify with a $C$-algebra map $\mathcal{F}:\widehat{\mathcal{O}}_{A}\to \widehat{\mathcal{O}}_{B}$. Let $\partial$ be the tuple of derivations on $\mathcal{O}_B$ from Remark \ref{derexist} and  $\widehat{\partial}$ its natural extension to $\widehat{\mathcal{O}}_B$.
To prove the second item, for any $m>0$ let us consider the following commutative diagram below
(given by Proposition \ref{weiltan} and Lemma \ref{unip}(1))
\begin{equation*}
  \xymatrix{
\hs^{(m)}_{\widehat{\mathcal{O}}_A} \ar[rr]^{\hs^{(m)}_{\mathcal{F}}} & &  \hs^{(m)}_{\widehat{\mathcal{O}}_B} \ar[d]^{\widehat{\partial}_B^*} & \\
U^{(m)}_{\widehat{\mathcal{O}}_A} \ar[u]_{} \ar[rr]^{l_A\partial^{(m)}_{\widehat{R}}(\mathcal{F}_A)}  & &   \widehat{\mathcal{O}}_B & & \hs^{(m)}_{\mathcal{O}_B}\ar[llu]_{} \ar[d]^{\partial^*_B}\\
 & &  & & \mathcal{O}_B  \ar[llu]_{} .}
\end{equation*}
\noindent
 By Remark \ref{specequiv}(2), it is enough to find a map $\Psi:U^{(m)}_{\widehat{\mathcal{O}}_A}\to \mathcal{O}_B$ which completes the commutative diagram above. We will find this map using the following commutative diagram of $C$-algebra maps
\begin{equation*}
  \xymatrix{
\hs^{(m)}_{\widehat{\mathcal{O}}_A} \ar[rr]^{\hs^{(m)}_{\mathcal{F}}} & & \hs^{(m)}_{\widehat{\mathcal{O}}_B} \ar[rr]^{\widehat{\partial}^*} \ar[d]^{\widehat{\partial}^*} & & \widehat{\mathcal{O}}_B & \\
U^{(m)}_{\widehat{\mathcal{O}}_A} \ar[u]^{} \ar[rr]^{U^{(m)}_{\mathcal{F}}} & & U^{(m)}_{\widehat{\mathcal{O}}_B} \ar[u]^{} & \hs^{(m)}_{\mathcal{O}_B} \ar[rr]^{\partial^*}\ar[lu]^{} & &   \mathcal{O}_B \ar[lu]^{}\\
 & & & U^{(m)}_{\mathcal{O}_B} \ar[u]^{\iota}\ar[lu]^{f}. & & }
\end{equation*}
\noindent
The fact that $\mathcal{F}$ is a formal homomorphism is used only for the commutativity of the left-upper square above, see Lemma \ref{unip}(2).
By Lemma \ref{unip}(1), $f$ is an isomorphism. We define
$$\Psi := \partial^*\circ \iota \circ f^{-1}\circ \mathcal{U}_{\mathcal{F}},$$
which completes the first diagram.
\end{proof}
\noindent%
For the inductive step of the proof of the main result (Theorem \ref{alimitthm}) of this section, we will need one more property of special formal maps which is analogous to Proposition \ref{frobfactorbetter}.
\begin{prop}\label{specialroot}
Let $\mathcal{F}:\widehat{V}\to \widehat{A}$ be a special formal map and assume that there is a formal map $\mathcal{F}':\widehat{V}^{\fr}\to \widehat{A}$ such that the following diagram commutes
\begin{equation*}
  \xymatrix{
\widehat{V}\ar[rr]^{\mathcal{F}}\ar[d]^{\fr} & & \widehat{A} \\
\widehat{V}^{\fr}\ar[rru]_{\mathcal{F}'}. & & }
\end{equation*}
Then $\mathcal{F}'$ is special.
\end{prop}
\begin{proof}
We consider $R$ as a $\mathbf{k}$-algebra by a choice of a $p$-basis of $R$. Such a $\mathbf{k}$-algebra will be denoted by $R_{\mathbf{k}}$. Similarly for any $m>0$, we have the $\mathbf{k}^{p^m}$-algebra $R^{p^m}_{\mathbf{k}^{p^m}}$. We also consider $R^p$ as a $\mathbf{k}$-algebra twisting the $\mathbf{k}^p$-algebra $R^p_{\mathbf{k}^p}$ by $\fr_{\mathbf{k}}$. We denote the last $\mathbf{k}$-algebra by $R^{p}_{\mathbf{k}}$. Clearly, the map $\fr_R:R_{\mathbf{k}}\to R^p_{\mathbf{k}}$ is an isomorphism of $\mathbf{k}$-algebras. Similarly, we also have the $\mathbf{k}^{p^{m-1}}$-algebra  $R^{p^m}_{\mathbf{k}^{p^{m-1}}}$. The same notation will apply to $\widehat{R}$ in place of $R$.
\\
For each $m>0$, consider the map
$$f_{\mathbf{k}/\mathbf{k}^{p^m}}:=f^{(m)}_A:\Pi_{\mathbf{k}/\mathbf{k}^{p^m}}A\to \Pi_{\mathbf{k}[\varepsilon^{(m)}]/\mathbf{k}^{p^m}}A.$$
Since $\mathcal{F}$ is special, we know that
$$f_{\mathbf{k}/\mathbf{k}^{p^m}}(\mathcal{F}_A^{(m)})\in \Pi_{\mathbf{k}[\varepsilon^{(m)}]/\mathbf{k}^{p^m}}A(R^{p^m}_{\mathbf{k}^{p^m}}),$$
where $\mathcal{F}_A^{(m)}\in \Pi_{\mathbf{k}/\mathbf{k}^{p^m}}A(R^{p^m}_{\mathbf{k}^{p^m}})$ corresponds to $\mathcal{F}_A\in A(R_{\mathbf{k}})$
under the natural bijection.
\\
By our assumption, $\mathcal{F}_A\in A(R^p_{\mathbf{k}^p})$. We have a natural isomorphism coming from $\fr_R$
$$A(R_{\mathbf{k}})\cong A(R^p_{\mathbf{k}})$$
and $\mathcal{F}_A$ corresponds to $\mathcal{F}'_A$ under this isomorphism. We need to show that
\begin{equation}
f_{\mathbf{k}/\mathbf{k}^{p^{m-1}}}((\mathcal{F}'_A)^{(m-1)})\in \Pi_{\mathbf{k}[\varepsilon^{(m-1)}]/\mathbf{k}^{p^{m-1}}}A(R^{p^m}_{\mathbf{k}^{p^{m-1}}}).\tag{$\clubsuit$}
\end{equation}
Let us consider the following commutative diagram
\begin{equation*}
  \xymatrix{
A(\widehat{R}^p_{\mathbf{k}}) \ar[d]^{\cong} \ar[rr]^{\cong} & & A(\widehat{R}^p_{\mathbf{k}^p})\ar[d]^{\cong} \ar[rr]^{\subseteq}
& & A(\widehat{R}_{\mathbf{k}}) \ar[d]^{\cong} \\
\Pi_{\mathbf{k}/\mathbf{k}^{p^{m-1}}}A(\widehat{R}^{p^m}_{\mathbf{k}^{p^{m-1}}}) \ar[d]^{f_{\mathbf{k}/\mathbf{k}^{p^{m-1}}}} \ar[rr]^{\cong}
& & \Pi_{\mathbf{k}^p/\mathbf{k}^{p^{m}}}A(\widehat{R}^{p^m}_{\mathbf{k}^{p^{m}}})\ar[d]^{f_{\mathbf{k}^p/\mathbf{k}^{p^{m}}}} \ar[rr]^{\hookrightarrow}
& & \Pi_{\mathbf{k}/\mathbf{k}^{p^{m}}}A(\widehat{R}^{p^m}_{\mathbf{k}^{p^{m}}})\ar[d]^{f_{\mathbf{k}/\mathbf{k}^{p^{m}}}} \\
\Pi_{\mathbf{k}[\varepsilon^{(m-1)}]/\mathbf{k}^{p^{m-1}}}A(\widehat{R}^{p^m}_{\mathbf{k}^{p^{m-1}}})  \ar[rr]^{\cong}
& & \Pi_{\mathbf{k}^p[\varepsilon^{(m-1)}]/\mathbf{k}^{p^{m}}}A(\widehat{R}^{p^m}_{\mathbf{k}^{p^{m}}}) \ar[rr]^{\hookrightarrow}
& & \Pi_{\mathbf{k}[\varepsilon^{(m-1)}]/\mathbf{k}^{p^{m}}}A(\widehat{R}^{p^m}_{\mathbf{k}^{p^{m}}}) \\
\Pi_{\mathbf{k}[\varepsilon^{(m-1)}]/\mathbf{k}^{p^{m-1}}}A(R^{p^m}_{\mathbf{k}^{p^{m-1}}}) \ar[u]^{\subseteq} \ar[rr]^{\cong}
& & \Pi_{\mathbf{k}^p[\varepsilon^{(m-1)}]/\mathbf{k}^{p^{m}}}A(R^{p^m}_{\mathbf{k}^{p^{m}}}) \ar[u]^{\subseteq} \ar[rr]^{\hookrightarrow}
& & \Pi_{\mathbf{k}[\varepsilon^{(m-1)}]/\mathbf{k}^{p^{m}}}A(R^{p^m}_{\mathbf{k}^{p^{m}}}) \ar[u]^{\subseteq}, \\
}
\end{equation*}
where each $\hookrightarrow$ denotes a closed embedding. We obtain $(\clubsuit)$ chasing the diagram above.
\end{proof}

\subsection{Special maps and $A$-limits}
In this part we prove below the main result of Section \ref{secspeciallimit} which is Theorem \ref{alimitthm}. We notice first an obvious consequence of Proposition \ref{homspecial}(1).
\begin{prop}\label{alimisspe}
Any $A$-limit map is special.
\end{prop}
\begin{proof}
Let us fix a positive integer $m$ and assume that $\mathcal{F}:\widehat{V}\to \widehat{A}$ is an $A$-limit formal map. Let $\varphi:\mathcal{O}_A\to R$ be such that $\mathcal{F}_A-\varphi_A\in A((\widehat{R})^{p^m})$ (see Remark \ref{defalimit2}). Let
$$v_m\in  \Pi^{(m)}_{\mathbf{k}[\varepsilon^{(m)}]}A(R^{p^m})$$
correspond to $\mathcal{F}_A$ and similarly for $w_m$ and $\varphi_A$. By Proposition \ref{homspecial}(1), we have
$$f_A^{(m)}(w_m)\in \Pi^{(m)}_{\mathbf{k}[\varepsilon^{(m)}]}A(R^{p^m}).$$
Since $\mathcal{F}_A-\varphi_A\in A((\widehat{R})^{p^m})$, by the ``degree $m$ version of Proposition \ref{weilneeded1}(1)'', we have $f_A^{(m)}(v_m-w_m)=0$. Hence we get
$$f_A^{(m)}(\mathcal{F}_A)\in \Pi^{(m)}_{\mathbf{k}[\varepsilon^{(m)}]}A(R^{p^m}),$$
so $\mathcal{F}$ is special.
\end{proof}

\begin{theorem}\label{alimitthm}
Assume that
\begin{itemize}
\item The maximal torus of $A$ splits;

\item $\dim V=1$ or $A$ is affine;

\item For any tower of fields $C\subseteq K_1\subseteq K_2$, if $K_1\subseteq K_2$ is purely inseparable, then the induced map $H^1(K_1,A)\to H^1(K_2,A)$ has a trivial kernel.
\end{itemize}
\noindent
Let $\mathcal{F}:\widehat{V}\to \widehat{A}$ be a special formal map. Then $\mathcal{F}$ is an $A$-limit map.
\end{theorem}
\noindent
The proof is divided into four steps.
\\
{\bf Step 1} There is $a\in A(K)$ such that $\mathcal{F}_A-a\in A(L^p)$.

Let $\mathcal{F}_A'\in \Pi_{\mathbf{k}}A(\widehat{R}^p)$ be the image of $\mathcal{F}_A\in A(\widehat{R}^p)$ under the bijection from Lemma \ref{p1}(2). In this step we just consider $\mathcal{F}_A'$ as an $L^p$-rational point. Let $y:=f_A(\mathcal{F}_A')$. By the definition of a special map, $y\in \Pi_{\mathbf{k}[\varepsilon]}A(K^p)$. To complete the proof of Step 1, it is enough to find $x\in \Pi_{\mathbf{k}}(A)(K^p)$ such that $f_A(x)=y$.
\\
Let $P:=f_A^{-1}(y)$ be the schematic fiber (a scheme over $K^p$). As in the proof of Proposition \ref{star}, $P$ is an PHS of $A$ over $K^p$.
It is enough to show that $P$ has a $K^p$-rational point, i.e. that $P$ corresponds to the zero element of $H^1(K^p,A)$ (see again \cite[Remark 4.8(a)]{milne1etale} for the necessary identifications). By our trivial kernel assumption, it is enough to show that
$$H^1(K^p,A)\ni [P]\mapsto 0\in H^1(M,A),$$
where $M=K^{p^{-\infty}}$.
\\
By Theorem \ref{weilneeded}(3,4), the morphism $\iota:A\times_CK\to \Pi_{\mathbf{k}}(A)\times_{{\mathbf{k}}^p}K$ has a section $s$. Let us consider
$$P_s:=\ker(s)\cap (P\times_{K^p}K),$$
which is a scheme over $K$. Let $\bar{M}$ denote the algebraic closure of $M$. Since
$$P_s(\bar{M})=\ker(s)(\bar{M})\cap P(\bar{M}),$$
the set $P_s(\bar{M})$ is a singleton $\{x'\}$. On the other hand, $P_s(\bar{M})$ is invariant under the action of the absolute Galois group of $M$, so $x'\in P(M)$. Hence indeed the class of $P$ is mapped to $0$ in the group $H^1(M,A)$.
\\
{\bf Step 2} There is $b\in A(R)$ such that $\mathcal{F}_A-b\in A(\widehat{R}^p)$.

This is exactly Proposition \ref{star}.
\\
{\bf Step 3} There is a local $C$-algebra homomorphism $\phi:\mathcal{O}_{A}\to R$ such that
$$(\mathcal{F}-\widehat{\phi})(\widehat{\mathcal{O}}_{A})\subseteq \widehat{R}^p.$$
Let us consider $b\in A(R)$ obtained in Step 2. By composing $b:\spec(R)\to A$ with our fixed $C$-rational point $v:\spec(C)\to \spec(R)$ we get a point $b_v\in A(C)$. If we consider $b_v$ as an element of $A(R)$, then $b$ and $b_v$ map the closed point of $\spec(R)$ to the same point of $A$. Therefore $b-b_v\in A(R)_*$ (see the notation and comments before Remark \ref{pointexample}), so $b-b_v$ factors through a morphism $\phi:\spec(R)\to \spec(\mathcal{O}_{A})$. By Step 2 we have
$$\mathcal{F}_A-\phi_A=\mathcal{F}_A-b+b_v\in A(\widehat{R}^p),$$
since $b_v\in A(C)$. We also have
$$\mathcal{F}_A-\phi_A=\mathcal{F}_A-\widehat{\phi}_A=(\mathcal{F}-\widehat{\phi})_A,$$
where the first equality is obvious and the second follows from Lemma \ref{newadd}. By Lemma \ref{points}(ii), we get $(\mathcal{F}-\widehat{\phi})(\widehat{\mathcal{O}}_{A})\subseteq \widehat{R}^p$.

After completing this step we actually have made the very first step to find a compatible sequence witnessing that $\mathcal{F}$ is an $A$-limit. In the  final step below we show that we have an inductive procedure at hand.
\\
{\bf Step 4} $\mathcal{F}$ is an $A$-limit.

Let us take $\phi_0:=\phi$ from Step 3. By Corollary \ref{facfrosch}, $\mathcal{F}-\widehat{\phi_0}:\widehat{V}\to \widehat{A}$ factors through $\mathcal{F}_1:\widehat{V}^{\fr}\to \widehat{A}$ i.e. there is a commutative diagram
\begin{equation*}
  \xymatrix{\widehat{V}\ar[rr]^{\mathcal{F}-\widehat{\phi_0}}\ar[d]^{\fr} & & \widehat{A} \\
\widehat{V}^{\fr}\ar[rru]_{\mathcal{F}_1}. & & }
\end{equation*}
\noindent%
By Lemma \ref{specialroot}, $\mathcal{F}_1$ is a special formal map. Applying Step 3 to $\mathcal{F}_1$ and using Corollary \ref{facfrosch} again, we get a morphism $\phi_1:\widehat{V}^{\fr^{-1}}\to \widehat{A}$ such that $\mathcal{F}_1-\widehat{\phi_1}$ factors through a morphism $\mathcal{F}_2:\widehat{V}^{\fr^{-2}}\to \widehat{A}$ i.e. we have a bigger commutative diagram as below
\begin{equation*}
  \xymatrix{
\widehat{V}\ar[rrrd]^{\ \ \ \mathcal{F}_1\circ \fr-\widehat{\phi_1}\circ \fr}\ar[d]^{\fr} & & &\\
\widehat{V}^{\fr}\ar[rrr]^{\mathcal{F}_1-\widehat{\phi_1}}\ar[d]^{\fr} & & &\widehat{A} \\
\widehat{V}^{\fr^2}\ar[rrru]_{\mathcal{F}_2}. & & &}
\end{equation*}
\noindent%
Hence we get
$$\mathcal{F}_2\circ \fr^2=\mathcal{F}_1\circ \fr-\widehat{\phi_1}\circ \fr=\mathcal{F}-\widehat{\phi}-\widehat{\phi_1}\circ \fr.$$
If we continue like this, we get a sequence $\phi_m:\widehat{V}^{\fr^{-m}}\to \widehat{A}$. We set $$g_m:=\phi_0+\phi_1\circ \fr+\ldots+\phi_m\circ\fr^m.$$
Similarly as in the proof of Theorem \ref{mainthm}, we conclude that $g_m$ is a compatible sequence converging to $\mathcal{F}$, hence $\mathcal{F}$ is an $A$-limit.
\begin{remark}\label{unipalimit}
\begin{enumerate}
\item It is possible that the map
$$H^1(K_1,A)\to H^1(K_2,A)$$
is always injective (or even an isomorphism) for a purely inseparable extension $K_1\subseteq K_2$, but we were not able to show it.

\item Clearly, the injectivity condition above is satisfied if the first cohomology group is trivial. In particular, any special map into an affine commutative group $A$ with split torus (e.g. into a unipotent group) is necessarily an $A$-limit map.
\end{enumerate}
\end{remark}
\noindent

\bibliographystyle{plain}
\bibliography{harvard}

\end{document}